\documentclass[a4paper,11pt]{amsart}
\usepackage{mathrsfs}
\usepackage{cases}

\usepackage{amsfonts}

\usepackage{graphicx}
\usepackage{amsmath}
\usepackage{amssymb}
\usepackage[all]{xypic}
\usepackage[all]{xy}
\begin{document}
\input xy
\xyoption{all}

\renewcommand{\mod}{\operatorname{mod}\nolimits}
\newcommand{\proj}{\operatorname{proj.}\nolimits}
\newcommand{\rad}{\operatorname{rad}\nolimits}
\newcommand{\Gproj}{\operatorname{Gproj}\nolimits}
\newcommand{\Ginj}{\operatorname{Ginj}\nolimits}
\newcommand{\Gd}{\operatorname{Gd}\nolimits}
\newcommand{\gldim}{\operatorname{gldim}\nolimits}
\newcommand{\ind}{\operatorname{inj.dim}\nolimits}
\newcommand{\Top}{\operatorname{top}\nolimits}
\newcommand{\ann}{\operatorname{Ann}\nolimits}
\newcommand{\id}{\operatorname{id}\nolimits}
\newcommand{\Mod}{\operatorname{Mod}\nolimits}
\newcommand{\End}{\operatorname{End}\nolimits}
\newcommand{\Ob}{\operatorname{Ob}\nolimits}
\newcommand{\Ht}{\operatorname{Ht}\nolimits}
\newcommand{\cone}{\operatorname{cone}\nolimits}
\newcommand{\rep}{\operatorname{rep}\nolimits}
\newcommand{\Ext}{\operatorname{Ext}\nolimits}
\newcommand{\Hom}{\operatorname{Hom}\nolimits}
\newcommand{\RHom}{\operatorname{RHom}\nolimits}
\renewcommand{\Im}{\operatorname{Im}\nolimits}
\newcommand{\Ker}{\operatorname{Ker}\nolimits}
\newcommand{\Coker}{\operatorname{Coker}\nolimits}
\renewcommand{\dim}{\operatorname{dim}\nolimits}
\newcommand{\Ab}{{\operatorname{Ab}\nolimits}}
\newcommand{\Coim}{{\operatorname{Coim}\nolimits}}
\newcommand{\pd}{\operatorname{proj.dim}\nolimits}
\newcommand{\sdim}{\operatorname{sdim}\nolimits}
\newcommand{\add}{\operatorname{add}\nolimits}
\newcommand{\pr}{\operatorname{pr}\nolimits}
\newcommand{\Tr}{\operatorname{Tr}\nolimits}
\newcommand{\Def}{\operatorname{Def}\nolimits}
\newcommand{\red}{\operatorname{red}\nolimits}
\newcommand{\soc}{\operatorname{soc}\nolimits}

\newcommand{\ca}{{\mathcal A}}
\newcommand{\cb}{{\mathcal B}}
\newcommand{\cc}{{\mathcal C}}
\newcommand{\cd}{{\mathcal D}}
\newcommand{\cg}{{\mathcal G}}
\newcommand{\cp}{{\mathcal P}}
\newcommand{\ce}{{\mathcal E}}
\newcommand{\cs}{{\mathcal S}}
\newcommand{\cm}{{\mathcal M}}
\newcommand{\cn}{{\mathcal N}}
\newcommand{\cx}{{\mathcal X}}
\newcommand{\ct}{{\mathcal T}}
\newcommand{\cu}{{\mathcal U}}
\newcommand{\co}{{\mathcal O}}
\newcommand{\cv}{{\mathcal V}}
\newcommand{\calr}{{\mathcal R}}
\newcommand{\ol}{\overline}
\newcommand{\ul}{\underline}
\newcommand{\st}{[1]}
\newcommand{\ow}{\widetilde}
\newcommand{\coh}{{\mathrm coh}}
\newcommand{\CM}{{\mathrm CM}}
\newcommand{\vect}{{\mathrm vect}}

\newcommand{\bp}{{\mathbf p}}
\newcommand{\bL}{{\mathbf L}}
\newcommand{\bS}{{\mathbf S}}

\newtheorem{theorem}{Theorem}[section]
\newtheorem{acknowledgement}[theorem]{Acknowledgement}
\newtheorem{algorithm}[theorem]{Algorithm}
\newtheorem{axiom}[theorem]{Axiom}
\newtheorem{case}[theorem]{Case}
\newtheorem{claim}[theorem]{Claim}
\newtheorem{conclusion}[theorem]{Conclusion}
\newtheorem{condition}[theorem]{Condition}
\newtheorem{conjecture}[theorem]{Conjecture}
\newtheorem{construction}[theorem]{Construction}
\newtheorem{corollary}[theorem]{Corollary}
\newtheorem{criterion}[theorem]{Criterion}
\newtheorem{definition}[theorem]{Definition}
\newtheorem{example}[theorem]{Example}
\newtheorem{exercise}[theorem]{Exercise}
\newtheorem{lemma}[theorem]{Lemma}
\newtheorem{notation}[theorem]{Notation}
\newtheorem{problem}[theorem]{Problem}
\newtheorem{proposition}[theorem]{Proposition}
\newtheorem{remark}[theorem]{Remark}
\newtheorem{solution}[theorem]{Solution}
\newtheorem{summary}[theorem]{Summary}
\newtheorem*{thm}{Theorem}
\newtheorem*{thma}{Theorem A}
\newtheorem*{thmb}{Theorem B}
\newtheorem*{thmc}{Theorem C}
\newtheorem*{thm1}{Main Theorem 1}
\newtheorem*{thm2}{Main Theorem 2}
\def \Z{{\Bbb Z}}
\def \X{{\Bbb X}}
\def \T{\Bbb T}
\def \K{\Bbb K}
\def \N{{\Bbb N}}
\def \D{{\Bbb D}}
\def \E{{\Bbb E}}
\def \A{{\Bbb A}}
\renewcommand{\P}{{\Bbb P}}
\title[Singularity categories of some $2$-CY-tilted algebras]{Singularity categories of some $2$-CY-tilted algebras}

\author[Lu]{Ming Lu}
\address{Department of Mathematics, Sichuan University, Chengdu 610064, P.R.China}
\email{luming@scu.edu.cn}


\begin{abstract}
We define a class of finite-dimensional Jacobian algebras, which are called (simple) polygon-tree algebras, as a generalization of cluster-tilted algebras of type $\D$. They are $2$-CY-tilted algebras. Using a suitable process of mutations of quivers with potentials (which are also BB-mutations) inducing derived equivalences, and one-pointed (co)extensions which preserve singularity equivalences, we find a connected selfinjective Nakayama algebra whose stable category is equivalent to the singularity category of a simple polygon-tree algebra. Furthermore, we also give a classification of algebras of this kind up to representation type.
\end{abstract}

\maketitle

\section{Introduction}

The Fomin-Zelevinsky mutation (FZ-mutation for short) of quivers plays an important role in the theory of cluster algebras initiated in \cite{FZ1}. Motivated by this theory via \cite{MRZ}, Buan, Marsh, Reiten, Reineke and Todorov introduced cluster categories to give a perfect categorical model for cluster algebras \cite{BMRRT}. Importantly, cluster-tilting objects in cluster categories are used to categorify clusters of the corresponding cluster algebras, and their mutations correspond to the FZ-mutation of quivers \cite{BMR2}. This is generalized to more general Hom-finite triangulated $2$-Calabi-Yau ($2$-CY for short) categories \cite{IY,BIRS1}.
On the other hand, in \cite{DWZ}, Derksen, Weyman and Zelevinsky studied quivers with potentials (QPs for short), that is, pairs consisting of a quiver and a special element
of its (complete) path algebra, and defined the mutations of such objects, thus, they provide
a new representation-theoretic interpretation for FZ-mutations of quivers.

Associated with cluster-tilting objects $T$ in cluster categories (resp. $2$-CY categories) $\cc$ are the endomorphism algebras $\End_\cc(T)$, called cluster-tilted algebras \cite{BMR1} (resp. $2$-CY-tilted algebras \cite{Rei}). Associated with QPs $(Q,W)$ are the Jacobian algebras (see e.g., \cite{DWZ}). In particular, cluster-tilted algebras are $2$-CY-tilted algebras, and a large class of $2$-CY-tilted algebras coming from triangulated $2$-CY categories associated with elements in Coxeter groups (\cite{BIRS1}) are Jacobian algebras, including cluster-tilted algebras \cite{BIRS2}. Furthermore, any finite-dimensional Jacobian algebras are $2$-CY-tilted algebras \cite{Am}.

The mutation of cluster-tilting objects induces an operation on the associated $2$-CY-tilted algebras and the mutation of QPs induces an operation on the associated Jacobian algebras. It is a conjecture that for $2$-CY-tilted algebras which are Jacobian algebras, the mutations of $2$-CY-tilted algebras coincides with those of Jacobian algebras. This is proved for a large class of $2$-CY-tilted algebras in \cite{BIRS2}, including cluster-tilted algebras, see Theorem \ref{theorem mutation cluster-tilting to mutation QP}.

To know when a single mutation of QP leads to derived equivalence of the corresponding Jacobian algebras, Ladkani \cite{La} defined the negative and positive mutations of algebras, which are endomorphism algebras of tilting complexes introduced and studied by Vit\'{o}ria in \cite{Vi}, Keller and Yang in \cite{KY}. In fact, the negative mutation of algebras is a generalization of BB-tilting modules \cite{BB}, which are themselves generalizations of the BGP reflection functors introduced in \cite{BGP}. For $2$-CY-tilted algebras, Ladkani proved that the mutation of cluster-tilting objects inducing derived equivalence coincides with the BB-mutation under some conditions, see Proposition \ref{proposition mutation to BB}. These kind of mutations are used to give derived equivalence classification
of cluster-tilted algebras of Dynkin type $\D$ \cite{BHL2} and $\E$ \cite{BHL1}.

Besides, $2$-CY-tilted algebras are Gorenstein algebra of dimension at most one, and their categories of Gorenstein projective modules (also called maximal Cohen-Macaulay modules) are stably $3$-CY \cite{KR1}. A fundamental result of Buchweitz \cite{Bu} and Happel \cite{Ha1} states that for a Gorenstein algebra $A$, the stable category of Gorenstein projective modules over $A$ is equivalent to its singularity category $D^b_{sg}(A)$, where the singularity category of an algebra is defined to be the Verdier quotient of the bounded derived category with respect to the thick subcategory formed by complexes isomorphic to bounded complexes of finitely generated projective modules \cite{Bu,Ha1,Or1}. The singularity categories of many algebras have been described clearly, see e.g., \cite{Chen1,Chen2,Ka,CL1}. In \cite{CGL}, we have settled the problem of singularity equivalence classification of the cluster-tilted algebras of type $\A$, $\D$ and $\E$.

In this paper, as a generalization of a type of cluster-tilted algebras of type $\D$, we define a class of algebras, which are called polygon-tree algebras, see Definition \ref{definition of polygon-tree algebras}. Roughly speaking, they are Jacobian algebras with their quivers constructed from several oriented cycles like trees, which are called polygon-tree quivers, and the potentials are primitive in the sense of \cite{DWZ}. Locally, the quiver of the polygon-tree algebra is called a floriated quiver, see Section 3. Both polygon-tree quivers and floriated quivers are cyclically oriented quivers in the sense of \cite{BGZ,BT}.

First, with the help of \cite{TV}, we prove that the polygon-tree algebras, including floriated algebras, are finite-dimensional Jacobian algebras, and then are $2$-CY-tilted algebras, see Proposition \ref{schurian of polygon-tree algebras}. Furthermore, we obtain that the polygon-tree algebras are schurian algebras, see Theorem \ref{lemma schurian algebra of simple polygon-tree algebras}. Second, we consider the singularity categories of the simple polygon-tree algebras, which form a subclass of polygon-tree algebras. Using a suitable process of mutations of QPs inducing derived equivalences (which are BB-mutations), and one-pointed (co)extensions which preserve singularity equivalences, we find a connected selfinjective Nakayama algebra whose stable category is equivalent to the singularity category of a simple polygon-tree algebra, see Theorem \ref{proposition to type Q(m,{i_1,...,i_r})}. In particular, simple polygon-tree algebras are CM-finite algebras. With the help of the description of the stable categories of representation finite selfinjective algebras of type $\A_n$ in \cite{Rie2}, we describe the singularity categories of the simple polygon-tree algebras clearly, see Corollary \ref{corollary singularity category}. Third, using the classification of the quivers of finite mutation type \cite{FeST}, and the representation type of Jacobian algebras in \cite{GLS}, we also give a classification of representation type for polygon-tree algebras, see Theorem \ref{theorem of representation type}.

\vspace{0.2cm} \noindent{\bf Acknowledgments.}
The author thanks Changjian Fu and Shengfei Geng for inspiring discussions, the author appreciates the referee
for careful reading and helpful insightful comments which improved this paper.

The author was supported by the National Natural Science Foundation of China(Grant No. 11401401).

\section{Preliminaries}
In this paper, $K$ is an algebraically closed field.
\subsection{Singularity categories and Gorenstein algebras}
Let $\Gamma$ be a finite-dimensional $K$-algebra. Let $\mod \Gamma$ be the category of finitely generated left $\Gamma$-modules. For an arbitrary $\Gamma$-module $_\Gamma X$, we denote by $\pd_\Gamma X$ (resp. $\ind_\Gamma X$) the projective dimension (resp. the injective dimension) of the module $_\Gamma X$. A $\Gamma$-module $G$ is \emph{Gorenstein projective}, if there is an exact sequence $$P^\bullet:\cdots \rightarrow P^{-1}\rightarrow P^0\xrightarrow{d^0}P^1\rightarrow \cdots$$ of projective $\Gamma$-modules, which stays exact under $\Hom_\Gamma(-,\Gamma)$, and such that $G\cong \Ker d^0$. We denote by $\Gproj(\Gamma)$ the subcategory of Gorenstein projective $\Gamma$-modules.

\begin{definition}[\cite{AR1,AR2,Ha1}]
A finite-dimensional algebra $\Gamma$ is called a Gorenstein algebra if $\Gamma$ satisfies $\ind \Gamma_\Gamma<\infty$ and $\ind\,_{\Gamma} \Gamma<\infty$.
\end{definition}

Observe that for a Gorenstein algebra $\Gamma$, we have $\ind _\Gamma\Gamma=\ind \Gamma_\Gamma$, \cite[Lemma 6.9]{Ha1}; the common value is denoted by $\Gd \Gamma$. If $\Gd \Gamma\leq d$, we say that $\Gamma$ is $d$-Gorenstein.

An algebra is of \emph{finite Cohen-Macaulay type}, or simply, \emph{CM-finite}, if there are only finitely many isomorphism classes of indecomposable finitely generated Gorenstein projective modules.

For an algebra $\Gamma$, the \emph{singularity category} of $\Gamma$ is defined to be the quotient category $D_{sg}^b(\Gamma):=D^b(\Gamma)/K^b(\proj \Gamma)$ \cite{Bu,Ha1,Or1}. Note that $D_{sg}^b(\Gamma)$ is zero if and only if $\gldim \Gamma<\infty$ \cite{Ha1}.

\begin{theorem}[\cite{Bu,Ha1}]
Let $\Gamma$ be a Gorenstein algebra. Then $\Gproj (\Gamma)$ is a Frobenius category with the projective modules as the projective-injective objects. The stable category $\underline{\Gproj}(\Gamma)$ is triangle equivalent to the singularity category $D^b_{sg}(\Gamma)$ of $\Gamma$.
\end{theorem}

\subsection{Quiver with potential and its mutations}
We follow \cite{DWZ} to introduce the quiver with potential and its mutation. A quiver $Q=(Q_v,Q_a,s,t)$ consists of a pair of finite sets $Q_v$ (vertices) and $Q_a$ (arrows) supplied with two maps $s:Q_a\rightarrow Q_v$ (source) and $t:Q_{a}\rightarrow Q_v$ (target). The complete path algebra $\widehat{KQ}$ is the completion of the path algebra $KQ$ with respect to the ideal generated by the arrows of $Q$. A \emph{potential} on $Q$ is an element of the closure $Pot(KQ)$ of the space generated by all non trivial cyclic paths of $Q$. Two potentials $W$ and $W'$ are called \emph{cyclically equivalent} if $W-W'\in\{KQ,KQ\}$, where $\{KQ,KQ\}$ denotes the closure of the vector space spanned by commutators. Let $\mathfrak{m}$ be the ideal of $\widehat{KQ}$ generated by all arrows of $Q$.
Let $W$ be a potential on $Q$ such that $W$ is in $\mathfrak{m}^2$ and no two cyclically equivalent cyclic paths appear in the decomposition of $W$. Then the pair $(Q,W)$ is called a \emph{quiver with potential} (QP for short).

Two QPs $(Q,W)$ and $(Q',W')$ are \emph{right-equivalent} if $Q$ and $Q'$ have the same set of vertices and there exists an algebra isomorphism $\varphi:\widehat{KQ}\rightarrow \widehat{KQ'}$ whose restriction on vertices is the identity map and $\varphi(W)$ and $W'$ are cyclically equivalent. Such an isomorphism $\varphi$ is called a \emph{right-equivalence}.

For an arrow $a$ of $Q$, the cyclic derivative $\partial_aW$ is defined by
$$\partial_a(a_1\cdots a_l)=\sum_{a_i=a}a_{i+1}\cdots a_la_1\cdots a_{i-1}$$
and extended linearly and continuously. The \emph{Jacobian algebra} of a QP $(Q,W)$, denoted by $J(Q,W)$, is the quotient of the complete path algebra $\widehat{KQ}$ by the \emph{Jacobian ideal} $J(W)$, where $J(W)$ is the closure of the ideal generated by $\partial_aW$, where $a$ runs over all arrows of $Q$.
It is clear that two right-equivalent QPs have isomorphic Jacobian algebras. A QP is called \emph{reduced} if $\partial_a W$ is contained in $\mathfrak{m}^2$ for all arrows $a$ of $Q$.

It is shown in \cite{DWZ} that for any QP $(Q,W)$, there exists a reduced QP $(Q_{\red},W_{\red})$ such that
$$J(Q,W)\simeq J(Q_{\red},W_{\red}),$$
which is uniquely determined up to right-equivalence. We call $(Q_{\red},W_{\red})$ the \emph{reduced part} of $(Q,W)$.

For every QP $(Q,W)$, we define its \emph{deformation space} $\Def(Q,W)$ by
$$\Def(Q,W)=\Tr(J(Q,W))/R,$$ where $\Tr(J(Q,W))=J(Q,W)/\{J(Q,W),J(Q,W)\}$ and $R=k^{Q_{v}}$. We call a QP $(Q,W)$ \emph{rigid} if $\Def(Q,W)=\{0\}$, i.e., if $\Tr(J(Q,W))=R$.

Let $(Q,W)$ be a QP. Let $i$ be a vertex of $Q$. Assume the following conditions:

(c1) the quiver $Q$ has no loops;

(c2) the quiver $Q$ does not have oriented $2$-cycles at $i$;

(c3) no cyclic path occurring in the expansion of $W$ starts and ends at $i$.

We define a new QP $\tilde{\mu}_i(Q,W)=(Q',W')$ as follows.
\begin{itemize}
\item[Step 1] For each arrow $\beta$ with target $i$ and each arrow $\alpha$ with source $i$, add a new arrow $[\alpha\beta]$ from the source of $\beta$ to the target of $\alpha$.
 \item[Step 2] Replace each arrow $\alpha$ with source or target $i$ with an arrow $\alpha^*$ in the opposite direction.
\item[Step 3]
The new potential $W'$ is the sum of two potentials $W_1'$ and $W_2'$. The potential $W_1'$ is obtained from $W$ by replacing each composition $\alpha\beta$ by $[\alpha\beta]$, where $\beta$ is an arrow with target $i$.
The potential $W_2'$ is given by $W_2'=\sum_{\alpha,\beta}[\alpha\beta]\beta^*\alpha^*$, where the sum ranges over all pairs of arrows $\alpha$ and $\beta$ such that $\beta$ ends at $i$ and $\alpha$ starts at $i$.
\end{itemize}
We define $\mu_i(Q,W)$ as the reduced part of $\tilde{\mu}_i(Q,W)$, and call $\mu_i$ \emph{the mutation at the vertex $i$}.
In this case, $\mu_i$ is also well-defined on the QP $\mu_i(Q,W)$, and $\mu_{i}(\mu_i(Q,W))$ is right-equivalent to $(Q,W)$ \cite{DWZ}.

\begin{definition}[\cite{DWZ}]
Let $k_1,\dots,k_l\in Q_{v}$ be a finite sequence of vertices such that $k_p\neq k_{p+1}$ for $p=1,\dots,l-1$. We say that a QP $(Q,W)$ is $(k_l,\dots,k_1)$-nondegenerate if all the quivers with potentials $$(Q,W),\mu_{k_1}(Q,W),\dots,\mu_{k_l}\cdots \mu_{k_1}(Q,W)$$ are $2$-acyclic. We say that $(Q,W)$ is nondegenerate if it is $(k_l,\dots,k_1)$-nondegenerate for every sequence of vertices as above.
\end{definition}

\begin{proposition}[\cite{DWZ}]\label{proposition rigid}
(a) If a reduced QP $(Q,W)$ is $2$-acyclic and rigid, then for any vertex $k$, $\mu_k(Q,W)$ is also rigid.

(b) Each rigid reduced QP $(Q,W)$ is $2$-acyclic.

(c) Each rigid QP is nondegenerate.
\end{proposition}

\begin{definition}[\cite{BGZ,BT}]
A walk of length $p$ in a quiver $Q$ is a $(2p+1)$-tuple
$$w=(x_p,\alpha_p,x_{p-1},\alpha_{p-1},\dots,x_1,\alpha_1,x_0)$$
such that for all $i$ we have $x_i\in Q_0$, $\alpha_i\in Q_1$ and $\{s(\alpha_i),t(\alpha_i)\}=\{x_p,x_{p-1}\}$. The walk $w$ is oriented if either $s(\alpha_i)=x_{i-1}$ and $t(\alpha_i)=x_i$ for all $i$ or $s(\alpha_i)=x_i$ and $t(\alpha_i)=x_{i-1}$ for all $i$. Furthermore, $w$ is called a cycle if $x_0=x_p$. An oriented walk is also called a path. If $w$ is oriented and $t(\alpha_i)=s(\alpha_{i+1})$ for any $i$, we omit the vertices and abbreviate $w$ by $\alpha_p\dots \alpha_1$.

A cycle $c=(x_p,\alpha_p,\dots,x_1,\alpha_1,x_p)$ is called non-self-intersecting if its vertices $x_1,\dots,x_p$ are pairwise distinct. A non-self-intersecting cycle of length $2$ is called a $2$-cycle. If $c$ is a non-self-intersecting cycle, then any arrow $\beta\in Q \setminus \{\alpha_1,\dots,\alpha_p\}$ with $\{s(\beta),e(\beta)\}\subseteq \{x_1,\dots,x_p\}$ is called a chord of $c$. A cycle $c$ is called chordless if it is non-self-intersecting and there is no chord of $c$.

A quiver $Q$ without loops or oriented $2$-cycles is called cyclically oriented if each chordless cycle is oriented.

A path $\gamma$, which is anti-parallel to an arrow $\eta$ in a quiver $Q$, is called a shortest path if the full subquiver generated by the induced oriented cycle $\eta\gamma$ is chordless.
\end{definition}

\begin{definition}[\cite{DWZ}]
Let $Q$ be a quiver. A primitive potential $S$ is a linear combination of all oriented chordless cycles in $Q$ with non-zero scalars.
\end{definition}
\subsection{Mutation of cluster-tilting objects}
Let $\cc$ be a Hom-finite triangulated $k$-category. We denote by $[1]$ the shift functor in $\cc$. Then $\cc$ is said to be $n$-Calabi-Yau ($n$-CY for simplicity) if there is a functorial isomorphism
$$ D\Hom_\cc(A,B)\simeq \Ext^n_\cc(B,A)$$
for $A,B$ in $\cc$ and $D=\Hom_k(-,k)$.

Let $\cc$ be a $2$-CY triangulated category. An object in $\cc$ is a \emph{cluster-tilting object} if
$$\add T=\{X\in\cc|\Ext^1_\cc(T,X)=0\}$$
(see \cite{BMRRT,KR1}). In this case the algebra $\End_\cc(T)$ is called a \emph{$2$-CY-tilted algebra} \cite{Rei}. \cite[Corollary 3.7]{Am} shows that each finite-dimensional Jacobian algebra $J(Q,W)$ is $2$-CY-tilted.

Let $T=T_1\oplus T_2\oplus\cdots \oplus T_n$ be a cluster-tilting object, where $T_i$ are nonisomorphic indecomposable objects. Then for any $1\leq k\leq n$ there exists a unique $T_k^*$ non-isomorphic to $T_k$ such that $\mu_k(T)=(T/T_k)\oplus T_k^*$ is a cluster-tilting object. Moreover, there are so-called exchange triangles
$$T_k^*\xrightarrow{g}U_k\xrightarrow{f} T_k\rightarrow T_k^*[1]\mbox{ and }T_k\xrightarrow{g'}U'_k\xrightarrow{f'}T_k^*\rightarrow T_k[1],$$
where $f$ and $f'$ are minimal right $\add(T/T_k)$-approximations, and $g$ and $g'$ are minimal left $\add(T/T_k)$-approximations \cite{BMRRT,IY}.

Recall that a finite-dimensional algebra $\Lambda$ satisfies the \emph{vanishing condition} at $k$ if $$\Hom_\Lambda(\Ext^1_\Lambda(D\Lambda,S_k),S_k)=0$$ holds for the simple $\Lambda$-module $S_k$.

\begin{theorem}[\cite{BIRS2}]\label{theorem mutation cluster-tilting to mutation QP}
Let $\cc$ be a $2$-CY triangulated category with a basic cluster-tilting object $T$. If $\End_\cc(T)\simeq J(Q,W)$ for a QP $(Q,W)$, no oriented $2$-cycles start in the vertex $k$ of $Q$ and the vanishing condition is satisfied at $k$, then $\End_\cc(\mu_k(T))\simeq J(\mu_k(Q,W))$.
\end{theorem}

Cluster categories are by definition the orbit categories $\cc_Q=D^b(KQ)/\tau^{-1}[1]$, where $Q$ is a finite connected acyclic quiver, and $\tau$ is the AR-translation in $D^b(KQ)$ \cite{BMRRT}. These orbit categories are triangulated categories \cite{Ke}, and are Hom-finite $2$-CY \cite{BMRRT}.
The cluster-tilted algebras are by definition the $2$-CY-tilted algebras coming from cluster categories \cite{BMR1}.

\subsection{Mutation of algebras}

We recall the notion of mutations of algebras from \cite{La}. Let $A=KQ/I$ be an algebra given as a quiver with relations. For any vertex $i$ of $Q$, there is a trivial path $e_i$ of length $0$; the corresponding indecomposable projective module $P_i=A e_i$ is spanned by the images of the paths starting at $i$. Thus an arrow $i\xrightarrow{\alpha}j$ gives rise to a map $P_j\rightarrow P_i$ given by right multiplication with $\alpha$. Furthermore, for any vertex $i$ of $Q$, there is an indecomposable injective module $I_i$, and a simple module $S_i$.

Let $k$ be a vertex of $Q$ without loops. Consider the following two complexes of projective $A$-modules
$$T_k^{-}(A)=(P_k\xrightarrow{f} \bigoplus_{j\rightarrow k}P_j)\oplus(\bigoplus_{i\neq k}P_i), \mbox{ } T_k^+(A)=(\bigoplus_{k\rightarrow j} P_j\xrightarrow{g} P_k )\oplus (\bigoplus_{i\neq k}P_i) $$
where the map $f$ is induced by all the maps $P_k\rightarrow P_j$ corresponding to the arrows $j\rightarrow k$ ending at $k$, the map $g$ is induced by the maps $P_j\rightarrow P_k$ corresponding to the arrows $k\rightarrow j$ starting at $k$, the term $P_k$ lies in degree $-1$ in $T^{-}_k(A)$ and in degree $1$ in
$T_k^+(A)$, and all other terms are in degree $0$.

\begin{definition}[\cite{BHL2}]
Let $A$ be an algebra given as a quiver with relations and $k$
a vertex without loops.

(a) We say that the negative mutation of $A$ at $k$ is defined if $T_k^-(A)$ is a tilting complex over $A$. In this case, we call the algebra
$\mu_k^-(A)=\End_{D^b(A)}(T^-_k(A))$ the negative mutation of $A$ at the vertex $k$.

(b) We say that the positive mutation of $A$ at $k$ is defined if $T_k^+(A)$ is a tilting complex over $A$. In this case, we call the algebra
$\mu_k^+(A)=\End_{D^b(A)}(T^+_k(A))$ the positive mutation of $A$ at the vertex $k$.
\end{definition}

Let $Q^{op}$ be the \emph{opposite quiver} of $Q$. Namely, it has the same set of vertices as $Q$, with the (opposite) arrow $\alpha^*:j\rightarrow i$ for any arrow $\alpha:i\rightarrow j$ of $Q$. If $A=KQ/I$, then the \emph{opposite algebra} $A^{op}$ can be written as $A^{op}=KQ^{op}/I^{op}$ where $I^{op}$ is generated by the paths opposite to those generating $I$. The indecomposable projective $A^{op}$-module corresponding to a vertex of $Q^{op}$ is then $P_i^t=\Hom_A(P_i,A)$. In particular, $(-)^t=\Hom(-,A):\mod A\rightarrow A^{op}$ induces a duality between the category $\proj A$ and the category $\proj A^{op}$.

\begin{lemma}\label{lemma negative positive mutation in opposite algebra}
Let $A$ be an algebra given as a quiver with relations and $k$
a vertex without loops. Then $\mu_k^-(A^{op})\simeq (\mu_k^+(A))^{op}$ and $\mu_k^+(A^{op})\simeq (\mu_k^-(A))^{op}$.
\end{lemma}
\begin{proof}
We only need prove $\mu_k^-(A^{op})\simeq (\mu_k^+(A))^{op}$. It is easy to see that $(T_k^+(A))^{t}$ is $T_k^-(A^{op})$, where $(-)^t$ is defined for complexes of projective modules naturally. Then the result follows immediately since $(-)^t$ induces a duality between $\proj A$ and $\proj A^{op}$.
\end{proof}

\begin{proposition}[\cite{La}]\label{proposition old good mutation}
Let $k$ be a vertex without loops.

(a) The negative mutation $\mu_k^-(A)$ is defined if and only if for any non-zero linear combination $\sum a_rp_r$ of paths $p_r$ starting at $k$ and ending at some vertex $i\neq k$, there exists at least one arrow $\alpha$ ending at $k$ such that the composition $\sum a_r p_r\alpha$ is not zero.

(b) The positive mutation $\mu_k^+(A)$ is defined if and only if for any non-zero linear combination $\sum a_rp_r$ of paths $p_r$ starting at some vertex $i\neq k$ and ending at $k$, there exists at least one arrow $\beta$ starting at $k$ such that the combination $\sum a_r\beta p_r$ is not zero.

(c) $T_k^-(A)$ is a tilting complex for $A$ if and only if $T_k^+(A^{op})$ is a tilting complex for $A^{op}$.

\end{proposition}

Let $\tau$ denote the \emph{Auslander-Reiten translation} in $\mod A$.

\begin{definition}[\cite{La}]
We say that the BB-tilting module is defined at the vertex $k$ if the $A$-module
$$T_k^{BB}=\tau^{-1}S_k\oplus (\bigoplus_{i\neq k}P_i)$$
is a tilting module of projective dimension at most $1$. In this case, $T_k^{BB}$ is called the BB-tilting module associated with $k$.
\end{definition}

\begin{lemma}[\cite{La}]\label{lemma BB-tilting and negative tilting}
Assume that $S_k$ is not a submodule of the radical of $P_k$. If $T_k^-$ is a tilting complex, then the BB-tilting module is defined at $k$ and $T_k^-\simeq T_k^{BB}$ in $D^b(A)$.
\end{lemma}

Note that the above lemma holds when $A$ is \emph{schurian}. Recall that an algebra $A=KQ/I$ is schurian if $\dim_k\Hom_A(P_i,P_j)\leq1$ for any two vertices $i,j$ of $Q$, or in other words, the entries of its Cartan matrix
are only $0$ or $1$.

\begin{lemma}[\cite{La}]\label{lemma mutation of algebra}
Let $k$ be a vertex of $Q$ without loops.

(a) If $\mu_k^-(A)$ is defined, then $\mu_k^+(\mu_k^-(A))$ is defined and isomorphic to $A$.

(b) If $\mu_k^+(A)$ is defined, then $\mu_k^-(\mu_k^+(A))$ is defined and isomorphic to $A$.

(c) If $\mu_k^{BB}(A)$ is defined, then $\mu_k^{BB}((\mu_k^{BB}(A))^{op})$ is defined and isomorphic to $A^{op}$.
\end{lemma}

\begin{proposition}[\cite{La}]\label{proposition mutation to BB}
Let $U$ be a cluster-tilting object in a $2$-CY triangulated category $\cc$ and let $\Lambda=\End_{\cc}(U)$ and $\Lambda'=\End_\cc(\mu_k(U))$ be two neighboring $2$-CY-tilted algebras. Then $\Lambda'\simeq \mu_k^{BB}(\Lambda)$ if and only if the BB-tilting modules $T_k^{BB}(\Lambda)$ and $T_k^{BB}({\Lambda'}^{op})$ are defined. In particular, in this case $\Lambda$ and $\Lambda'$ are derived equivalent.
\end{proposition}

\begin{lemma}\label{lemma extension vanishes}
Let $Q$ be a finite quiver without loops, $k$ a vertex of $Q$. Assume that the quotient algebra $A=KQ/I$ is finite-dimensional and there is at most one arrow $\alpha$ ending at $k$. Then $\Ext^1_A(I_k,S_k)=0$.
\end{lemma}
\begin{proof}
Since $A$ is finite-dimensional, we get that $D\Ext_A^1(I_k,S_k)\simeq \underline{\Hom}_A(\tau^{-1}S_k,I_k)$.

Case (1) there does not exist any arrow ending at $k$. Then $S_k\simeq I_k$, and so
$$\Ext^1_A(I_k,S_k)=\Ext^1_A(I_k,I_k)=0.$$

Case (2) there exists only one arrow $\alpha:j\rightarrow k$ ending at $k$. Note that $j\neq k$ since $Q$ has no loops. Then there is an exact sequence
$$P_k\xrightarrow{f} P_j\rightarrow \tau^{-1}S_k\rightarrow 0$$
where $f$ is induced by $\alpha$. So $\tau^{-1}S_k\simeq S_j$.
It is easy to see that $\Hom_A(S_j,I_k)=0$ since $\soc(I_k)=S_k\ncong S_j$.
So $D\Ext_A^1(I_k,S_k)\simeq \underline{\Hom}_A(S_j,I_k)=0$.
\end{proof}

\begin{lemma}\label{lemma vanishing conditions for schurian jacobian algebras}
Let $\cc$ be a $2$-CY triangulated category with a basic cluster-tilting object $T$. Assume that $\End_\cc(T)\simeq J(Q,W)$ for a QP $(Q,W)$ where $Q$ has no loops, no oriented $2$-cycles start in the vertex $k$ of $Q$, and there is at most one arrow $\alpha$ ending at $k$. If $\End_\cc(T)$ is finite-dimensional, then $\End_\cc(\mu_k(T))\simeq J(\mu_k(Q,W))$.
\end{lemma}
\begin{proof}
Set $A=\End_\cc(T)\simeq J(Q,W)$.
Since $A$ is finite-dimensional, we get that $$D\Hom_A(\Ext^1_A(DA,S_k),S_k)\cong D S_k\otimes _A \Ext^1_A(DA,S_k).$$
Note that
$D(S_k)$ is the right simple module corresponding to vertex $k$.
Let $e=e_k$ be the trivial path corresponding to $k$. For any $\alpha\otimes \xi\in D S_k\otimes _A \Ext^1_A(DA,S_k)$, where $\xi\in \Ext^1_A(DA,S_k)$ is represented by a short exact sequence $0\rightarrow S_k\xrightarrow{f} M\xrightarrow{g} DA\rightarrow0$, we get that $\alpha\otimes \xi=\alpha e\otimes \xi=\alpha\otimes e\xi$. The bimodule structure of $\Ext^1_A(DA,S_k)$ is induced by the bimodule structure of $DA$, and $e$ acts on $DA$ on the right inducing a morphism of left modules $\beta:DA\rightarrow DA$. Noting that $DA=\oplus_{i\in Q_v}DAe_i=DA e\oplus (\bigoplus_{i\neq k}DA e_i)$, we get that $\beta=\left( \begin{array}{ccc} 1&0\\0&0 \end{array} \right)$, where $1$ is the identity morphism for $DAe$.
Then $e\xi$ is constructed by a pull-back as the following diagram shows.
\[\xymatrix{ e\xi:&S_k\ar@{.>}[r]^{f'} \ar@{=}[d]  & M'\ar@{.>}[r]^{g'}  \ar@{.>}[d]  & DA \ar[d]^{\beta} \\
\xi:& S_k\ar[r]^{f}   & M\ar[r]^{g}  & DA.}\]
Since Lemma \ref{lemma extension vanishes} yields that $\Ext^1_A(I_k,S_k)=0$, we get that $e\xi=0$, and then $\alpha\otimes \xi=0$, which implies $D S_k\otimes _A \Ext^1_A(DA,S_k)=0$.
So $A$ satisfies the vanishing condition at $k$. By Theorem \ref{theorem mutation cluster-tilting to mutation QP}, we get that $\End_\cc(\mu_k(T))\simeq J(\mu_k(Q,W))$.
\end{proof}

\begin{proposition}\label{proposition mutations for schurian algebras}
Let $A$ be the Jacobian algebra $J(Q,W)$ of a QP $(Q,W)$, $k$ a vertex of $Q$ and $B=J(\mu_k(Q,W))$. Assume that $A$ is finite-dimensional, both $(Q,W)$ and $\mu_k(Q,W)$ have no loops or oriented $2$-cycles, and $Q$ has at most one arrow ending at $k$. If $A$ and $B$ have the property that the simple modules $_AS_k$ and $_BS_k$ are not submodules of the radicals of $_AP_k$ and $_BP_k$ respectively, then

(a) $B=J(\mu_k(Q,W))\simeq \mu_k^-(A)$ if and only if the two algebra mutations $\mu_k^-(A)$ and $\mu_k^+(B)$ are defined.

(b) $B=J(\mu_k(Q,W))\simeq \mu_k^+(A)$ if and only if the two algebra mutations $\mu_k^+(A)$ and $\mu_k^-(B)$ are defined.
\end{proposition}

\begin{proof}
Since $A$ is finite-dimensional and $(Q,W)$ has no loops or oriented $2$-cycles, \cite[Proposition 6.4]{DWZ} yields that $B$ is finite-dimensional. \cite[Corollary 3.7]{Am} shows that $A$ is a $2$-CY-tilted algebra. Denote by $\cc$ the $2$-CY category, and $T$ the cluster-tilting object in $\cc$ such that $A\simeq\End_\cc(T)$.
Since $Q$ has only one arrow ending at $k$, Lemma \ref{lemma vanishing conditions for schurian jacobian algebras} shows that $B=J(\mu_k(Q,W))\simeq \End_\cc(\mu_k(T))$.

(a) If $B=J(\mu_k(Q,W))\simeq \mu_k^-(A)$, then the proof of \cite[Theorem 4.2]{La} implies that $T_k^-(A)$ is a tilting complex which is isomorphic to $T_k^{BB}(A)$. So $\mu_k^-(A)$ is defined, and $J(\mu_k(Q,W))\simeq \mu_k^{BB}(A)$. Then
Proposition \ref{proposition mutation to BB} shows that $T_k^{BB}(B^{op})$ is a tilting complex, which yields that $T_k^{-}(B^{op})$ is tilting by \cite[Proposition 2.8]{La}. So $T_k^{+}(B)$ is a tilting complex by Proposition \ref{proposition old good mutation} (c).

Conversely, if the two algebra mutations $\mu_k^-(A)$ and $\mu_k^+(B)$ are defined, then $\mu_k^{BB}(A)$ and $\mu_k^{BB}(B^{op})$ are defined by Lemma \ref{lemma BB-tilting and negative tilting} since $A$ and $B$ have the property that the simple modules $_AS_k$ and $_BS_k$ are not submodules of the radicals of $_AP_k$ and $_BP_k$ respectively. The desired result follows from Proposition \ref{proposition mutation to BB} easily.

(b) We consider the opposite algebras $A^{op}$ and $B^{op}$. It is easy to see that $B^{op}=J(\mu_k(Q^{op},W^{op}))$, where $W^{op}$ is defined naturally. If $B=J(\mu_k(Q,W))\simeq \mu_k^+(A)$, then $B^{op}=J(\mu_k(Q^{op},W^{op}))\simeq (\mu_k^+(A))^{op}= \mu_k^-(A^{op})$, so $\mu_k^-(A^{op})$ and $\mu_k^+(B^{op})$ are defined by (a). Then we get that $\mu_k^+(A)$ and $\mu_k^-(B)$ are defined by Proposition \ref{proposition old good mutation} (c).

Conversely, let $(Q',W')=\mu_k(Q,W)$. Then $B=J(Q',W')$ and $A=J(\mu_k(Q',W'))$. If the two algebra mutations $\mu_k^+(A)$ and $\mu_k^-(B)$ are defined,
by (a), we get that $A=J(\mu_k(Q',W'))\simeq \mu_k^{-}(B)$. Lemma \ref{lemma mutation of algebra} (a) shows that $\mu_k^+(\mu_k^-(B))$ is defined and isomorphic to $B$ since $\mu_k^-(B)$ is defined. So $B\simeq \mu_k^+(\mu_k^-(B))\simeq \mu_k^+(A)$.
\end{proof}

\section{Polygon-tree algebras}
As the introduction says, motivated by the structure of the cluster-tilted algebras of Dynkin type $\D$, we define a class of algebras, which are called \emph{floriated algebras}.

In this section, let $Q_0,Q_1,\dots,Q_n$ be oriented cycles such that each $Q_i$ has $m_i$ ($m_i\geq 3$) vertices for each $Q_i$. The vertices of $Q_i$ are denoted by $v_1^i,v_2^i,\dots,v_{m_i}^i$, such that there is an arrow from $v_j^i$ to $v_{j+1}^i$ for $1\leq j\leq m_i$, where we set $m_i+1=1$.

We also assume that $m_0\geq n$. Let $i_1,\dots,i_n$ be distinct vertices of $Q_0$ satisfying $i_j<i_k$ for $j<k$. Identifying $v_{i_j}^0$ with $v_{1}^j$, $v_{i_{j}+1}^0$ with $v_2^j$, and also the arrow $v_{i_j}^0\rightarrow v_{i_{j}+1}^0$ with the arrow $v_1^j\rightarrow v_2^j$ for $1\leq j\leq n$, we get a quiver $Q$ as the following diagram shows.

\setlength{\unitlength}{0.9mm}
\begin{center}
\begin{picture}(20,70)(0,-12)
\put(0,0){\circle*{1.5}}
\put(-10,10){\circle*{1.5}}
\put(-10,10){\vector(1,-1){9.5}}
\put(-10,25){\circle*{1.5}}
\put(-10,25){\vector(0,-1){14.5}}
\put(0,35){\circle*{1.5}}
\put(0,35){\vector(-1,-1){9.5}}
\put(15,35){\circle*{1.5}}
\put(15,35){\vector(-1,0){14.5}}
\put(25,25){\circle*{1.5}}
\put(25,25){\vector(-1,1){9.5}}
\put(25,10){\circle*{1.5}}
\put(25,10){\vector(0,1){14.5}}
\put(15,0){\circle*{1.5}}
\put(15,0){\vector(1,1){9.5}}

\multiput(0.5,0)(1,0){14}{\circle*{0.3}}
\put(25,25){\vector(1,0){14.5}}
\put(40,25){\circle*{1.5}}
\put(40,10){\circle*{1.5}}
\put(40,10){\vector(-1,0){14.5}}
\multiput(40,11)(0,1){14}{\circle*{0.3}}

\put(0,35){\vector(0,1){14.5}}
\put(0,50){\circle*{1.5}}
\put(15,50){\circle*{1.5}}
\put(15,50){\vector(0,-1){14.5}}
\multiput(0.5,50)(1,0){14}{\circle*{0.3}}

\put(-10,45){\circle*{1.5}}
\put(-10,45){\vector(1,-1){9.5}}

\put(-20,35){\circle*{1.5}}
\put(-10,25){\vector(-1,1){9.5}}
\multiput(-20,35)(1,1){10}{\circle*{0.3}}

\put(13,-4){\tiny$v_{m_0}^0$}
\put(24,7){\tiny$v_{1}^0$}
\put(20,11){\tiny$v_{i_1}^0$}
\put(25,26){\tiny$v_{2}^0$}
\put(16.5,24){\tiny$v_{i_1+1}^0$}

\put(40,26){\tiny$v_{3}^1$}
\put(40,8){\tiny$v_{m_1}^1$}

\put(16,35){\tiny$v_3^0$}
\put(13,31){\tiny$v_{i_2}^0$}

\put(0,37){\tiny$v_4^0$}
\put(-1,32){\tiny$v_{i_2+1}^0$}
\put(0,29){\tiny$\|$}
\put(-1,26){\tiny$v_{i_3}^0$}

\put(-3,52){\tiny$v_3^2$}
\put(15,51){\tiny$v_{m_2}^2$}

\put(-9,23){\tiny$v_5^0$}
\put(-9,10){\tiny$v_6^0$}
\put(-1,-4){\tiny$v_7^0$}
\put(-25,33){\tiny$v_3^3$}
\put(-12,47){\tiny$v_{m_3}^3$}

\put(8,15){$Q_0$}
\put(30,15){$Q_1$}
\put(7,40){$Q_2$}
\put(-12,35){$Q_3$}

\put(-15,-10){Figure 1. Floriated quiver. }
\end{picture}

\end{center}
$(Q_0,\{i_1,\dots,i_n\})$ (or just $Q_0$) is called the \emph{central cycle}, $Q_i$ are called the \emph{petals} for $1\leq i\leq n$ and $v_{i_1}^0,\dots,v_{i_n}^0$ are called \emph{petal vertices}, $Q$ is called the \emph{floriated quiver of $(Q_0,\{i_1,\dots,i_n\})$ by $Q_1,\dots,Q_n$} (or simply \emph{floriated quiver}).

Let $W=\sum_{j=0}^nw_j$, where $w_j$ is the chordless oriented cycle $v_1^j\rightarrow v_2^j\rightarrow \dots\rightarrow v_{m_j}^j\rightarrow v_{1}^j$ along $Q_j$ for $j=0,1,\dots, n$. It is easy to see that $W$ is a primitive potential, and $(Q,W)$ is a QP. We call the Jacobian algebra $J(Q,W)$ to be the \emph{floriated algebra of $(Q_0,\{i_1,\dots,i_n\})$ by $Q_1,\dots,Q_n$} (or simply \emph{floriated algebra}).

\begin{remark}
From \cite{Va}, we get that if $Q_1,\dots,Q_n$ satisfy $m_1=\cdots=m_n=3$, then the floriated algebra $A=J(Q,W)$ is a cluster-tilted algebra of type $\D$.
\end{remark}

For a quiver $Q$, we denote by $\cs(Q)$ the set of all the chordless cycles in $Q$¡£

\begin{proposition}\label{proposition floriated algbra to one cycle}
Let $A=J(Q,W)$ be a floriated algebra of $(Q_0,\{i_1,\dots,i_n\})$ by $Q_1, \dots,Q_n$. Then its associated QP $(Q,W)$ is mutation equivalent to a QP $(Q',W')$, where $Q'$ has only one oriented cycle $c$, and $W'=c$. Furthermore, $(Q,W)$ is rigid and then nondegenerate.
\end{proposition}
\begin{proof}
If $n=0$, then we have nothing to prove.

For $n>0$, we assume that the quiver $Q$ is as Figure 1 shows. Note that $m_i\geq3$ for $1\leq i\leq n$. Without losing generality, we assume that $m_1=\max\{m_i|1\leq i\leq n\}$. If $m_1=3$, then $m_1=m_2=\cdots=m_n=3$, and so $A=J(Q,W)$ is a cluster-tilted algebra of type $\D$.  From \cite{GP,Va}, we know that the mutation class of a cluster-tilted algebra of type $\D$ contains an oriented cycle with the potential given by that cycle, which yields our desired result immediately.
So we assume that $m_1>3$. We mutate the QP at the vertex $v_{m_1}^1$ and denote the resulting QP by $\mu_{v_{m_1}^1}(Q,W)=(Q^1,W^1)$. Then the quiver $Q^1$ is as Figure 2.1 shows and $W^1=\sum_{w\in\cs(Q^1)}w$.

\setlength{\unitlength}{0.8mm}
\begin{center}

\begin{picture}(180,55)
\put(30,0){\begin{picture}(80,50)
\put(0,0){\circle*{1.5}}
\put(-10,10){\circle*{1.5}}
\put(-10,10){\vector(1,-1){9.5}}
\put(-10,25){\circle*{1.5}}
\put(-10,25){\vector(0,-1){14.5}}
\put(0,35){\circle*{1.5}}
\put(0,35){\vector(-1,-1){9.5}}
\put(15,35){\circle*{1.5}}
\put(15,35){\vector(-1,0){14.5}}
\put(25,25){\circle*{1.5}}
\put(25,25){\vector(-1,1){9.5}}
\put(25,10){\circle*{1.5}}
\put(25,10){\vector(0,1){14.5}}
\put(15,0){\circle*{1.5}}
\put(15,0){\vector(1,1){9.5}}

\multiput(0.5,0)(1,0){14}{\circle*{0.3}}

\put(40,25){\circle*{1.5}}
\put(40,10){\circle*{1.5}}
\put(40,10){\vector(-1,0){14.5}}
\put(32.5,5){\vector(3,2){7}}
\put(32.5,5){\circle*{1.5}}
\put(25,10){\vector(3,-2){7}}
\put(25,25){\vector(3,2){7}}
\put(32.5,30){\circle*{1.5}}
\put(40,25){\vector(0,-1){14.5}}
\multiput(32.5,30)(0.9,-0.6){9}{\circle*{0.3}}
\put(41,8){\tiny$v_{m_1-1}^1$}

\put(0,35){\vector(0,1){14.5}}
\put(0,50){\circle*{1.5}}
\put(15,50){\circle*{1.5}}
\put(15,50){\vector(0,-1){14.5}}
\multiput(0.5,50)(1,0){14}{\circle*{0.3}}

\put(-10,45){\circle*{1.5}}
\put(-10,45){\vector(1,-1){9.5}}

\put(-20,35){\circle*{1.5}}
\put(-10,25){\vector(-1,1){9.5}}
\multiput(-20,35)(1,1){10}{\circle*{0.3}}

\put(8,15){$Q_0$}

\put(7,40){$Q_2$}
\put(-12,35){$Q_3$}

\put(10,3){\tiny$v_{m_0}^0$}
\put(21,11){\tiny$v_{1}^0$}

\put(21,22){\tiny$v_{2}^0$}

\put(30,32){\tiny$v_3^1$}

\put(40,26){\tiny$v_{m_1-2}^1$}
\put(30,1){\tiny$v_{m_1}^1$}

\put(12,31){\tiny$v_3^0$}

\put(0,31){\tiny$v_4^0$}

\put(-3,51){\tiny$v_3^2$}
\put(15,51){\tiny$v_{m_2}^2$}

\put(-10,23){\tiny$v_5^0$}
\put(-9,10){\tiny$v_6^0$}
\put(0,1){\tiny$v_7^0$}
\put(-24,33){\tiny$v_3^3$}
\put(-12,47){\tiny$v_{m_3}^3$}

\put(-15,-10){Figure 2.1. $(Q^1,W^1)$. }
\end{picture}}

\put(120,0){\begin{picture}(80,80)
\put(0,0){\circle*{1.5}}
\put(-10,10){\circle*{1.5}}
\put(-10,10){\vector(1,-1){9.5}}
\put(-10,25){\circle*{1.5}}
\put(-10,25){\vector(0,-1){14.5}}
\put(0,35){\circle*{1.5}}
\put(0,35){\vector(-1,-1){9.5}}
\put(15,35){\circle*{1.5}}
\put(15,35){\vector(-1,0){14.5}}
\put(25,25){\circle*{1.5}}
\put(25,25){\vector(-1,1){9.5}}
\put(25,10){\circle*{1.5}}
\put(25,10){\vector(0,1){14.5}}
\put(15,0){\circle*{1.5}}
\put(15,0){\vector(1,1){9.5}}

\multiput(0.5,0)(1,0){14}{\circle*{0.3}}

\put(40,25){\circle*{1.5}}
\put(40,10){\circle*{1.5}}
\put(25,10){\vector(1,0){14.5}}
\put(40,10){\vector(-3,-2){7}}
\put(32.5,5){\circle*{1.5}}
\put(40,25){\vector(-1,-1){14.5}}
\put(25,25){\vector(3,2){7}}
\put(32.5,30){\circle*{1.5}}
\put(40,10){\vector(0,1){14.5}}
\multiput(32.5,30)(0.9,-0.6){9}{\circle*{0.3}}
\put(41,26){\tiny$v_{m_1-2}^1$}

\put(0,35){\vector(0,1){14.5}}
\put(0,50){\circle*{1.5}}
\put(15,50){\circle*{1.5}}
\put(15,50){\vector(0,-1){14.5}}
\multiput(0.5,50)(1,0){14}{\circle*{0.3}}

\put(-10,45){\circle*{1.5}}
\put(-10,45){\vector(1,-1){9.5}}

\put(-20,35){\circle*{1.5}}
\put(-10,25){\vector(-1,1){9.5}}
\multiput(-20,35)(1,1){10}{\circle*{0.3}}

\put(8,15){$Q_0$}

\put(7,40){$Q_2$}
\put(-12,35){$Q_3$}

\put(10,3){\tiny$v_{m_0}^0$}
\put(21,11){\tiny$v_{1}^0$}

\put(21,22){\tiny$v_{2}^0$}

\put(30,32){\tiny$v_3^1$}

\put(30,1){\tiny$v_{m_1}^1$}

\put(12,31){\tiny$v_3^0$}

\put(0,31){\tiny$v_4^0$}

\put(-3,51){\tiny$v_3^2$}
\put(15,51){\tiny$v_{m_2}^2$}

\put(-10,23){\tiny$v_5^0$}
\put(-9,10){\tiny$v_6^0$}
\put(0,1){\tiny$v_7^0$}
\put(-24,33){\tiny$v_3^3$}
\put(-12,47){\tiny$v_{m_3}^3$}
\put(41,8){\tiny$v_{m_1-1}^1$}

\put(-15,-10){Figure 2.2. $(Q^2,W^2)$. }
\end{picture}}

\end{picture}
\vspace{0.7cm}
\end{center}

Let $(Q^2,W^2)=\mu_{v_{m_1-1}^1}(Q^1,W^1)$. Then $Q^2$ is as Figure 2.2 shows and
$W^2=\sum_{w\in\cs(Q^2)}w.$
Now apply $\mu_{v_{m_1-2}^1}$ to $(Q^2,W^2)$, and so on. After $m_1-3$ steps, $(Q,W)$ is mutation equivalent to $(Q^{m_1-3},W^{m_1-3})$, where
$Q^{m_1-3}$ is as Figure 2.3 shows, and
$W^{m_1-3}=\sum_{w\in\cs(Q^{m_1-3})}w.$

Let $(Q^{m_1-2},W^{m_1-2})=\mu_{v_3^1}(Q^{m_1-3},W^{m_1-3})$. Then $Q^{m_1-2}$ is as Figure 2.4 shows, where $Q_0^1$ is an oriented cycle with $m_0+1$ vertices, and
$W^{m_1-2}=\sum_{w\in\cs(Q^{m_1-2})}w.$

\setlength{\unitlength}{0.8mm}
\begin{center}

\begin{picture}(180,60)
\put(20,0){\begin{picture}(80,80)
\put(0,0){\circle*{1.5}}
\put(-10,10){\circle*{1.5}}
\put(-10,10){\vector(1,-1){9.5}}
\put(-10,25){\circle*{1.5}}
\put(-10,25){\vector(0,-1){14.5}}
\put(0,35){\circle*{1.5}}
\put(0,35){\vector(-1,-1){9.5}}
\put(15,35){\circle*{1.5}}
\put(15,35){\vector(-1,0){14.5}}
\put(25,25){\circle*{1.5}}
\put(25,25){\vector(-1,1){9.5}}
\put(25,10){\circle*{1.5}}
\put(25,10){\vector(0,1){14.5}}
\put(15,0){\circle*{1.5}}
\put(15,0){\vector(1,1){9.5}}

\multiput(0.5,0)(1,0){14}{\circle*{0.3}}

\put(40,25){\circle*{1.5}}
\put(40,10){\circle*{1.5}}
\put(25,10){\vector(1,0){14.5}}
\put(40,25){\vector(-1,-1){14.5}}

\put(25,25){\vector(1,0){14.5}}

\put(40,10){\vector(0,1){14.5}}
\put(40,10){\vector(1,0){9.5}}
\put(50,10){\circle*{1.5}}
\put(60,10){\circle*{1.5}}
\put(60,10){\vector(1,0){9.5}}
\put(70,10){\circle*{1.5}}
\multiput(50,10)(1,0){9}{\circle*{0.3}}
\put(41,26){\tiny$v_3^1$}

\put(0,35){\vector(0,1){14.5}}
\put(0,50){\circle*{1.5}}
\put(15,50){\circle*{1.5}}
\put(15,50){\vector(0,-1){14.5}}
\multiput(0.5,50)(1,0){14}{\circle*{0.3}}

\put(-10,45){\circle*{1.5}}
\put(-10,45){\vector(1,-1){9.5}}

\put(-20,35){\circle*{1.5}}
\put(-10,25){\vector(-1,1){9.5}}
\multiput(-20,35)(1,1){10}{\circle*{0.3}}

\put(8,15){$Q_0$}

\put(7,40){$Q_2$}
\put(-12,35){$Q_3$}

\put(10,3){\tiny$v_{m_0}^0$}
\put(21,11){\tiny$v_{1}^0$}

\put(21,22){\tiny$v_{2}^0$}

\put(12,31){\tiny$v_3^0$}

\put(0,31){\tiny$v_4^0$}

\put(-3,51){\tiny$v_3^2$}
\put(15,51){\tiny$v_{m_2}^2$}

\put(-10,23){\tiny$v_5^0$}
\put(-9,10){\tiny$v_6^0$}
\put(0,1){\tiny$v_7^0$}
\put(-24,33){\tiny$v_3^3$}
\put(-12,47){\tiny$v_{m_3}^3$}
\put(39,6){\tiny$v_{4}^1$}
\put(49,6){\tiny$v_{5}^1$}
\put(69,6){\tiny$v_{m_1}^1$}

\put(-15,-10){Figure 2.3. $(Q^{m_1-3},W^{m_1-3})$. }
\end{picture}}

\put(120,0){\begin{picture}(80,80)
\put(0,0){\circle*{1.5}}
\put(-10,10){\circle*{1.5}}
\put(-10,10){\vector(1,-1){9.5}}
\put(-10,25){\circle*{1.5}}
\put(-10,25){\vector(0,-1){14.5}}
\put(0,35){\circle*{1.5}}
\put(0,35){\vector(-1,-1){9.5}}
\put(15,35){\circle*{1.5}}
\put(15,35){\vector(-1,0){14.5}}
\put(25,25){\circle*{1.5}}
\put(25,25){\vector(-1,1){9.5}}
\put(25,10){\circle*{1.5}}
\put(25,10){\vector(0,1){14.5}}
\put(15,0){\circle*{1.5}}
\put(15,0){\vector(1,1){9.5}}

\multiput(0.5,0)(1,0){14}{\circle*{0.3}}

\put(25,10){\vector(1,0){9.5}}
\put(35,10){\circle*{1.5}}
\put(45,10){\circle*{1.5}}
\put(45,10){\vector(1,0){9.5}}
\put(55,10){\circle*{1.5}}
\multiput(35,10)(1,0){9}{\circle*{0.3}}

\put(0,35){\vector(0,1){14.5}}
\put(0,50){\circle*{1.5}}
\put(15,50){\circle*{1.5}}
\put(15,50){\vector(0,-1){14.5}}
\multiput(0.5,50)(1,0){14}{\circle*{0.3}}

\put(-10,45){\circle*{1.5}}
\put(-10,45){\vector(1,-1){9.5}}

\put(-20,35){\circle*{1.5}}
\put(-10,25){\vector(-1,1){9.5}}
\multiput(-20,35)(1,1){10}{\circle*{0.3}}

\put(8,15){$Q^1_0$}

\put(7,40){$Q_2$}
\put(-12,35){$Q_3$}

\put(12,2){\tiny$v_{1}^0$}
\put(21,11){\tiny$v_{3}^1$}

\put(21,22){\tiny$v_{2}^0$}

\put(12,31){\tiny$v_3^0$}

\put(0,31){\tiny$v_4^0$}

\put(-3,51){\tiny$v_3^2$}
\put(15,51){\tiny$v_{m_2}^2$}

\put(-10,23){\tiny$v_5^0$}
\put(-9,10){\tiny$v_6^0$}
\put(0,1){\tiny$v_7^0$}
\put(-24,33){\tiny$v_3^3$}
\put(-12,47){\tiny$v_{m_3}^3$}
\put(33,6){\tiny$v_{4}^1$}

\put(53,6){\tiny$v_{m_1}^1$}
\put(-15,-10){Figure 2.4. $(Q^{m_1-2},W^{m_1-2})$. }
\end{picture}}

\end{picture}

\vspace{1cm}

\end{center}

For $(Q^{m_1-2},W^{m_1-2})$, we do mutations at the vertices in $Q_2$ similar to the above, and inductively, we get that $(Q,W)$ is mutation equivalent to
$(Q^{m-2n},W^{m-2n})$, where $m=\sum_{i=1}^nm_i$, and $Q^{m-2n}$ is as Figure 3 shows, where $Q_0^n$ is an oriented cycle with $m_0+n$ vertices, and
$W^{m-2n}=\sum_{w\in \cs(Q^{m-2n})}w.$
 \setlength{\unitlength}{0.9mm}
\begin{center}
\begin{picture}(20,55)(0,-10)

\put(0,0){\circle*{1.5}}
\put(-10,10){\circle*{1.5}}
\put(-10,10){\vector(1,-1){9.5}}
\put(-10,25){\circle*{1.5}}
\put(-10,25){\vector(0,-1){14.5}}
\put(0,35){\circle*{1.5}}
\put(0,35){\vector(-1,-1){9.5}}
\put(15,35){\circle*{1.5}}
\put(15,35){\vector(-1,0){14.5}}
\put(25,25){\circle*{1.5}}
\put(25,25){\vector(-1,1){9.5}}
\put(25,10){\circle*{1.5}}
\put(25,10){\vector(0,1){14.5}}
\put(15,0){\circle*{1.5}}
\put(15,0){\vector(1,1){9.5}}

\multiput(0.5,0)(1,0){14}{\circle*{0.3}}

\put(25,10){\vector(1,0){9.5}}
\put(35,10){\circle*{1.5}}
\put(45,10){\circle*{1.5}}
\put(45,10){\vector(1,0){9.5}}
\put(55,10){\circle*{1.5}}
\multiput(35,10)(1,0){9}{\circle*{0.3}}

\put(0,35){\vector(-1,0){9}}
\put(-10,35){\circle*{1.5}}
\multiput(-11,35)(-1,0){9}{\circle*{0.3}}
\put(-20,35){\circle*{1.5}}
\put(-20,35){\vector(-1,0){9}}
\put(-30,35){\circle*{1.5}}

\put(-10,10){\vector(-1,0){9}}
\put(-20,10){\circle*{1.5}}
\multiput(-21,10)(-1,0){9}{\circle*{0.3}}
\put(-30,10){\circle*{1.5}}
\put(-30,10){\vector(-1,0){9}}
\put(-40,10){\circle*{1.5}}

\put(8,15){$Q^n_0$}

\put(12,2){\tiny$v_{1}^0$}

\put(16,-2){\tiny$v_{i_1}^0$}

\put(21,11){\tiny$v_{3}^1$}

\put(21,22){\tiny$v_{2}^0$}
\put(26,24){\tiny$v_{i_1+1}^0$}

\put(12,31){\tiny$v_3^0$}
\put(12,38){\tiny$v_{i_2}^0$}

\put(0,31){\tiny$v_3^2$}

\put(-11,37){\tiny $v_4^2$}

\put(-31,37){\tiny $v_{m_2}^2$}

\put(-10,23){\tiny$v_4^0$}
\put(-18,28){\tiny$v_{i_2+1}^0$}
\put(-15,24){\tiny$\|$}
\put(-16,20){\tiny$v_{i_3}^0$}

\put(-9,10){\tiny$v_3^3$}

\put(-21,12){\tiny $v_4^3$}

\put(-41,12){\tiny $v_{m_3}^3$}

\put(0,2){\tiny$v_5^0$}

\put(-8,-2){\tiny$v_{i_3+1}^0$}

\put(33,6){\tiny$v_{4}^1$}

\put(53,6){\tiny$v_{m_1}^1$}

\put(-15,-10){Figure 2.5. $(Q^{m-2n},W^{m-2n})$. }
\end{picture}
\end{center}

For $(Q^{m-2n}, W^{m-2n})$, we get that $W^{m-2n}=c$ where $c$ is the chordless cycle around $Q_0^n$ in $Q^{m-2n}$. Since $Q^{m-2n}$ has only one oriented cycle, we get that every potential
on $Q$ is cyclically equivalent to $c^t$ for some $t>0$, which is in $J(W^{m-2n})$, and so the proof of \cite[Proposition 8.1]{DWZ} implies that $(Q^{m-2n}, W^{m-2n})$ is rigid.
Therefore, Proposition \ref{proposition rigid} yields that $(Q,W)$ is rigid and nondegenerate.
\end{proof}

\begin{proposition}
Let $A=J(Q,W)$ be a floriated algebra of $(Q_0,\{i_1,\dots,i_n\})$ by $Q_1,\dots,Q_n$. If one of the following conditions is satisfied, then
$A$ is a cluster-tilted algebra.

(a) There is at most one petal with more than $3$ vertices.

(b) There are two petals, denoted by $Q_j,Q_k$ for $j<k$, with more than $3$ vertices, and either $k=j+1$ and $i_{j+1}=i_j+1$, or $k=n$, $j=1$ and $i_1+m_0-i_n=1$.
\end{proposition}
\begin{proof}

First, we prove (b). We only prove for the case $j=1$, $k=2$. Proposition \ref{proposition floriated algbra to one cycle} implies that $(Q,W)$ is mutation equivalent to
$(Q^1,W^1)$ with $Q^1$ as Figure 3.1 shows, and $W^1=c$, where $c$ is the non-self-intersecting oriented cycle in $Q^1$.

We do mutations of QPs at $v_4^0$, $v_5^0$, $\dots$, $v_{m_0}^0$, $v_1^0$, $v_3^1$, $v_4^1$, $\dots$, $v_{m_1}^1$ successively, and then get a QP $(Q^l,W^l=0)$, where $Q^l$ is as Figure 3.2 shows. Since $Q^l$ is acyclic, let $\cc$ be the cluster category of $KQ^l$, then $KQ^l$ is a cluster-tilted algebra. Since mutations of cluster-tilted objects coincides with the mutations of QPs in this case, we get that $J(Q,W)$ is a cluster-tilted algebra.

\setlength{\unitlength}{0.9mm}
\begin{center}
\begin{picture}(20,50)(0,-10)

\put(0,0){\circle*{1.5}}
\put(-10,10){\circle*{1.5}}
\put(-10,10){\vector(1,-1){9.5}}
\put(-10,25){\circle*{1.5}}
\put(-10,25){\vector(0,-1){14.5}}
\put(0,35){\circle*{1.5}}
\put(0,35){\vector(-1,-1){9.5}}
\put(15,35){\circle*{1.5}}
\put(15,35){\vector(-1,0){14.5}}
\put(25,25){\circle*{1.5}}
\put(25,25){\vector(-1,1){9.5}}
\put(25,10){\circle*{1.5}}
\put(25,10){\vector(0,1){14.5}}
\put(15,0){\circle*{1.5}}
\put(15,0){\vector(1,1){9.5}}

\multiput(0.5,0)(1,0){14}{\circle*{0.3}}

\put(25,10){\vector(1,0){9.5}}
\put(35,10){\circle*{1.5}}
\put(45,10){\circle*{1.5}}
\put(45,10){\vector(1,0){9.5}}
\put(55,10){\circle*{1.5}}
\multiput(35,10)(1,0){9}{\circle*{0.3}}

\put(15,35){\vector(1,0){9.5}}
\put(25,35){\circle*{1.5}}
\put(35,35){\circle*{1.5}}
\put(35,35){\vector(1,0){9.5}}
\put(45,35){\circle*{1.5}}
\multiput(25,35)(1,0){9}{\circle*{0.3}}

\put(10,2){\tiny$v_{m_0}^0$}

\put(21,11){\tiny$v_{1}^0$}

\put(21,22){\tiny$v_{2}^0$}

\put(12,31){\tiny$v_3^0$}

\put(0,31){\tiny$v_4^0$}

\put(-10,23){\tiny$v_5^0$}

\put(-9,10){\tiny$v_6^0$}

\put(0,2){\tiny$v_7^0$}

\put(23,37){\tiny $v_4^2$}
\put(43,37){\tiny $v_{m_2}^2$}

\put(33,6){\tiny$v_{4}^1$}

\put(53,6){\tiny$v_{m_1}^1$}

\put(-10,-10){Figure 3.1. $(Q^{1},W^{1})$. }
\end{picture}
\end{center}

\setlength{\unitlength}{0.9mm}
\begin{center}
\begin{picture}(160,15)(-20,0)

\put(0,0){\circle*{1.5}}

\put(10,0){\circle*{1.5}}
\put(20,0){\circle*{1.5}}
\put(30,0){\circle*{1.5}}
\put(40,0){\circle*{1.5}}
\put(50,0){\circle*{1.5}}
\put(60,0){\circle*{1.5}}
\put(70,0){\circle*{1.5}}
\put(80,0){\circle*{1.5}}
\put(90,0){\circle*{1.5}}
\put(100,0){\circle*{1.5}}
\put(110,0){\circle*{1.5}}

\put(120,0){\circle*{1.5}}

\put(130,0){\circle*{1.5}}

\put(0,0){\vector(1,0){9}}

\multiput(10.5,0)(1,0){9}{\circle*{0.3}}
\put(20,0){\vector(1,0){9}}

\put(30,0){\vector(1,0){9}}

\put(40,0){\vector(1,0){9}}

\put(50,0){\vector(1,0){9}}

\put(70,0){\vector(1,0){9}}

\put(80,0){\vector(1,0){9}}

\put(90,0){\vector(1,0){9}}

\put(100,0){\vector(1,0){9}}

\put(120,0){\vector(1,0){9}}

\put(40,10){\circle*{1.5}}
\put(40,10){\vector(0,-1){9}}

\multiput(60.5,0)(1,0){9}{\circle*{0.3}}

\multiput(110.5,0)(1,0){9}{\circle*{0.3}}

\put(-2,1){$^{v_4^0}$}

\put(8,1){$^{v_5^0}$}
\put(16,1){$^{v_{m_0-1}^0}$}

\put(28,1){$^{v_{m_0}^0}$}

\put(38,-3){$_{v_1^0}$}
\put(38,11){$^{v_2^0}$}

\put(48,1){$^{v_4^1}$}
\put(58,1){$^{v_5^1}$}
\put(66,1){$^{v_{m_1-1}^1}$}

\put(78,1){$^{v_{m_1}^1}$}
\put(88,1){$^{v_3^0}$}

\put(98,1){$^{v_4^2}$}
\put(108,1){$^{v_5^2}$}
\put(116,1){$^{v_{m_2-1}^2}$}

\put(128,1){$^{v_{m_2}^2}$}

\put(45,-10){Figure 3.2. $(Q^{l},W^{l})$. }
\end{picture}
\vspace{1cm}
\end{center}

(a) Proposition \ref{proposition floriated algbra to one cycle} implies that $(Q,W)$ is mutation equivalent to
$(Q^1,W^1)$ with $Q^1$ as the following diagram shows, and $W^1=c$, where $c$ is the oriented cycle in $Q^1$. It is a special case of (b), we omit the proof here.

 \setlength{\unitlength}{0.9mm}
\begin{center}
\begin{picture}(20,50)(0,-10)

\put(0,0){\circle*{1.5}}
\put(-10,10){\circle*{1.5}}
\put(-10,10){\vector(1,-1){9.5}}
\put(-10,25){\circle*{1.5}}
\put(-10,25){\vector(0,-1){14.5}}
\put(0,35){\circle*{1.5}}
\put(0,35){\vector(-1,-1){9.5}}
\put(15,35){\circle*{1.5}}
\put(15,35){\vector(-1,0){14.5}}
\put(25,25){\circle*{1.5}}
\put(25,25){\vector(-1,1){9.5}}
\put(25,10){\circle*{1.5}}
\put(25,10){\vector(0,1){14.5}}
\put(15,0){\circle*{1.5}}
\put(15,0){\vector(1,1){9.5}}

\multiput(0.5,0)(1,0){14}{\circle*{0.3}}

\put(25,10){\vector(1,0){9.5}}
\put(35,10){\circle*{1.5}}
\put(45,10){\circle*{1.5}}
\put(45,10){\vector(1,0){9.5}}
\put(55,10){\circle*{1.5}}
\multiput(35,10)(1,0){9}{\circle*{0.3}}

\put(21,11){\tiny$v_{3}^1$}

\put(33,6){\tiny$v_{4}^1$}

\put(53,6){\tiny$v_{m_1}^1$}

\put(-15,-10){Figure 3.3. $(Q^{1},W^{1})$. }
\end{picture}
\end{center}

\end{proof}

At the end of this section, we consider the following algebras, generalizing floriated algebras.

\begin{definition}\label{definition of polygon-tree algebras}
Let $Q_0,Q_1,\dots,Q_N$ be oriented cycles with more than $3$ vertices for each. We define a gluing quiver $Q$ inductively.

Choose two arrows $\alpha_0\in Q_0,\alpha_1\in Q_1$. Define a quiver $Q_{0,1}$ from $Q_0$ and $Q_1$ by identifying $\alpha_0$ and $\alpha_1$.
Choose two arrows $\alpha_{0,1}\in Q_{0,1}, \alpha_2\in Q_2$, where $\alpha_{0,1}$ is different from the glued arrow $\alpha_0=\alpha_1$ in $Q_{0,1}$. Define a quiver
$Q_{0,1,2}$ from $Q_{0,1}$ and $Q_2$ by identifying $\alpha_{0,1}$ and $\alpha_2$.

Inductively, choose two arrows $\alpha_{0,1,\dots,N-1}\in Q_{0,1,\dots,N-1}$ and $\alpha_N\in Q_N$, where $\alpha_{0,1,\dots,N-1}$ is different from the glued arrows $\alpha_0,\alpha_{0,1},\dots,\alpha_{0,1,\dots,N-2}$. Define a quiver $Q=Q_{0,1,\dots,N}$ from $Q_{0,1,\dots,N-1}$ and $Q_N$ by identifying $\alpha_{0,1,\dots,N-1}$ and $\alpha_N$.

The quiver $Q$ is called a polygon-tree quiver, and $Q_0,Q_1,\dots,Q_N$ are called gluing components of $Q$. Define $W=\sum_{i=0}^Nw_i$, where $w_i$ is the non-self-intersecting oriented cycle along $Q_i$ for $0\leq i\leq N$. Then $(Q,W)$ is a QP, and its Jacobian algebra $J(Q,W)$ is called a polygon-tree algebra.
\end{definition}

Let $J(Q,W)$ be a polygon-tree algebra, where the gluing components of $Q$ are $Q_0,Q_1,\dots,Q_N$. In the following, we view these gluing components as subquivers of $Q$ naturally.
We call $Q_i$ and $Q_j$ to be \emph{adjacent} if $i\neq j$ and they have a common arrow in $Q$. For each $Q_i$, we define a full subquiver $Q(i)$ of $Q$ consisting of $Q_i$ and the quivers $Q_j$ which are adjacent to $Q_i$. We also set that $W(i)=\sum_{w\in\cs(Q(i))}w $.

\begin{lemma}
For any polygon-tree quiver $Q$ with gluing components $Q_0,Q_1,\dots,Q_N$, the Jacobian algebra associated to $(Q(i),W(i))$ is a floriated algebra for any $0\leq i\leq N$.
\end{lemma}
\begin{proof}
It follows from the definition of the polygon-tree algebra directly.
\end{proof}

In order to prove that the polygon-tree algebras are finite-dimensional, we need to give the following definitions.
For a quiver $Q$, we denote by $Path(Q)$ the set of all the oriented paths in $Q$.

\begin{definition}[\cite{TV}]
Let $Q$ be a cyclically oriented quiver such that for any arrow $\alpha$, there are at most $2$ shortest paths anti-parallel to $\alpha$. Let $c=\beta_L\dots\beta_1\beta_0$ be an oriented chordless cycle. We construct a sequence of triples $(\alpha_n,\rho_n,\rho_n')\in Q_a\times (Path(Q)\cup\{0\})\times Path(Q)$ for each $n\in\N\cup\{0\}$, such that $\rho_n$ and $\rho_n'$
are the shortest paths anti-parallel to $\alpha_n$ or $\rho_n=0$ and $\rho_n'$ is the shortest path anti-parallel to $\alpha_n$ if there is only one shortest path anti-parallel to $\alpha_n$, in the following way:

\begin{itemize}
\item[Step 0] We denote by $\alpha_0$ the arrow $\beta_0$ and by $\rho_0'$ the shortest path $\beta_L\dots\beta_1$ anti-parallel to the arrow $\alpha_0$ in $c$. If there exists other shortest path $\rho_0$ anti-parallel to $\alpha_0$ different to $\rho_0'$, then the first element of the sequence is $(\alpha_0,\rho_0,\rho_0')$. Otherwise, the sequence is $(\alpha_n,\rho_n,\rho_n')=(\alpha_0,0,\rho_0')$ for all $n$.
\item[Step 1] We denote by $\alpha_1$ the arrow in the path $\rho_0$ such that $t(\alpha_0)=s(\alpha_1)$ and by $\rho_1'$ the shortest path anti-parallel to $\alpha_1$ in $\rho_0\alpha_0$.
If there exists other shortest path $\rho_1$ anti-parallel to $\alpha_1$ different to $\rho_1'$, then the second element of the sequence is $(\alpha_1,\rho_1,\rho_1')$. Otherwise, $(\alpha_n,\rho_n,\rho_n')=(\alpha_1,0,\rho_1')$ for each $n\geq1$.

\item[] {\hspace{5cm}$\vdots$}

\item[Step i] We denote by $\alpha_i$ the arrow in the path $\rho_{i-1}$ such that
$s(\alpha_{i-1})=t(\alpha_{i})$ if $i$ is even or, $t(\alpha_{i-1})=s(\alpha_{i})$ if $i$ is odd, and by $\rho_i'$ the shortest path anti-parallel to $\alpha_i$ in $\rho_{i-1}\alpha_{i-1}$. If there exists other shortest path $\rho_i$ anti-parallel to $\alpha_i$ different to $\rho_i'$, then the $i+1$th element of the sequence is $(\alpha_i,\rho_i,\rho_i')$. Otherwise, $(\alpha_n,\rho_n,\rho_n')=(\alpha_i,0,\rho_i')$ for each $n\geq i$.
\end{itemize}
The sequence $\{(\alpha_n,\rho_n,\rho_n')\}_{\N\cup\{0\}}$ is called the cyclic sequence of $c$. We say that the cyclic sequence $\{(\alpha_n,\rho_n,\rho_n')\}_{\N\cup\{0\}}$ is finite if there exists $m\in\N$ such that
$(\alpha_n,\rho_n,\rho_n')=(\alpha_m,0,\rho_m')$ for every $n\geq m$.
\end{definition}

\begin{lemma}[\cite{TV}]\label{lemma finite dimension of simplex gluing Nakayama algebra}
Let $Q$ be a cyclically oriented quiver such that:

(a) for any arrow $\alpha$, there are at most $2$ shortest paths anti-parallel to $\alpha$;

(b) the cyclic sequence of any oriented chordless cycle $c$ is finite.

\noindent If $W$ is a primitive potential of $Q$, then the Jacobian algebra $J(Q,W)$ is a finite-dimensional algebra.
\end{lemma}

\begin{proposition}\label{schurian of polygon-tree algebras}
Let $A=J(Q,W)$ be a polygon-tree algebra. Then $A$ is a finite-dimensional algebra. Moreover, $A$ has the property that $_AS_i$ is not a submodule of the radical of $P_i$ for any vertex $i\in Q$.
\end{proposition}
\begin{proof}
Let $Q_0,Q_1,\dots,Q_N$ be the gluing components. Assume that $Q_i$ has $m_i$ vertices for $0\leq i\leq N$.

First, we prove that $Q$ is cyclically oriented by induction on $N$. If $N=0$, there is nothing to prove.
We assume that $Q_{0,\dots,N-1}$ is cyclically oriented, where $Q_{0,\dots,N-1}$ is a subquiver of $Q$ formed by $Q_0,\dots,Q_{N-1}$.
Then $Q$ is as the following diagram shows.

\setlength{\unitlength}{0.9mm}
\begin{center}
\begin{picture}(20,60)(0,-10)

\put(0,0){\circle*{1.5}}
\put(0,0){\vector(0,1){14}}
\put(0,15){\circle*{1.5}}
\put(0,15){\vector(1,0){14}}

\put(15,15){\circle*{1.5}}

\put(15,15){\vector(0,-1){14}}
\put(15,0){\circle*{1.5}}

\put(15,15){\vector(0,1){14}}
\put(15,30){\circle*{1.5}}

\put(0,30){\circle*{1.5}}
\put(0,30){\vector(0,-1){14}}

\multiput(0.5,0)(1,0){14}{\circle*{0.3}}
\multiput(0.5,30)(1,0){14}{\circle*{0.3}}
\qbezier[34](0,15)(-20,30)(0,45)

\qbezier[34](15,15)(35,30)(15,45)
\qbezier[14](0,45)(7.5,48)(15,45)

\put(5,5){$Q_N$}

\put(0,35){$Q_{0,\dots,N-1}$}

\put(-8,8){$_{\gamma_{m_N}}$}
\put(16,8){$_{\gamma_{1}}$}
\put(6,15){$^{\gamma_0}$}
\put(-10,-8){Figure 4. The quiver $Q$.}
\end{picture}
\end{center}

For any chordless cycle $c$ of $Q$, if $c$ lies in the subquiver $Q_{0,\dots,N-1}$, then $c$ is oriented by our inductive assumption. Otherwise, if $c$ does not lie in the subquiver $Q_{0,\dots, N-1}$, then it is easy to see that $\gamma_{m_N}\dots \gamma_1$ is a subpath of $c$ since $c$ is a cycle. If $c$ is not equivalent to $\gamma_{m_N}\dots \gamma_1\gamma_0$, then $\gamma_0$ is a chord of $c$, which is a contradiction. So $c$ is equivalent to $\gamma_{m_N}\dots \gamma_1\gamma_0$, which is oriented. Thus, $Q$ is cyclically oriented.

Similar to the above, we also can get that the oriented chordless cycles are precisely the non-self-intersecting cycles $Q_i$ for $0\leq i\leq N$. So $W$ is a primitive potential.
By the construction of $Q$, it is easy to see that for any arrow $\alpha$, there are at most $2$ shortest paths anti-parallel to $\alpha$.

It is also easy to get that the cyclic sequence of any oriented chordless cycle $c$ is finite from the ``tree-like'' nature of the polygon-tree quiver, we omit the proof here. Then Lemma \ref{lemma finite dimension of simplex gluing Nakayama algebra} implies that $A$ is finite-dimensional. Furthermore, by the proof of Lemma \ref{lemma finite dimension of simplex gluing Nakayama algebra} in \cite{TV}, any non-self-intersecting oriented cycle $c$ is zero in $A$. So $A$ has the property that $_AS_i$ is not a submodule of the radical of $P_i$ for any vertex $i\in Q$.
\end{proof}

\begin{corollary}
Let $A$ be a polygon-tree algebra. Then $A$ is a $2$-CY-tilted algebra. In particular, $A$ is a Gorenstein algebra of dimension at most $1$.
\end{corollary}
\begin{proof}
This is obvious from \cite[Corollary 3.7]{Am} since $A$ is a finite-dimensional Jacobian algebra. The last statement follows from \cite[Proposition 2.1]{KR1} immediately, see also \cite[Theorem 4.4]{KZ}.
\end{proof}

At the end of this section, we prove that all polygon-tree algebras are schurian algebras.

\begin{definition}\label{definition of floriated algebra}
Let $A=KQ/I$ be a floriated algebra of $(Q_0,\{i_1,\dots,i_n\})$ by $Q_1,\dots,Q_n$. Define $d_1,d_2,\dots,d_n$ to be
$$d_1=i_2-i_1,d_2=i_3-i_2,\dots,d_n=i_1+m_0-i_n,$$
and define $d_Q$ to be $\sharp\{d_i|d_i=1,i=1,\dots,n\}$. $d_i$ is called the distance from $Q_i$ to $Q_{i+1}$.
\end{definition}

\begin{theorem}\label{lemma schurian algebra of simple polygon-tree algebras}
Let $A=KQ/I$ be a polygon-tree algebra. Then $A$ is a schurian algebra.
\end{theorem}
\begin{proof}
We denote the gluing components of the polygon-tree quiver $Q$ by $Q_0$, $Q_1$, $\dots$, $Q_N$, and prove the statement by induction on $N$.

For any polygon-tree quiver with only one gluing component, the result is obvious.

We assume that the result holds for any polygon-tree quiver with no more than $N$ gluing components.

For $Q$ with $N+1$ gluing components, renumbering $Q_0,Q_1,\dots,Q_N$, we can assume that $Q_1$ has only one adjacent quiver, which is denoted by $Q_0$, and $Q(0)$ is the floriated quiver of $(Q_0,\{i_1,\dots,i_n\})$ by $Q_1,\dots,Q_n$. We construct a graph $T(Q)$ from $Q$ with vertices the gluing components $Q_0,Q_1,\dots, Q_N$. If $Q_i$ and $Q_j$ are adjacent, we add an edge between them. Then $T(Q)$ is a tree.
We can view $Q_0$ as the root, so that the other vertices form $n$ branches. We order the branches in such a way that the $i$-th branch is that containing $Q_i$.

The proof can be broken up into the following four cases.

{\bf Case (1)}: $Q(0)$ satisfies $d_1=i_2-i_1>1$ and $d_n=i_1+m_0-i_n>1$. We assume that $Q$ is as Figure 6.1 shows.

{\bf Case (2)}: $Q(0)$ satisfies $d_1=i_2-i_1=1$ and $d_n=i_1+m_0-i_n>1$. We assume that $Q$ is as Figure 7.1 shows.

{\bf Case (3)}: $Q(0)$ satisfies $d_1=i_2-i_1>1$ and $d_n=i_1+m_0-i_n=1$. It follows immediately by considering Case (2) in $A^{op}$ which is also a polygon-tree algebra.

{\bf Case (4)}: $Q(0)$ satisfies $d_1=i_2-i_1=1$ and $d_n=i_1+m_0-i_n=1$. We assume that $Q$ is as Figure 8.1 shows.

We give the proof only in Case (4), since the others are similar.

In Case (4), as Figure 8.1 shows, let $U=KQ'/I'$ be the polygon-tree algebra where $Q'$ is the polygon-tree subquiver of $Q$ with gluing components $Q_0,Q_2,\dots,Q_N$. For simplicity, we denote by $i$ the vertex $v^0_i$ in $Q_0$ for each vertex $v_i^0$. Then $U$ is a schurian algebra by the inductive assumption and it is a quotient algebra of $A$. Let $V=KQ'/I''$ be the full subalgebra of $A$ on the quiver $Q'$. Then $I'=\langle I'',  \xi\rangle$, where $\xi$ is the path $2\xrightarrow{\gamma_2} 3\xrightarrow{\gamma_3}\cdots\xrightarrow{\gamma_{m_0-1}} m_0\xrightarrow{\gamma_{m_0}} 1$ located in $Q_0$.

Let $a,b\in Q$ be two vertices. For any path $\alpha$ between them in $Q$, if there is an oriented cycle appearing as a subpath of $\alpha$, then $\alpha\in I$ by Proposition \ref{schurian of polygon-tree algebras}. So we only need to check that $\dim e_b A e_a\leq 1$ for $a\neq b$. The proof can be broken up into the following four cases.

{\bf Case (4a)}: $a,b\in Q_1$.

If $a=1$ and $b=2$, then there is an arrow $\gamma_1$ from $1$ to $2$. For any other path $\alpha$ from $1$ to $2$ different to $\gamma_1$, since every oriented cycle is in $I$, we assume that $\alpha$ is in $Q'$. Because $U=KQ'/I'$ is a schurian algebra, there exist $k_1,k_2\in K$, not all zero, such that $k_1\alpha+k_2\gamma_1\in I'$, which yields that $\alpha\in I'$. So there exist $p_1\in I''$ and $p_2\in e_2KQ'\xi KQ'e_1$ such that
$\alpha=p_1+p_2$. Obviously, $p_1\in I$. Since $p_2$ passes through an oriented cycle, we get that $p_2\in I$, and then $\alpha\in I$. Therefore, there is only one path $\gamma_1$ from $1$ to $2$ which is not in $I$, and then $\dim e_2 A e_1=1$.

For any path $\alpha$ from $a=2$ to $b=1$ in $Q$, if $\alpha$ is not located in $Q'$, then it is not hard to see that the path $\eta:2\rightarrow v_3^1\rightarrow \cdots \rightarrow v_{m_1}^1\rightarrow1$ is a subpath of $\alpha$. Since $\eta-\xi\in I$, after replacing all $\eta$ with $\xi$ in $\alpha$, we can assume that $\alpha$ is in $Q'$.
Because $Q'$ is a polygon-tree quiver, we get that $\alpha$ passes through the vertices $3,4,\dots, m_0$ in $Q_0$, and then $\alpha=\alpha_{m_0}\cdots\alpha_3\alpha_2$ such that $i+1$ is the end point of $\alpha_i$, and the starting point of $\alpha_{i+1}$ for any $2\leq i\leq m_0-1$. For any $2\leq i\leq m_0$, there is an arrow $\gamma_i$ from $i$ to $i+1$ in $Q'$ (where $m_0+1=1$), similar to the above, if $\alpha_i$ is different to $\gamma_i$, then we get that $\alpha_i\in I$, and then $\alpha\in I$. So if $\alpha\notin I$, then $\alpha=\xi$, which implies that $\dim e_1A e_2=1$.

For both of $a=v_i^1$ and $b=v_j^1$ not in $\{1,2\}$, if $3\leq i< j\leq m_1$, then there is only one path from $a$ to $b$ which does not go through any oriented cycle. So $\dim e_b A e_a\leq1$.
If $3\leq j< i\leq m_1$, then any path $\alpha$ from $a$ to $b$ passes through the vertices $1$ and $2$. So $\alpha=\beta_3\beta_2\beta_1$, where $\beta_1$ is the path $a=v_i^1\rightarrow \cdots\rightarrow v_{m_1}^1\rightarrow 1$, and $\beta_3$ is the path $2\rightarrow \cdots \rightarrow v_{j}^1=b$ in $Q_1$. We can assume that $\beta_2$ is a path from $1$ to $2$ in $Q'$. From the above, if $\beta_2\neq \gamma_1$, then $\beta_2\in I$. So if $\alpha\notin I$, then $\alpha=\beta_3\gamma_1\beta_1$, which implies that $\dim e_bA e_a\leq1$.

For $a=1$ and $b=v_j^1\notin \{1,2\}$, then any path $\alpha$ from $1$ to $b$ is of form
$\beta_2\beta_1$, where $\beta_1$ is a path from $1$ to $2$, and $\beta_2$ is the path
$2\rightarrow v_3^1\rightarrow\cdots\rightarrow v_j^1$ in $Q_1$.
From the above, we get that if $\alpha\notin I$, then $\beta_1=\gamma_1$, so $\alpha=\beta_2\gamma_1$, which implies that $\dim e_bAe_1\leq 1$.

For $a=2$ and $b=v_j^1\notin\{1,2\}$, there is only one path $\beta_2=2\rightarrow v_3^1\rightarrow\cdots\rightarrow v_j^1$ from $a=2$ to $b$ which does not pass through any oriented cycle. So $\dim e_bAe_2\leq 1$.

For $a=v_i^1\notin \{1,2\}$ and $b\in\{1,2\}$, by considering the polygon-tree algebras $A^{op}$, we can get that $\dim e_b A e_a\leq1$ from the above.

{\bf Case (4b)}: $a=v_i^1\in Q_1,b\in Q\setminus Q_1$. If $a\neq 2$, then any path $\alpha$ in $Q$ from $a=v_i^1$ to $b$ is of form $a=v_i^1\rightarrow v_{i+1}^1\rightarrow\cdots \rightarrow v_{m_1}^1\rightarrow 1\rightarrow \cdots \rightarrow b$. Denote by $\beta$ the path $a=v_i^1\rightarrow v_{i+1}^1\rightarrow\cdots \rightarrow v_{m_1}^1\rightarrow 1$.

For any two paths $\alpha_1,\alpha_2$ from $a$ to $b$ in $Q$, there are two paths $\delta_1,\delta_2$ from $1$ to $b$ in $Q$ such that $\alpha_i=\delta_i\beta$. We need to check that $\alpha_1,\alpha_2$ are linearly dependent in $A$. Since every oriented cycle is in $I$, we only need prove for the case: $\delta_1$ and $\delta_2$ are in $Q'$. Because $U=KQ'/I'$ is a schurian algebra, there exist $k_1,k_2\in K$, not all zero, such that $k_1\delta_1+k_2\delta_2\in I'$.
Since $I'=\langle I'',  \xi\rangle$, there exist $p_1\in I''$ and $p_2\in e_bKQ'\xi KQ'e_1$ such that $k_1\delta_1+k_2\delta_2=p_1+p_2$, and then
$k_1\alpha_1+k_2\alpha_2= p_1\beta+p_2\beta$. It is easy to see that $p_1\beta\in I$. Since each oriented cycle is in $I$, we get that $p_2\beta\in I$ and then $k_1\alpha_1+k_2\alpha_2\in I$. Therefore, $\dim e_b A e_a\leq1$.

If $a=2$, for any two paths $\alpha_1,\alpha_2$ from $2$ to $b$, then we need to check that they are linearly dependent in $A$. We also only need to prove it for the case : $\alpha_1$ and $\alpha_2$ are in $Q'$, and both of them do not go through any oriented cycles.
Since $U$ is a schurian algebra, we get that there exist $k_1,k_2\in K$, not all zero, such that $k_1\alpha_1+k_2\alpha_2\in I'$. Since $I'=\langle I'',  \xi\rangle$, there exist $q_1\in I''$ and $q_2\in e_bKQ'\xi KQ'e_2$ such that $k_1\alpha_1+k_2\alpha_2=q_1+q_2$. Obviously, $q_1\in I$.

If $q_2\in I''$, then $\alpha_1$ and $\alpha_2$ are linearly dependent in $V$ and then in $A$.

If $q_2\notin I''$, then there exists at least one path $l=l_2\xi l_1$ from $a=2$ to $b$
in $Q'$, where $l_1$ is of the form $2\rightarrow\cdots \rightarrow 2$, and $l_2$ is of the form $1\rightarrow \cdots \rightarrow b$, such that $l\notin I''$. So $l_1=e_2$. Note that $l_2$ does not pass through $\gamma_1$. In fact, otherwise, $l$ passes through an oriented cycle, and then $l\in I''$, contradicts. If $b$ is not located in an oriented cycle lying in the $n$-th branch, then the subpath $l_2$ passes through the vertex $m_0$ since $Q$ is a polygon-tree quiver, which yields that $l\in I''$ since every oriented cycle in $Q'$ is in $I''$, contradicts. So $b$ is located in an oriented cycle lying in the $n$-th branch and $b\neq m_0$.

Since $\alpha_1$ is a path from $2$ to $b$ in $Q'$ and $Q'$ is a polygon-tree quiver, $\alpha_1$ passes through the vertex $m_0$. Then $\alpha_1=\beta_2\beta_1$, where $\beta_1$ is a path from $2$ to $m_0$, $\beta_2$ is a path from $m_0$ to $b$. Since $\beta_1$ and $\gamma_{m_0-1}\cdots \gamma_2$ are two paths from $2$ to $m_0$ in $Q'$, they are linear dependent in $U$. Similar to the discussion in the above, we get that $\beta_1$ and $\gamma_{m_0-1}\cdots \gamma_2$ are linearly dependent in $V$. Since $\beta_2$ and the subpath of $l$: $l_2\gamma_{m_0}$ are two paths from $m_0$ to $b$ in $Q'$, we get that they are linearly dependent in $U$. Similar to the discussion in the above, we can get that they are linearly dependent in $V$. Therefore, $\alpha_1$ and $l$ are linearly dependent in $A$. Similarly, $\alpha_2$ and $l$ are linearly dependent in $A$, and then $\alpha_1$ and $\alpha_2$ are linearly dependent in $A$.

{\bf Case (4c)}: $a\in Q\setminus Q_1, b=v_i^1\in  Q_1$. It follows immediately by considering Case (4b) in the polygon-tree algebra $A^{op}$.

{\bf Case (4d)}: $a,b\in Q\setminus Q_1$. For any path $\alpha$ from $a$ to $b$ in $Q$, similarly, we can assume that $\alpha$ is in $Q'$.

For any two paths $\alpha_1,\alpha_2$ from $a$ to $b$, we need to check that they are linearly dependent in $A$. We also only need prove for the case : $\alpha_1$ and $\alpha_2$ are in $Q'$ which do not go through any oriented cycles.
Since $U$ is a schurian algebra, we get that there exist $k_1,k_2\in K$, not all zero, such that $k_1\alpha_1+k_2\alpha_2\in I'$. Since $I'=\langle I'',  \xi\rangle$, there exist $p_1\in I''$ and $p_2\in e_bKQ'\xi KQ'e_a$ such that $k_1\alpha_1+k_2\alpha_2=p_1+p_2$. Obviously, $p_1\in I$.

If $p_2\in I''$, then $\alpha_1$ and $\alpha_2$ are linearly dependent in $V$ and then in $A$.

If $p_2\notin I''$, then there exists at least one path $l=l_2\xi l_1$ from $a$ to $b$
in $Q'$, where $l_1$ is of form $a\rightarrow\cdots \rightarrow 2$, and $l_2$ is of form $1\rightarrow \cdots \rightarrow b$, such that $l\notin I''$. Note that both of $l_1$ and $l_2$ do not pass through $\gamma_1$. If $a$ is not located in an oriented cycle lying in the second branch, then the subpath $l_1$ passes through the vertex $3$ since $Q$ is a polygon-tree quiver, which yields that $l\in I''$ since every oriented cycle in $Q'$ is in $I''$, contradicts. So $a$ is located in an oriented cycle lying in the second branch and $a\neq 3$. Similarly, $b$ is located in an oriented cycle lying in the $n$-th branch and $b\neq m_0$.

Since $a$ is located in an oriented cycle lying in the second branch and $b$
is located in an oriented cycle lying in the $n$-th branch, we get that $\alpha_1$ goes through the vertices $3$ and $m_0$ since $Q'$ is a polygon-tree quiver. So we assume that $\alpha_1=\beta_3\beta_2\beta_1$ with $\beta_1$ from $a$ to $3$, $\beta_2$ from $3$ to $m_0$ and $\beta_3$ from $m_0$ to $b$.
Similar to the above, we can get that $\beta_1$ and $\gamma_2l_1$, $\beta_2$ and $\gamma_{m_0-1}\cdots \gamma_2$, $\beta_3$ and $l_2\gamma_{m_0}$ are three linearly dependent collections in $V$, and so $\alpha_1$ and $l$ are linearly dependent in $A$.

Similarly, $\alpha_2$ and $l$ are linearly dependent in $V$, which implies that $\alpha_1$ and $\alpha_2$ are linearly dependent in $A$. So $\dim e_bA e_a\leq1$.

To sum up, we get that $\dim e_bAe_a\leq 1$ for any two vertices $a$, $b$ in $Q$, which means $A$ is a schurian algebra.
\end{proof}

\section{Singularity categories of polygon-tree algebras}

In order to prove the main result of this section, we describe a construction of matrix algebras which is obtained by X-W. Chen in \cite{Chen2}, see also \cite[Section 4]{KN}. Let $A$ be a finite-dimensional algebra over a field $K$. Let $_AM$ and $N_A$ be a left and right $A$-module, respectively. Then $M\otimes_KN$ becomes an $A\mbox{-}A$-bimodule. Consider an $A\mbox{-}A$-bimodule monomorphism $\phi:M\otimes_K N\rightarrow A$ such that $\phi$ vanishes both on $M$ and $N$. Note that $\Im \phi\subseteq A$ is an ideal. The vector space $\Gamma=\left( \begin{array}{cc} A &M\\ N& K\end{array}\right)$ becomes an associative algebra via the multiplication
$$\left( \begin{array}{cc} a &m\\ n& \lambda\end{array}\right)\left( \begin{array}{cc} a' &m'\\ n'& \lambda'\end{array}\right)=\left( \begin{array}{cc} aa'+\phi(m\otimes n) &am'+\lambda'm\\ na'+\lambda n'& \lambda\lambda'\end{array}\right).$$
\begin{proposition}[\cite{Chen2}]\label{proposition homological epimorphism induces singularity equivalences}
Keep the notation and assumptions as above. Then there is a triangle equivalence $D^b_{sg}(\Gamma)\simeq D^b_{sg}(A/\Im\phi)$.
\end{proposition}

Note that the above construction contains the one-point extension and one-point coextension of algebras, where $M$ or $N$ is zero, see also \cite{Chen1}.

\begin{definition}\label{definition of simple polygon-tree alegbra}
Let $A=KQ/I$ be a polygon-tree algebra, where the gluing components of $Q$ are $Q_0,Q_1,\dots,Q_N$. If $Q$ does not admit any polygon-tree subquiver (with four gluing components) of one of the two forms in Figure 5,
then $A$ is called a simple polygon-tree algebra. In this case, $Q$ is called a simple polygon-tree quiver.
\end{definition}
\setlength{\unitlength}{0.5mm}
\begin{center}
\begin{picture}(180,80)(0,-10)
\put(10,0){\begin{picture}(80,60)

\put(20,20){\circle*{2}}
\put(40,20){\circle*{2}}

\put(20,40){\circle*{2}}
\put(20,60){\circle*{2}}
\put(20,20){\line(0,1){20}}
\put(20,40){\line(0,1){20}}

\put(0,40){\circle*{2}}
\put(0,40){\line(1,0){20}}
\put(0,60){\line(1,0){20}}
\put(0,60){\circle*{2}}

\put(40,20){\circle*{2}}
\put(40,20){\line(0,1){20}}
\put(20,40){\line(1,0){20}}
\put(20,60){\line(1,0){20}}

\put(40,60){\circle*{2}}
\put(40,40){\circle*{2}}

\qbezier[15](0,40)(0,50)(0,60)

\qbezier[15](40,40)(40,50)(40,60)

\qbezier[15](20,20)(30,20)(40,20)

\qbezier[15](0,10)(0,20)(0,30)

\put(0,30){\circle*{2}}
\put(0,30){\line(2,1){20}}
\put(0,10){\circle*{2}}
\put(0,10){\line(2,1){20}}
\end{picture}}

\put(100,0){\begin{picture}(80,60)

\put(20,20){\circle*{2}}
\put(40,20){\circle*{2}}

\put(20,40){\circle*{2}}
\put(20,60){\circle*{2}}
\put(20,20){\line(0,1){20}}
\put(20,40){\line(0,1){20}}

\put(0,40){\circle*{2}}
\put(0,40){\line(1,0){20}}
\put(0,60){\line(1,0){20}}
\put(0,60){\circle*{2}}

\put(40,20){\circle*{2}}
\put(40,20){\line(0,1){20}}
\put(20,40){\line(1,0){20}}
\put(20,60){\line(1,0){20}}

\put(40,60){\circle*{2}}
\put(40,40){\circle*{2}}

\qbezier[15](0,40)(0,50)(0,60)

\qbezier[15](40,40)(40,50)(40,60)

\qbezier[15](20,20)(30,20)(40,20)

\put(40,40){\line(1,0){20}}
\put(40,20){\line(1,0){20}}
\put(60,20){\circle*{2}}
\put(60,40){\circle*{2}}

\qbezier[15](60,20)(60,30)(60,40)

\end{picture}}
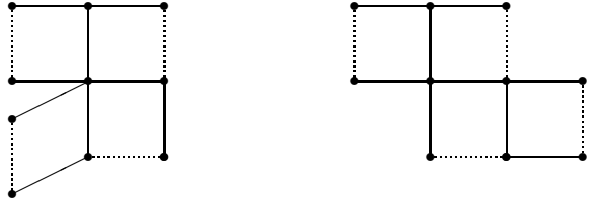
\put(-10,0){Figure 5. The banned subquivers: the edges may be oriented}
\put(-10,-10){in any way such that the resulting quiver is cyclically oriented.}
\end{picture}
\end{center}

For schurian algebras, including polygon-tree algebra, it is a little easier to check that the negative and positive mutations are defined. In fact, all the algebras appearing in the Lemmas \ref{lemma singularity category of cluster-tilted algebra of D}-\ref{the final lemma singularity category of cluster-tilted algebra of D 2} are schurian algebras.

\begin{lemma}[\cite{La,BHL2}]\label{proposition good mutation}
Let $\Gamma$ be a schurian algebra.

(a) The negative mutation $\mu_k^-(\Gamma)$ is defined if and only if for any non-zero path $k\rightsquigarrow i$ starting at $k$ and ending at some vertex $i$, there exists an arrow $j\rightarrow k$ such that the composition $j\rightarrow k\rightsquigarrow i$ is non-zero.

(b) The positive mutation $\mu_k^+(\Gamma)$ is defined if and only if for any non-zero path $i\rightsquigarrow k$ starting at some vertex $i$ and ending at $k$, there exists an arrow $k\rightarrow j$ such that the composition $i\rightsquigarrow k\rightarrow j$ is non-zero.
\end{lemma}

\noindent {\bf Conventions}\emph{
For any polygon-tree quiver $Q$ with gluing components $Q_0,Q_1,\dots,Q_N$, we assume that $Q_i$ has $m_i$ vertices for any $0\leq i\leq N$. Set $m=\sum_{i=0}^N m_i$, and $d_{Q}=\sum_{i=0}^N d_{Q(i)}$, where $d_{Q(i)}$ is defined since $Q(i)$ is a floriated quiver. It is easy to see that there is no confusion if $Q$ is a floriated quiver.
}

In the following, we denote by $\cn_d$ the self-injective Nakayama algebra given by the path algebra of a cyclic quiver with $d$ vertices modulo the ideal generated by paths of length $d-1$ for any $d\geq3$.

\begin{theorem}\label{proposition to type Q(m,{i_1,...,i_r})}
Let $A=KQ/I$ be a simple polygon-tree algebra, where the gluing components of $Q$ are $Q_0$, $Q_1$, $\dots$, $Q_N$. Then
$$D^b_{sg}(A)\simeq \underline{\mod }(\cn_{m-3N+d_Q}).$$
In particular, $A$ is CM-finite.
\end{theorem}

\begin{proof}
We prove the statement by induction on $N$.

For any simple polygon-tree algebra with only one gluing component, the result is obvious.

We assume that the result holds for any simple polygon-tree algebra with no more than $N$ gluing components.

For $Q$ with $N+1$ gluing components, renumbering $Q_0,Q_1,\dots,Q_N$, we can assume that $Q_1$ has only one adjacent quiver, which is denoted by $Q_0$, and $Q(0)$ is the floriated quiver of $(Q_0,\{i_1,\dots,i_n\})$ by $Q_1,\dots,Q_n$.

The proof can be broken up into the following four cases.

{\bf Case (1):} $Q(0)$ satisfies $d_1=i_2-i_1>1$ and $d_n=i_1+m_0-i_n>1$. We prove it in Lemma \ref{lemma singularity category of cluster-tilted algebra of D}.

{\bf Case (2):} $Q(0)$ satisfies $d_1=i_2-i_1=1$ and $d_n=i_1+m_0-i_n>1$. We prove it in Lemma \ref{the other lemma singularity category of cluster-tilted algebra of D}.

{\bf Case (3):} $Q(0)$ satisfies $d_1=i_2-i_1>1$ and $d_n=i_1+m_0-i_n=1$. We prove it in Lemma \ref{the another lemma singularity category of cluster-tilted algebra of D}.

{\bf Case (4):} $Q(0)$ satisfies $d_1=i_2-i_1=1$ and $d_n=i_1+m_0-i_n=1$. We prove it in Lemma \ref{the final lemma singularity category of cluster-tilted algebra of D 2}.

In conclusion, we get that $$D^b_{sg}(A)\simeq \underline{\mod }(\cn_{m-3N+d_Q}).$$
\end{proof}

\begin{lemma}\label{lemma singularity category of cluster-tilted algebra of D}
Keep the notation as in the proof of Theorem \ref{proposition to type Q(m,{i_1,...,i_r})}. If $Q(0)$ satisfies $d_1=i_2-i_1>1$ and $d_n=i_1+m_0-i_n>1$, then
$$D^b_{sg}(A)\simeq \underline{\mod }(\cn_{m-3N+d_Q}).$$
\end{lemma}

\begin{proof}
Without loss of generality, we can assume that $Q$ is as Figure 6.1 shows.
\setlength{\unitlength}{0.5mm}
\begin{center}
\begin{picture}(200,70)(0,-10)
\put(-20,0){\begin{picture}(60,60)
\put(60,0){\circle*{2}}
\put(60,0){\vector(0,1){19}}
\put(59,-4){${}_1$}
\put(54,1){${}_{i_1}$}

\put(60,20){\circle*{2}}

\put(60,20){\vector(-1,1){14}}
\put(59,23){${}_{2}$}
\put(47,17){${}_{i_1+1}$}

\put(60,20){\vector(1,1){9}}
\put(70,30){\circle*{2}}
\put(70,30){\vector(1,0){9}}
\put(80,30){\circle*{2}}
\put(80,30){\vector(1,-1){9}}
\put(90,20){\circle*{2}}

\put(90,20){\vector(0,-1){19}}
\put(90,0){\circle*{2}}
\put(90,0){\vector(-1,-1){9}}
\put(80,-10){\circle*{2}}
\put(70,-10){\vector(-1,1){9}}
\put(70,-10){\circle*{2}}
\qbezier[8](80,-10)(75,-10)(70,-10)
\put(45,35){\circle*{2}}
\put(45,35){\vector(-1,0){19}}
\put(47,37){${}_{3}$}
\put(43,31){${}_{i_2}$}

\put(25,35){\circle*{2}}

\put(25,35){\vector(-1,-1){14}}
\put(21,37){${}_{4}$}
\put(23,31){${}_{i_2+1}$}

\put(25,35){\vector(0,1){19}}

\put(25,55){\circle*{2}}

\put(22,55){$^{v_3^2}$}
\qbezier[14](25,55)(35,55)(45,55)

\put(45,55){\circle*{2}}
\put(45,55){\vector(0,-1){19}}

\put(45,55){$^{v_{m_2}^2}$}
\put(10,20){\circle*{2}}
\put(10,20){\vector(0,-1){19}}

\put(5,20){${}_{5}$}

\put(10,0){\circle*{2}}
\put(5,0){${}_{6}$}

\put(45,-15){\circle*{2}}
\put(45,-15){\vector(1,1){14}}
\put(47,-16){\tiny$m_0$}

\put(70,30){$^{v_{3}^1}$}
\put(80,30){$^{v_{4}^1}$}
\put(90,20){$^{v_{5}^1}$}

\put(90,-2){$_{v_{6}^1}$}

\put(80,-12){$_{v_{7}^1}$}

\put(68,-14){$_{v_{m_1}^1}$}

\qbezier[25](10,0)(20,-25)(45,-15)

\put(30,10){$Q_0$}
\put(30,42){$Q_2$}
\put(70,10){$Q_1$}
\put(10,-30){Figure 6.1. Quiver of $A$}
\end{picture}}

\put(100,0){\begin{picture}(80,70)
\put(60,0){\circle*{2}}
\put(60,0){\vector(0,1){19}}
\put(59,-4){${}_1$}
\put(54,1){${}_{i_1}$}

\put(60,20){\circle*{2}}

\put(60,20){\vector(-1,1){14}}
\put(59,23){${}_{2}$}
\put(47,17){${}_{i_1+1}$}

\put(70,30){\vector(-1,-1){9}}
\put(70,30){\circle*{2}}
\put(80,30){\vector(-1,0){9}}
\put(60,20){\vector(2,1){19}}
\put(80,30){\circle*{2}}
\put(80,30){\vector(1,-1){9}}
\put(90,20){\circle*{2}}

\put(90,20){\vector(0,-1){19}}
\put(90,0){\circle*{2}}
\put(90,0){\vector(-1,-1){9}}
\put(80,-10){\circle*{2}}
\put(70,-10){\vector(-1,1){9}}
\put(70,-10){\circle*{2}}
\qbezier[8](80,-10)(75,-10)(70,-10)
\put(45,35){\circle*{2}}
\put(45,35){\vector(-1,0){19}}
\put(47,37){${}_{3}$}
\put(43,31){${}_{i_2}$}

\put(25,35){\circle*{2}}

\put(25,35){\vector(-1,-1){14}}
\put(21,37){${}_{4}$}
\put(23,31){${}_{i_2+1}$}

\put(25,35){\vector(0,1){19}}

\put(25,55){\circle*{2}}

\qbezier[14](25,55)(35,55)(45,55)

\put(45,55){\circle*{2}}
\put(45,55){\vector(0,-1){19}}

\put(10,20){\circle*{2}}
\put(10,20){\vector(0,-1){19}}

\put(5,20){${}_{5}$}

\put(10,0){\circle*{2}}
\put(5,0){${}_{6}$}

\put(45,-15){\circle*{2}}
\put(45,-15){\vector(1,1){14}}
\put(47,-16){\tiny$m_0$}

\put(68,31){$^{v_{3}^1}$}
\put(80,30){$^{v_{4}^1}$}
\put(90,20){$^{v_{5}^1}$}

\put(90,-2){$_{v_{6}^1}$}

\put(80,-12){$_{v_{7}^1}$}

\put(68,-14){$_{v_{m_1}^1}$}

\qbezier[25](10,0)(20,-25)(45,-15)

\put(10,-30){Figure 6.2. $(Q^1,W^1)$}
\end{picture}}

\end{picture}
\vspace{1cm}
\end{center}

Denote the QP of $A$ by $(Q,W)$.

If $m_1=3$, let $e_{v_3^1}$ be the idempotent corresponding to $v_3^1$, and $V=A/A e_{v_3^1}A$. Similar to the proof of \cite[Lemma 4.4]{CGL} by using Proposition \ref{proposition homological epimorphism induces singularity equivalences}, we can prove that $D^b_{sg}(A)\simeq D^b_{sg}(V)$. It is easy to see that the quiver $Q'$ of $V$ is obtained from $Q$ by removing the vertex $v_3^1$ and its adjacent two arrows. Note that $V$ is a polygon-tree algebra with $N$ gluing components,
and $d_{Q'}=d_Q$, we get that
$$D^b_{sg}(V)\simeq \underline{\mod}(\cn_{\sum_{i\neq1} m_i-3(N-1)+d_{Q'}})\simeq \underline{\mod}(\cn_{ m-3N+d_{Q}})$$
by the assumption of induction in the proof of Theorem \ref{proposition to type Q(m,{i_1,...,i_r})} and $m_1=3$.

If $m_1>4$, it is routine to check that the vertex $v_3^1$ of $Q$ satisfies Lemma \ref{proposition good mutation} (a), which implies that the negative mutation of $A$ at $v_3^1$ is defined. Doing mutation of QP at $v_3^1$, we get that $\mu_{v_3^1}(Q,W)=(Q^1,W^1)$, where $Q^1$ is as Figure 6.2 shows, and $W^1=\sum_{w\in\cs(Q^1)}w$. Denote by $B_1$ the Jacobian algebra of $(Q^1,W^1)$, which is also a polygon-tree algebra. It is routine to check that the vertex $v_3^1$ of $Q^1$ satisfies Lemma \ref{proposition good mutation} (b), which implies that the positive mutation of $B_1$ at $v_3^1$ is defined. Note that $Q,Q^1$ have no loops or oriented $2$-cycles, and both of them are polygon-tree algebras. Proposition \ref{schurian of polygon-tree algebras} and Proposition \ref{proposition mutations for schurian algebras} (a) yield that $A$ is derived equivalent, and then singularity equivalent to $B_1$.

$Q^2$ is constructed from $Q^1$ by one-pointed coextension by adding a vertex $c_1$, and $W^2=\sum_{w\in\cs(Q^2)}w$. Denote by $B_2$ the Jacobian algebra of $(Q^2,W^2)$. Then Proposition \ref{proposition homological epimorphism induces singularity equivalences} implies that $D_{sg}^b(B_1)\simeq D_{sg}^b(B_2)$, see also \cite[Theorem 4.1]{Chen1}.
Easily, the vertex $v_4^1$ of $Q^2$ satisfies Proposition \ref{proposition old good mutation} (b), which implies that the positive mutation of $B_2$ at $v_4^1$ is defined.
Doing mutation at $v_4^1$, we get that $\mu_{v_4^1}(Q^2,W^2)=(Q^3,W^3)$, where $Q^3$ is as Figure 6.4 shows, and $W^3=\sum_{w\in\cs(Q^2)}w$. Denote by $B_3$ the Jacobian algebra of $(Q^3,W^3)$.

\setlength{\unitlength}{0.5mm}
\begin{center}
\begin{picture}(200,70)(0,-10)

\put(-20,0){\begin{picture}(80,70)
\put(60,0){\circle*{2}}
\put(60,0){\vector(0,1){19}}
\put(59,-4){${}_1$}
\put(54,1){${}_{i_1}$}

\put(60,20){\circle*{2}}

\put(60,20){\vector(-1,1){14}}
\put(59,23){${}_{2}$}
\put(47,17){${}_{i_1+1}$}

\put(70,30){\vector(-1,-1){9}}
\put(70,30){\circle*{2}}
\put(80,30){\vector(-1,0){9}}
\put(60,20){\vector(2,1){19}}
\put(80,30){\circle*{2}}
\put(80,30){\vector(1,-1){9}}
\put(90,20){\circle*{2}}

\put(90,20){\vector(0,-1){19}}
\put(90,0){\circle*{2}}
\put(90,0){\vector(-1,-1){9}}
\put(80,-10){\circle*{2}}
\put(70,-10){\vector(-1,1){9}}
\put(70,-10){\circle*{2}}
\qbezier[8](80,-10)(75,-10)(70,-10)
\put(45,35){\circle*{2}}
\put(45,35){\vector(-1,0){19}}
\put(46,37){${}_{3}$}
\put(43,31){${}_{i_2}$}

\put(25,35){\circle*{2}}

\put(25,35){\vector(-1,-1){14}}
\put(21,37){${}_{4}$}
\put(23,31){${}_{i_2+1}$}

\put(25,35){\vector(0,1){19}}

\put(25,55){\circle*{2}}

\qbezier[14](25,55)(35,55)(45,55)

\put(45,55){\circle*{2}}
\put(45,55){\vector(0,-1){19}}

\put(10,20){\circle*{2}}
\put(10,20){\vector(0,-1){19}}

\put(5,20){${}_{5}$}

\put(10,0){\circle*{2}}
\put(5,0){${}_{6}$}

\put(45,-15){\circle*{2}}
\put(45,-15){\vector(1,1){14}}
\put(47,-16){\tiny$m_0$}

\put(80,44){$^{c_1}$}

\put(70,30){$^{v_{3}^1}$}
\put(80,29){$^{v_{4}^1}$}

\put(80,30){\vector(0,1){14}}
\put(80,45){\circle*{2}}
\put(90,20){$^{v_{5}^1}$}

\put(90,-2){$_{v_{6}^1}$}

\put(80,-12){$_{v_{7}^1}$}

\put(68,-14){$_{v_{m_1}^1}$}

\qbezier[25](10,0)(20,-25)(45,-15)

\put(10,-30){Figure 6.3. $(Q^2,W^2)$}
\end{picture}}

\put(100,0){\begin{picture}(80,70)
\put(60,0){\circle*{2}}
\put(60,0){\vector(0,1){19}}
\put(59,-4){${}_1$}
\put(54,1){${}_{i_1}$}

\put(60,20){\circle*{2}}

\put(60,20){\vector(-1,1){14}}
\put(59,23){${}_{2}$}
\put(47,17){${}_{i_1+1}$}

\put(60,20){\vector(1,1){9}}
\put(70,30){\circle*{2}}
\put(70,30){\vector(1,0){9}}
\put(80,30){\vector(-2,-1){19}}
\put(80,30){\circle*{2}}
\put(90,20){\vector(-1,1){9}}
\put(90,20){\circle*{2}}
\put(60,20){\vector(1,0){29}}
\put(90,20){\vector(0,-1){19}}
\put(90,0){\circle*{2}}
\put(90,0){\vector(-1,-1){9}}
\put(80,-10){\circle*{2}}
\put(70,-10){\vector(-1,1){9}}
\put(70,-10){\circle*{2}}
\qbezier[8](80,-10)(75,-10)(70,-10)
\put(45,35){\circle*{2}}
\put(45,35){\vector(-1,0){19}}
\put(47,37){${}_{3}$}
\put(43,31){${}_{i_2}$}

\put(25,35){\circle*{2}}

\put(25,35){\vector(-1,-1){14}}
\put(21,37){${}_{4}$}
\put(23,31){${}_{i_2+1}$}

\put(25,35){\vector(0,1){19}}

\put(25,55){\circle*{2}}

\qbezier[14](25,55)(35,55)(45,55)

\put(45,55){\circle*{2}}
\put(45,55){\vector(0,-1){19}}

\put(10,20){\circle*{2}}
\put(10,20){\vector(0,-1){19}}

\put(5,20){${}_{5}$}

\put(10,0){\circle*{2}}
\put(5,0){${}_{6}$}

\put(45,-15){\circle*{2}}
\put(45,-15){\vector(1,1){14}}
\put(47,-16){\tiny$m_0$}

\put(70,30){$^{c_1}$}
\put(81,30){$^{v_{4}^1}$}

\put(80,45){\vector(0,-1){14}}
\put(80,45){\circle*{2}}
\put(90,20){$^{v_{5}^1}$}

\put(90,-2){$_{v_{6}^1}$}

\put(80,-12){$_{v_{7}^1}$}
\put(80,45){$^{v_3^1}$}
\put(68,-14){$_{v_{m_1}^1}$}

\qbezier[25](10,0)(20,-25)(45,-15)

\put(10,-30){Figure 6.4. $(Q^3,W^3)$}
\end{picture}}

\end{picture}
\vspace{1cm}
\end{center}

It is routine to check that $v_4^1$ of $Q^3$ satisfies Proposition \ref{proposition old good mutation} (a), which implies that the negative mutation of $B_3$ at $v_4^1$ is defined. Note that $Q^2,Q^3$ have no loops or oriented $2$-cycles. It is easy to see that $B_2$ and $B_3$ have the property that the simple $B_2$-module $_{B_2}S_i$ and the simple $B_3$-module $_{B_3}S_i$ are not submodules of the radicals of the indecomposable projective modules $_{B_2}P_i$ and $_{B_3}P_i$ for any vertex $i$ in $Q^2$ and $Q^3$ respectively. Thus Proposition \ref{proposition mutations for schurian algebras} (b) implies that $B_2$ is derived equivalent, and then singularity equivalent to $B_3$.

\setlength{\unitlength}{0.5mm}
\begin{center}
\begin{picture}(200,70)(0,-10)

\put(-20,0){\begin{picture}(80,70)

\put(60,0){\circle*{2}}
\put(60,0){\vector(0,1){19}}
\put(59,-4){${}_1$}
\put(54,1){${}_{i_1}$}

\put(60,20){\circle*{2}}

\put(60,20){\vector(-1,1){14}}
\put(59,23){${}_{2}$}
\put(47,17){${}_{i_1+1}$}

\put(60,20){\vector(1,1){9}}
\put(70,30){\circle*{2}}
\put(70,30){\vector(1,0){9}}
\put(80,30){\vector(-2,-1){19}}
\put(80,30){\circle*{2}}
\put(90,20){\vector(-1,1){9}}
\put(90,20){\circle*{2}}
\put(60,20){\vector(1,0){29}}
\put(90,20){\vector(0,-1){19}}
\put(90,0){\circle*{2}}
\put(90,0){\vector(-1,-1){9}}
\put(80,-10){\circle*{2}}
\put(70,-10){\vector(-1,1){9}}
\put(70,-10){\circle*{2}}
\qbezier[8](80,-10)(75,-10)(70,-10)
\put(45,35){\circle*{2}}
\put(45,35){\vector(-1,0){19}}
\put(47,37){${}_{3}$}
\put(43,31){${}_{i_2}$}

\put(25,35){\circle*{2}}

\put(25,35){\vector(-1,-1){14}}
\put(21,37){${}_{4}$}
\put(23,31){${}_{i_2+1}$}

\put(25,35){\vector(0,1){19}}

\put(25,55){\circle*{2}}

\qbezier[14](25,55)(35,55)(45,55)

\put(45,55){\circle*{2}}
\put(45,55){\vector(0,-1){19}}

\put(10,20){\circle*{2}}
\put(10,20){\vector(0,-1){19}}

\put(5,20){${}_{5}$}

\put(10,0){\circle*{2}}
\put(5,0){${}_{6}$}

\put(45,-15){\circle*{2}}
\put(45,-15){\vector(1,1){14}}
\put(47,-16){\tiny$m_0$}

\put(70,30){$^{c_1}$}
\put(80,30){$^{v_{4}^1}$}

\put(90,20){$^{v_{5}^1}$}

\put(90,-2){$_{v_{6}^1}$}

\put(80,-12){$_{v_{7}^1}$}

\put(68,-14){$_{v_{m_1}^1}$}

\qbezier[25](10,0)(20,-25)(45,-15)

\put(10,-30){Figure 6.5. $(Q^4,W^4)$}
\end{picture}}
\put(100,0){\begin{picture}(80,70)

\put(60,0){\circle*{2}}
\put(60,0){\vector(0,1){19}}
\put(59,-4){${}_1$}
\put(54,1){${}_{i_1}$}

\put(60,20){\circle*{2}}

\put(60,20){\vector(-1,1){14}}
\put(59,23){${}_{2}$}
\put(47,17){${}_{i_1+1}$}

\put(60,20){\vector(1,1){9}}
\put(70,30){\circle*{2}}
\put(70,30){\vector(1,0){9}}

\put(80,30){\circle*{2}}
\put(80,30){\vector(1,-1){9}}
\put(90,20){\circle*{2}}
\put(90,20){\vector(-1,0){29}}
\put(90,0){\vector(0,1){19}}
\put(60,20){\vector(3,-2){29}}
\put(90,0){\circle*{2}}
\put(90,0){\vector(-1,-1){9}}
\put(80,-10){\circle*{2}}
\put(70,-10){\vector(-1,1){9}}
\put(70,-10){\circle*{2}}
\qbezier[8](80,-10)(75,-10)(70,-10)
\put(45,35){\circle*{2}}
\put(45,35){\vector(-1,0){19}}
\put(47,37){${}_{3}$}
\put(43,31){${}_{i_2}$}

\put(25,35){\circle*{2}}

\put(25,35){\vector(-1,-1){14}}
\put(21,37){${}_{4}$}
\put(23,31){${}_{i_2+1}$}

\put(25,35){\vector(0,1){19}}

\put(25,55){\circle*{2}}

\qbezier[14](25,55)(35,55)(45,55)

\put(45,55){\circle*{2}}
\put(45,55){\vector(0,-1){19}}

\put(10,20){\circle*{2}}
\put(10,20){\vector(0,-1){19}}

\put(5,20){${}_{5}$}

\put(10,0){\circle*{2}}
\put(5,0){${}_{6}$}

\put(45,-15){\circle*{2}}
\put(45,-15){\vector(1,1){14}}
\put(47,-16){\tiny$m_0$}

\put(70,30){$^{c_1}$}
\put(80,30){$^{v_{4}^1}$}

\put(90,20){$^{v_{5}^1}$}

\put(90,-2){$_{v_{6}^1}$}

\put(80,-12){$_{v_{7}^1}$}

\put(68,-14){$_{v_{m_1}^1}$}

\qbezier[25](10,0)(20,-25)(45,-15)

\put(10,-30){Figure 6.6. $(Q^5,W^5)$}
\end{picture}}

\end{picture}
\vspace{1cm}
\end{center}

Let $Q^4$ be as Figure 6.5 shows, and $W^4= \sum_{w\in\cs(Q^4)}w$. Denote by $B_4$ the Jacobian algebra of $(Q^4,W^4)$. Then Proposition \ref{proposition homological epimorphism induces singularity equivalences} shows that
$D^b_{sg}(B_3)\simeq D^b_{sg}(B_4)$, see also \cite[Theorem 4.1]{Chen1}. Lemma \ref{proposition good mutation} (b) implies that the positive mutation of $B_4$ at $v_5^1$ is defined. Do mutation at $v_5^1$, and denote $\mu_{v_5^1}(Q^4,W^4)=(Q^5,W^5)$, where $Q^5$ is as Figure 6.6 shows, and $W^5=\sum_{w\in\cs(Q^5)}w$. Denote by $B_5$ the Jacobian algebra of $(Q^5,W^5)$. Similar to the above, we get that $B_4$ is derived equivalent and then singularity equivalent to $B_5$.

For $(Q^5,W^5)$, it is easy to see that the positive mutation at $v_6^1$ is defined, we take mutation of QP at $v_6^1$, and denote the resulting QP by $(Q^6,W^6)$, and its Jacobian algebra by $B_6$. Then take mutation at $v_7^1$ and inductively, after taking mutation at $v_{m_1-1}^1$, we get the Jacobian algebra $B_{m_1-1}$ of $(Q^{m_1-1},W^{m_1-1})$ with $Q^{m_1-1}$ as Figure 6.7 shows, and $W^{m_1-1}=\sum_{w\in\cs(Q^{m_1-1})}w$. So $B_6$ is derived equivalent and then singularity equivalent to $B_{m_1-1}$.

\setlength{\unitlength}{0.5mm}
\begin{center}
\begin{picture}(200,70)(0,-10)

\put(-20,0){\begin{picture}(80,70)

\put(60,0){\circle*{2}}
\put(60,0){\vector(0,1){19}}
\put(59,-4){${}_1$}
\put(54,1){${}_{i_1}$}

\put(60,20){\circle*{2}}

\put(60,20){\vector(-1,1){14}}
\put(59,23){${}_{2}$}
\put(47,17){${}_{i_1+1}$}

\put(60,20){\vector(1,1){9}}
\put(70,30){\circle*{2}}
\put(70,30){\vector(1,0){9}}

\put(80,30){\circle*{2}}
\put(80,30){\vector(1,-1){9}}
\put(90,20){\circle*{2}}

\put(90,20){\vector(0,-1){19}}
\put(60,20){\vector(1,-3){9.4}}
\put(90,0){\circle*{2}}

\put(80,-10){\circle*{2}}
\put(70,-10){\vector(-1,1){9}}
\put(70,-10){\circle*{2}}
\qbezier[8](80,-10)(85,-5)(90,0)
\put(45,35){\circle*{2}}
\put(45,35){\vector(-1,0){19}}
\put(47,37){${}_{3}$}
\put(43,31){${}_{i_2}$}

\put(25,35){\circle*{2}}

\put(25,35){\vector(-1,-1){14}}
\put(21,37){${}_{4}$}
\put(23,31){${}_{i_2+1}$}

\put(25,35){\vector(0,1){19}}

\put(25,55){\circle*{2}}

\qbezier[14](25,55)(35,55)(45,55)

\put(45,55){\circle*{2}}
\put(45,55){\vector(0,-1){19}}

\put(10,20){\circle*{2}}
\put(10,20){\vector(0,-1){19}}

\put(5,20){${}_{5}$}

\put(10,0){\circle*{2}}
\put(5,0){${}_{6}$}

\put(45,-15){\circle*{2}}
\put(45,-15){\vector(1,1){14}}
\put(47,-16){\tiny$m_0$}

\put(70,30){$^{c_1}$}
\put(80,30){$^{v_{4}^1}$}

\put(90,20){$^{v_{5}^1}$}

\put(90,-2){$_{v_{6}^1}$}
\put(70,-10){\vector(1,0){9}}
\put(80,-10){\vector(-2,3){19.2}}
\put(80,-12){$_{v_{m_1-1}^1}$}

\put(68,-14){$_{v_{m_1}^1}$}

\qbezier[25](10,0)(20,-25)(45,-15)

\put(10,-30){Figure 6.7. $(Q^{m_1-1},W^{m_1-1})$}
\end{picture}}

\put(100,0){\begin{picture}(80,70)
\put(60,0){\circle*{2}}
\put(60,0){\vector(0,1){19}}
\put(60,-1){${}^{v_{m_1}^1}$}
\put(54,1){${}_{i_1}$}

\put(60,20){\circle*{2}}

\put(60,20){\vector(-1,1){14}}
\put(59,23){${}_{2}$}
\put(47,17){${}_{i_1+1}$}

\put(60,20){\vector(1,1){9}}
\put(70,30){\circle*{2}}
\put(70,30){\vector(1,0){9}}
\put(80,30){\circle*{2}}
\put(80,30){\vector(1,-1){9}}
\put(90,20){\circle*{2}}

\put(90,20){\vector(0,-1){19}}
\put(90,0){\circle*{2}}
\put(90,0){\vector(-1,-1){9}}
\put(80,-10){\circle*{2}}
\put(70,-10){\vector(-1,1){9}}
\put(70,-10){\circle*{2}}
\qbezier[8](80,-10)(75,-10)(70,-10)
\put(45,35){\circle*{2}}
\put(45,35){\vector(-1,0){19}}
\put(47,37){${}_{3}$}
\put(43,31){${}_{i_2}$}

\put(25,35){\circle*{2}}

\put(25,35){\vector(-1,-1){14}}
\put(21,37){${}_{4}$}
\put(23,31){${}_{i_2+1}$}

\put(25,35){\vector(0,1){19}}

\put(25,55){\circle*{2}}

\qbezier[14](25,55)(35,55)(45,55)

\put(45,55){\circle*{2}}
\put(45,55){\vector(0,-1){19}}

\put(10,20){\circle*{2}}
\put(10,20){\vector(0,-1){19}}

\put(5,20){${}_{5}$}

\put(10,0){\circle*{2}}
\put(5,0){${}_{6}$}

\put(45,-15){\circle*{2}}
\put(45,-15){\vector(1,1){14}}
\put(22,-18){$_{m_0}$}

\put(70,30){$^{c_1}$}
\put(80,30){$^{v_{4}^1}$}
\put(90,20){$^{v_{5}^1}$}

\put(90,-2){$_{v_{6}^1}$}

\put(80,-12){$_{v_{7}^1}$}

\put(62,-12){$_{v_{m_1-1}^1}$}
\put(45,-19){$_1$}

\put(25,-15){\circle*{2}}
\put(25,-15){\vector(1,0){19}}
\qbezier[15](10,0)(20,-10)(25,-15)
\put(10,-30){Figure 6.8. $(Q^{m_1},W^{m_1})$}

\put(30,10){$Q_0'$}
\put(70,10){$Q_1'$}

\end{picture}}

\end{picture}
\vspace{1cm}
\end{center}

Then the positive mutation of $B_{m_1-1}$ at $v_{m_1}^1$ is defined. Take mutation at $v_{m_1}^1$, let $(Q^{m_1},W^{m_1})=\mu_{v_{m_1}^1}(Q^{m_1-1},W^{m_1-1})$, where $Q^{m_1}$ is as Figure 6.8 shows, and $W^{m_1}=\sum_{w\in\cs(Q^{m_1})}w$. Denote by $B_{m_1}$ the Jacobian algebra associated to $(Q^{m_1},W^{m_1})$. Similarly, $B_{m_1-1}$ is derived equivalent and then singularity equivalent to $B_{m_1}$. For convenience, we also denote the vertex $c_1$ in $Q^{m_1}$ by $v_3^1$.
$B_{m_1}$ is a polygon-tree algebra, denote by $Q_0'$, $Q_1'$ the oriented chordless cycles as in Figure 6.8. Then $Q_0'$ has $m_0+1$ vertices and $Q_1'$ has $m_1-1$ vertices. We do the same mutations at vertices in $Q_1'$ for $(Q^{m_1},W^{m_1})$ as  above, and then obtain a QP $(Q^{l},W^l)$, where $Q^l$ is as Figure 6.9 shows, $W^l= \sum_{w\in\cs(Q^l)}w$. Denote by $B_l$ its Jacobian algebra. Then $B_{m_1}$ is singularity equivalent to $B_l$.

\setlength{\unitlength}{0.5mm}
\begin{center}
\begin{picture}(200,70)(0,-10)

\put(-20,0){\begin{picture}(80,70)
\put(60,0){\circle*{2}}
\put(60,0){\vector(0,1){19}}
\put(58,-4){${}_{v_{5}^1}$}
\put(54,1){${}_{i_1}$}

\put(60,20){\circle*{2}}

\put(60,20){\vector(-1,1){14}}
\put(60,23){${}_{2}$}
\put(47,17){${}_{i_1+1}$}

\put(60,20){\vector(1,0){19}}

\put(80,20){\circle*{2}}
\put(80,20){\vector(0,-1){19}}

\put(80,0){\vector(-1,0){19}}
\put(80,0){\circle*{2}}

\put(45,35){\circle*{2}}
\put(45,35){\vector(-1,0){19}}
\put(47,37){${}_{3}$}
\put(43,31){${}_{i_2}$}

\put(25,35){\circle*{2}}

\put(25,35){\vector(-1,-1){14}}
\put(21,37){${}_{4}$}
\put(23,31){${}_{i_2+1}$}

\put(25,35){\vector(0,1){19}}

\put(25,55){\circle*{2}}

\put(45,55){\circle*{2}}
\put(45,55){\vector(0,-1){19}}

\put(10,20){\circle*{2}}
\put(10,20){\vector(0,-1){19}}

\put(5,20){${}_{5}$}

\put(10,0){\circle*{2}}
\put(5,0){${}_{6}$}

\qbezier[14](25,55)(35,55)(45,55)

\put(45,-15){\circle*{2}}
\put(45,-15){\vector(1,1){14}}
\put(23,-18){$_{v_7^1}$}

\put(80,20){$^{v_3^1}$}
\put(81,-1){$_{v_4^1}$}

\put(46,-19){\tiny$v_6^1$}

\put(25,-15){\circle*{2}}
\put(25,-15){\vector(1,0){19}}
\qbezier[15](10,0)(20,-10)(25,-15)

\put(10,-30){Figure 6.9. $(Q^{l},W^{l})$}

\end{picture}}

\put(100,0){\begin{picture}(80,70)
\put(60,0){\circle*{2}}
\put(60,0){\vector(0,1){19}}
\put(58,-4){${}_{v_{5}^1}$}
\put(54,1){${}_{i_1}$}

\put(60,20){\circle*{2}}
\put(60,20){\vector(1,-1){19}}
\put(60,20){\vector(-1,1){14}}
\put(60,23){${}_{2}$}
\put(47,17){${}_{i_1+1}$}

\put(80,20){\vector(-1,0){19}}

\put(80,20){\circle*{2}}
\put(80,0){\vector(0,1){19}}

\put(80,0){\vector(-1,0){19}}
\put(80,0){\circle*{2}}

\put(45,35){\circle*{2}}
\put(45,35){\vector(-1,0){19}}
\put(47,37){${}_{3}$}
\put(43,31){${}_{i_2}$}

\put(25,35){\circle*{2}}

\put(25,35){\vector(-1,-1){14}}
\put(21,37){${}_{4}$}
\put(23,31){${}_{i_2+1}$}

\put(25,35){\vector(0,1){19}}

\put(25,55){\circle*{2}}

\put(45,55){\circle*{2}}
\put(45,55){\vector(0,-1){19}}

\put(10,20){\circle*{2}}
\put(10,20){\vector(0,-1){19}}

\put(5,20){${}_{5}$}

\put(10,0){\circle*{2}}
\put(5,0){${}_{6}$}

\put(45,-15){\circle*{2}}
\put(45,-15){\vector(1,1){14}}
\put(22,-19){$_{v_7^1}$}

\qbezier[14](25,55)(35,55)(45,55)

\put(80,20){$^{v_3^1}$}
\put(80,-2){$_{v_4^1}$}

\put(46,-19){\tiny$v_6^1$}

\put(25,-15){\circle*{2}}
\put(25,-15){\vector(1,0){19}}
\qbezier[15](10,0)(20,-10)(25,-15)
\put(10,-30){Figure 6.10. $(Q^{l+1},W^{l+1})$}

\end{picture}}
\end{picture}
\vspace{1cm}
\end{center}

For $(Q^l,W^l)$, it is easy to see that the negative mutation at $v_3^1$ is defined. Take mutation at $v_3^1$, denote the resulting QP by $(Q^{l+1},W^{l+1})$, and its Jacobian algebra by $B_{l+1}$. Then $Q^{l+1}$ is as Figure 6.10 shows, and $W^{l+1}= \sum_{w\in\cs(Q^{l+1})}w$. Similar to the above, we get that $D^b_{sg}(B_{l})\simeq D^b_{sg}(B_{l+1})$.

Let $B_{l+2}$ be the one-point coextension algebra of $B_{l+1}$, which is also a Jacobian algebra with its QP $(Q^{l+2},W^{l+2})$, where $Q^{l+2}$ is as Figure 6.11 shows, and $W^{l+2}= \sum_{w\in\cs(Q^{l+2})}w$. Then Proposition \ref{proposition homological epimorphism induces singularity equivalences} implies that $D_{sg}^b(B_{l+1})\simeq D_{sg}^b(B_{l+2})$.

It is easy to see that the negative mutation at $v_4^1$ is defined, we take mutation at $v_4^1$, and denote the resulting algebra by $B_{l+3}$, which is also a Jacobian algebra with its QP $(Q^{l+3},W^{l+3})$. Then $D_{sg}^b(B_{l+2})\simeq D_{sg}^b(B_{l+3})$.

\setlength{\unitlength}{0.5mm}
\begin{center}
\begin{picture}(200,70)(0,-10)

\put(-20,0){\begin{picture}(80,70)
\put(60,0){\circle*{2}}
\put(60,0){\vector(0,1){19}}
\put(58,-4){${}_{v_{5}^1}$}
\put(54,1){${}_{i_1}$}

\put(60,20){\circle*{2}}
\put(60,20){\vector(1,-1){19}}
\put(60,20){\vector(-1,1){14}}
\put(60,23){${}_{2}$}
\put(47,17){${}_{i_1+1}$}

\put(80,20){\vector(-1,0){19}}

\put(80,20){\circle*{2}}
\put(80,0){\vector(0,1){19}}

\put(80,0){\vector(-1,0){19}}
\put(80,0){\circle*{2}}

\put(45,35){\circle*{2}}
\put(45,35){\vector(-1,0){19}}
\put(47,37){${}_{3}$}
\put(43,31){${}_{i_2}$}

\put(25,35){\circle*{2}}

\put(25,35){\vector(-1,-1){14}}
\put(21,37){${}_{4}$}
\put(23,31){${}_{i_2+1}$}

\put(25,35){\vector(0,1){19}}

\put(25,55){\circle*{2}}

\put(45,55){\circle*{2}}
\put(45,55){\vector(0,-1){19}}

\put(10,20){\circle*{2}}
\put(10,20){\vector(0,-1){19}}

\put(5,20){${}_{5}$}

\put(10,0){\circle*{2}}
\put(5,0){${}_{6}$}

\put(45,-15){\circle*{2}}
\put(45,-15){\vector(1,1){14}}
\put(22,-19){$_{v_7^1}$}

\qbezier[14](25,55)(35,55)(45,55)

\put(80,20){$^{v_3^1}$}
\put(80,-2){$_{v_4^1}$}

\put(46,-19){\tiny$v_6^1$}

\put(25,-15){\circle*{2}}
\put(25,-15){\vector(1,0){19}}
\qbezier[15](10,0)(20,-10)(25,-15)

\put(80,0){\vector(0,-1){19}}
\put(80,-20){\circle*{2}}
\put(80,-22){$_{c_2}$}

\put(10,-34){Figure 6.11. $(Q^{l+2},W^{l+2})$}

\end{picture}}

\put(100,0){\begin{picture}(80,70)
\put(60,0){\circle*{2}}
\put(60,0){\vector(0,1){19}}
\put(60,0){${}^{v_{4}^1}$}
\put(54,1){${}_{i_1}$}

\put(60,20){\circle*{2}}
\put(60,20){\vector(1,-1){19}}
\put(60,20){\vector(-1,1){14}}
\put(60,23){${}_{2}$}
\put(47,17){${}_{i_1+1}$}

\put(80,0){\vector(-1,0){19}}
\put(80,0){\circle*{2}}

\put(45,35){\circle*{2}}
\put(45,35){\vector(-1,0){19}}
\put(47,37){${}_{3}$}
\put(43,31){${}_{i_2}$}

\put(25,35){\circle*{2}}

\put(25,35){\vector(-1,-1){14}}
\put(21,37){${}_{4}$}
\put(23,31){${}_{i_2+1}$}

\put(25,35){\vector(0,1){19}}

\put(25,55){\circle*{2}}

\put(45,55){\circle*{2}}
\put(45,55){\vector(0,-1){19}}

\put(10,20){\circle*{2}}
\put(10,20){\vector(0,-1){19}}

\put(5,20){${}_{5}$}

\put(10,0){\circle*{2}}
\put(5,0){${}_{6}$}

\put(45,-15){\circle*{2}}
\put(45,-15){\vector(1,1){14}}
\put(23,-18){$_{v_6^1}$}

\put(80,-2){$_{c_2}$}

\put(70,-22){$_{v_3^1}$}
\put(70,-20){\circle*{2}}

\put(70,-20){\vector(-1,2){9.2}}
\put(43,-20){\tiny$v_5^1$}
\qbezier[14](25,55)(35,55)(45,55)
\put(25,-15){\circle*{2}}
\put(25,-15){\vector(1,0){19}}
\qbezier[15](10,0)(20,-10)(25,-15)
\put(10,-34){Figure 6.12. $(Q^{l+3},W^{l+3})$}

\end{picture}}
\end{picture}
\vspace{1cm}
\end{center}

For $B_{l+3}$, using Proposition \ref{proposition homological epimorphism induces singularity equivalences} twice, we can get that $D_{sg}^b(B_{l+3})\simeq D_{sg}^b(B_{l+4})$, where $B_{l+4}$ is a simple polygon-tree algebra with its quiver $Q^{l+4}$ as Figure 6.13 shows. Note that $B_{l+4}$ has $N$ gluing components $Q_0'',Q_2,\dots,Q_N$, and the oriented chordless cycle $Q_0''$ has $m_0+m_1-3$ vertices, the other gluing components are the same as in $Q$. It is easy to see that $d_{Q^{l+4}}=d_Q$, so we get that
$$D^b_{sg}(A)\simeq D^b_{sg}(B_{l+4})\simeq\underline{\mod}(\cn_{m_0+m_1-3+\sum_{i=2}^N m_i-3(N-1)+d_{Q'}})\simeq \underline{\mod}(\cn_{ m-3N+d_{Q}})$$
by the inductive assumption in the proof of Theorem \ref{proposition to type Q(m,{i_1,...,i_r})}.

\setlength{\unitlength}{0.5mm}
\begin{center}
\begin{picture}(60,70)(0,-10)

\put(0,0){\begin{picture}(80,70)
\put(60,0){\circle*{2}}
\put(60,0){\vector(0,1){19}}
\put(61,-2){${}^{v_{4}^1}$}

\qbezier[14](25,55)(35,55)(45,55)

\put(60,20){\circle*{2}}

\put(60,20){\vector(-1,1){14}}
\put(60,23){${}_{2}$}

\put(45,35){\circle*{2}}
\put(45,35){\vector(-1,0){19}}
\put(47,37){${}_{3}$}
\put(43,31){${}_{i_2}$}

\put(25,35){\circle*{2}}

\put(25,35){\vector(-1,-1){14}}
\put(21,37){${}_{4}$}
\put(23,31){${}_{i_2+1}$}

\put(25,35){\vector(0,1){19}}

\put(25,55){\circle*{2}}

\put(45,55){\circle*{2}}
\put(45,55){\vector(0,-1){19}}

\put(10,20){\circle*{2}}
\put(10,20){\vector(0,-1){19}}

\put(5,20){${}_{5}$}

\put(10,0){\circle*{2}}
\put(5,0){${}_{6}$}

\put(45,-15){\circle*{2}}
\put(45,-15){\vector(1,1){14}}
\put(22,-20){\tiny${v_6^1}$}

\put(30,10){$Q_0''$}

\put(30,42){$Q_2$}
\put(45,-20){\tiny$v_5^1$}

\put(25,-15){\circle*{2}}
\put(25,-15){\vector(1,0){19}}
\qbezier[15](10,0)(20,-10)(25,-15)
\put(0,-30){Figure 6.13. $(Q^{l+4},W^{l+4})$}

\end{picture}}
\end{picture}
\vspace{1cm}
\end{center}

If $m_1=4$, then $Q$ looks like the quiver as Figure 6.9 shows, after doing the same operations, we can get our desired result.
\end{proof}

\begin{lemma}\label{the other lemma singularity category of cluster-tilted algebra of D}
Keep the notation as in the proof of Theorem \ref{proposition to type Q(m,{i_1,...,i_r})}. If $Q(0)$ satisfies $d_1=i_2-i_1=1$ and $d_n=i_1+m_0-i_n>1$, then
$$D^b_{sg}(A)\simeq \underline{\mod }(\cn_{m-3N+d_Q}).$$
\end{lemma}

\begin{proof} Without loss of generality, we assume that the quiver $Q$
is as Figure 7.1 shows.

\setlength{\unitlength}{0.5mm}
\begin{center}
\begin{picture}(180,70)(0,-10)

\put(-20,0){\begin{picture}(80,70)
\put(60,0){\circle*{2}}
\put(60,0){\vector(0,1){19}}
\put(59,-4){${}_{1}$}
\put(55,1){${}_{i_1}$}

\put(60,20){\circle*{2}}

\put(60,20){\vector(-1,1){14}}
\put(58,23){${}_{2}$}
\put(47,17){${}_{i_1+1}$}
\put(53,12){\tiny$\|$}
\put(52,8){$_{i_2}$}
\put(60,20){\vector(1,0){19}}

\put(80,20){\circle*{2}}

\qbezier[14](80,20)(80,10)(80,0)

\put(80,0){\vector(-1,0){19}}
\put(80,0){\circle*{2}}

\put(45,35){\circle*{2}}
\put(45,35){\vector(-1,0){19}}
\put(43,38){${}_{3}$}
\put(35,31){${}_{i_2+1}$}

\put(25,35){\circle*{2}}

\put(25,35){\vector(-1,-1){14}}
\put(22,38){${}_{4}$}

\put(60,50){$^{v_3^2}$}

\put(60,50){\circle*{2}}
\put(45,35){\vector(1,1){14}}

\put(75,35){\circle*{2}}
\put(75,35){\vector(-1,-1){14}}

\put(77,35){$_{v_{m_2}^2}$}

\put(10,20){\circle*{2}}
\put(10,20){\vector(0,-1){19}}

\put(5,20){${}_{5}$}

\qbezier[14](75,35)(67.5,42.5)(60,50)

\put(10,0){\circle*{2}}
\put(5,0){${}_{6}$}

\put(45,-15){\circle*{2}}
\put(45,-15){\vector(1,1){14}}
\put(22,-19){$_{v_{m_0-1}^1}$}

\put(80,20){$^{v_3^1}$}
\put(80,-2){$_{v_{m_1}^1}$}

\put(43,-20){\tiny$v_{m_0}^1$}

\put(25,-15){\circle*{2}}
\put(25,-15){\vector(1,0){19}}
\qbezier[15](10,0)(20,-10)(25,-15)
\put(10,-34){Figure 7.1. $(Q,W)$}

\put(30,10){$Q_0$}
\put(65,10){$Q_1$}
\put(55,35){$Q_2$}

\end{picture}}

\put(100,0){\begin{picture}(80,70)
\put(60,0){\circle*{2}}
\put(60,0){\vector(0,1){19}}
\put(58,-3){${}_{v_4^1}$}
\put(54,1){${}_{i_1}$}

\put(60,20){\circle*{2}}

\put(60,20){\vector(-1,1){14}}
\put(59,23){${}_{2}$}
\put(47,17){${}_{i_1+1}$}
\put(53,12){\tiny$\|$}
\put(52,8){$_{i_2}$}

\put(80,0){\vector(-1,0){19}}
\put(80,0){\circle*{2}}

\put(45,35){\circle*{2}}
\put(45,35){\vector(-1,0){19}}
\put(43,38){${}_{3}$}
\put(35,31){${}_{i_2+1}$}

\put(25,35){\circle*{2}}

\put(25,35){\vector(-1,-1){14}}
\put(22,38){${}_{4}$}

\put(60,50){\circle*{2}}
\put(45,35){\vector(1,1){14}}

\put(75,35){\circle*{2}}
\put(75,35){\vector(-1,-1){14}}

\put(10,20){\circle*{2}}
\put(10,20){\vector(0,-1){19}}

\put(5,20){${}_{5}$}

\qbezier[14](75,35)(67.5,42.5)(60,50)

\put(10,0){\circle*{2}}
\put(5,0){${}_{6}$}

\put(45,-15){\circle*{2}}
\put(45,-15){\vector(1,1){14}}
\put(23,-18){$_{v_{6}^1}$}

\put(60,20){\vector(1,-1){19}}
\put(80,-2){$_{c_1}$}

\put(44,-20){\tiny$v_{5}^1$}

\put(25,-15){\circle*{2}}
\put(25,-15){\vector(1,0){19}}
\qbezier[15](10,0)(20,-10)(25,-15)
\put(10,-34){Figure 7.2. $(Q^1,W^1)$}

\end{picture}}
\end{picture}
\vspace{1.05cm}
\end{center}

If $m_1>3$, doing the same mutations as in Lemma \ref{lemma singularity category of cluster-tilted algebra of D}, we can get that
$A$ is singularity equivalent to $B_1$, where $B_1$ is a simple polygon-tree algebra with its QP $(Q^1,W^1)$ as Figure 7.2 shows.
For $(Q^1,W^1)$, the negative mutation of $B_1$ at $c_1$ is defined, so we do this mutation and get a polygon-tree algebra $B_2$ with its QP $(Q^2,W^2)$ as Figure 7.3 shows.
The positive mutation of $B_2$ at $c_1$ is defined, so Proposition \ref{proposition mutations for schurian algebras} implies that $B_1$ is derived equivalent to $B_2$.
\setlength{\unitlength}{0.5mm}
\begin{center}
\begin{picture}(80,60)(0,-10)

\put(0,0){\begin{picture}(80,60)
\put(60,0){\circle*{2}}
\put(60,0){\vector(0,1){19}}
\put(60,-3){${}_{c_1}$}
\put(54,1){${}_{i_1}$}

\put(60,20){\circle*{2}}

\put(60,20){\vector(-1,1){14}}
\put(59,23){${}_{2}$}
\put(47,17){${}_{i_1+1}$}
\put(53,12){\tiny$\|$}
\put(52,8){$_{i_2}$}

\put(45,35){\circle*{2}}
\put(45,35){\vector(-1,0){19}}
\put(42,38){${}_{3}$}
\put(35,31){${}_{i_2+1}$}

\put(25,35){\circle*{2}}

\put(25,35){\vector(-1,-1){14}}
\put(21,38){${}_{4}$}

\put(60,50){\circle*{2}}
\put(45,35){\vector(1,1){14}}

\put(75,35){\circle*{2}}
\put(75,35){\vector(-1,-1){14}}

\put(10,20){\circle*{2}}
\put(10,20){\vector(0,-1){19}}

\put(5,20){${}_{5}$}

\qbezier[14](75,35)(67.5,42.5)(60,50)

\put(10,0){\circle*{2}}
\put(5,0){${}_{6}$}

\put(45,-15){\circle*{2}}
\put(45,-15){\vector(1,1){14}}
\put(23,-18){$_{v_{5}^1}$}

\put(44,-20){\tiny$v_{4}^1$}

\put(25,-15){\circle*{2}}
\put(25,-15){\vector(1,0){19}}
\qbezier[15](10,0)(20,-10)(25,-15)

\put(30,10){$Q_0'$}
\put(10,-34){Figure 7.3. $(Q^2,W^2)$}

\put(55,35){$Q_2$}

\end{picture}}
\end{picture}
\vspace{1cm}
\end{center}
Note that $B_2$ is a simple polygon-tree algebra with $N$ gluing components $Q_0',Q_2,\dots,Q_N$, the oriented chordless cycle $Q_0'$ of $Q^2$ has $m_0+m_1-2$ vertices and $d_{Q^2}=d_Q-1$.
So we get that
$$D^b_{sg}(A)\simeq D^b_{sg}(B_{2})\simeq\underline{\mod}(\cn_{m_0+m_1-2+\sum_{i=2}^N m_i-3(N-1)+d_{Q^2}})\simeq \underline{\mod}(\cn_{ m-3N+d_{Q}})$$
by the assumption of induction in the proof of Theorem \ref{proposition to type Q(m,{i_1,...,i_r})}.

If $m_3=3$, then the quiver $Q$ looks like $Q^1$ as Figure 7.2 shows, so we can get the desired result by doing similar mutations.
\end{proof}

\begin{lemma}\label{the another lemma singularity category of cluster-tilted algebra of D}
Keep the notation as in the proof of Theorem \ref{proposition to type Q(m,{i_1,...,i_r})}. If $Q(0)$ satisfies $d_1=i_2-i_1>1$ and $d_n=i_1+m_0-i_n=1$, then
$$D^b_{sg}(A)\simeq \underline{\mod }(\cn_{m-3N+d_Q}).$$
\end{lemma}
\begin{proof}
This case is dual to the above case, by doing the dual mutations to those in the proof of Lemma \ref{the other lemma singularity category of cluster-tilted algebra of D}, or considering the opposite algebras.
\end{proof}

\begin{lemma}\label{the final lemma singularity category of cluster-tilted algebra of D 2}
Keep the notation as in the proof of Theorem \ref{proposition to type Q(m,{i_1,...,i_r})}. If $Q(0)$ satisfies $d_1=i_2-i_1=1$ and $d_n=i_1+m_0-i_n=1$, then
$$D^b_{sg}(A)\simeq \underline{\mod }(\cn_{m-3N+d_Q}).$$
\end{lemma}
\begin{proof}
We assume that the quiver $Q$ is as Figure 8.1 shows. $Q^1$ is the one-pointed extension of $Q$ by adding a vertex $c_1$ as Figure 8.2 shows, and
$W^{1}=\sum_{w\in\cs(Q^{1})}w$. Denote by $B_1$ the Jacobian algebra associated to $(Q^1,W^1)$.

\setlength{\unitlength}{0.5mm}
\begin{center}
\begin{picture}(200,90)(0,-10)

\put(-20,0){\begin{picture}(80,90)
\put(60,0){\circle*{2}}
\put(60,0){\vector(0,1){19}}
\put(60,-3){${}_{v_{m_1-1}^1}$}

\put(60,20){\circle*{2}}

\put(60,20){\vector(-1,1){14}}
\put(60,23){$_{v_{m_1}^1}$}

\put(45,35){\circle*{2}}
\put(45,35){\vector(-1,0){19}}
\put(46,37){${}_{1}$}
\put(42,31){${}_{i_1}$}

\put(25,35){\circle*{2}}

\put(25,31){$_{i_2}$}

\put(25,35){\vector(-1,-1){14}}
\put(22,37){${}_{2}$}

\put(10,20){\circle*{2}}
\put(10,20){\vector(0,-1){19}}

\put(3,20){${}_{v_3^1}$}

\put(45,55){\circle*{2}}

\put(45,55){\vector(0,-1){19}}

\put(25,35){\vector(0,1){19}}

\put(35,65){\circle*{2}}

\put(35,65){\vector(1,-1){9}}

\put(65,55){\vector(-1,0){19}}
\put(65,55){\circle*{2}}

\put(65,35){\circle*{2}}

\qbezier[15](65,55)(65,45)(65,35)

\qbezier[10](35,65)(30,60)(25,55)
\put(45,35){\vector(1,0){19}}
\put(25,55){\circle*{2}}

\put(10,0){\circle*{2}}
\put(3,0){${}_{v_4^1}$}

\put(2,55){$^{v_{3}^2}$}

\put(1,31){$_{v_{m_2}^2}$}

\put(62,56){$^{v_{m_n}^n}$}

\put(64,31){$_{v_{3}^n}$}

\put(45,-15){\circle*{2}}
\put(45,-15){\vector(1,1){14}}
\put(24,-19){$_{v_{m_1-3}^1}$}

\put(5,35){\circle*{2}}
\put(5,35){\vector(1,0){19}}

\put(25,55){\vector(-1,0){19}}

\put(5,55){\circle*{2}}

\qbezier[15](5,55)(5,45)(5,35)

\put(45,-18){$_{v_{m_1-2}^1}$}

\put(25,-15){\circle*{2}}
\put(25,-15){\vector(1,0){19}}
\qbezier[15](10,0)(20,-10)(25,-15)

\put(27,55){$_{3}$}
\put(36,55){$_{m_0}$}
\put(10,-34){Figure 8.1. $(Q,W)$}

\put(30,10){$Q_1$}
\put(30,45){$Q_0$}
\put(10,45){$Q_2$}

\put(50,45){$Q_{n}$}
\end{picture}}

\put(100,0){\begin{picture}(80,90)
\put(60,0){\circle*{2}}
\put(60,0){\vector(0,1){19}}
\put(60,-3){${}_{v_{m_1-1}^1}$}

\put(60,20){\circle*{2}}

\put(60,20){\vector(-1,1){14}}
\put(59,26){$_{v_{m_1}^1}$}

\put(80,20){\circle*{2}}
\put(80,20){$^{c_1}$}

\put(80,20){\vector(-1,0){19}}

\put(45,35){\circle*{2}}
\put(45,35){\vector(-1,0){19}}
\put(46,37){${}_{1}$}
\put(42,31){${}_{i_1}$}

\put(25,35){\circle*{2}}

\put(25,31){$_{i_2}$}

\put(25,35){\vector(-1,-1){14}}
\put(22,37){${}_{2}$}

\put(10,20){\circle*{2}}
\put(10,20){\vector(0,-1){19}}

\put(3,20){${}_{v_3^1}$}

\put(45,55){\circle*{2}}

\put(45,55){\vector(0,-1){19}}

\put(25,35){\vector(0,1){19}}

\put(35,65){\circle*{2}}

\put(35,65){\vector(1,-1){9}}

\put(65,55){\vector(-1,0){19}}
\put(65,55){\circle*{2}}

\put(65,35){\circle*{2}}

\qbezier[15](65,55)(65,45)(65,35)

\qbezier[10](35,65)(30,60)(25,55)
\put(45,35){\vector(1,0){19}}
\put(25,55){\circle*{2}}

\put(10,0){\circle*{2}}
\put(3,0){${}_{v_4^1}$}

\put(45,-15){\circle*{2}}
\put(45,-15){\vector(1,1){14}}
\put(24,-19){$_{v_{m_1-3}^1}$}

\put(5,35){\circle*{2}}
\put(5,35){\vector(1,0){19}}

\put(25,55){\vector(-1,0){19}}

\put(5,55){\circle*{2}}

\qbezier[15](5,55)(5,45)(5,35)

\put(45,-18){$_{v_{m_1-2}^1}$}

\put(25,-15){\circle*{2}}
\put(25,-15){\vector(1,0){19}}
\qbezier[15](10,0)(20,-10)(25,-15)

\put(2,55){$^{v_{3}^2}$}

\put(1,31){$_{v_{m_2}^2}$}

\put(62,56){$^{v_{m_n}^n}$}

\put(64,31){$_{v_{3}^n}$}

\put(27,55){$_{3}$}
\put(36,55){$_{m_0}$}
\put(10,-34){Figure 8.2. $(Q^1,W^1)$}

\put(30,10){$Q_1$}
\put(30,45){$Q_0$}
\put(10,45){$Q_2$}

\put(50,45){$Q_{n}$}
\end{picture}}
\end{picture}
\vspace{1.3cm}
\end{center}

First, we consider the case $m_1>3$.

For $(Q^1,W^1)$, the negative mutation at $v_{m_1}^1$ is defined, and we take the mutation of QP at $v_{m_1}^1$, denoting by $(Q^2,W^2)$ the resulting QP, where $Q^2$ is as Figure 8.3 shows, and $W_2=\sum_{w\in\cs(Q^{2})}w$.
The Jacobian algebra associated to $(Q^2,W^2)$ is denoted by $B_2$. Then $B_2$ is a polygon-tree algebra and so it is a schurian algebra by Theorem \ref{lemma schurian algebra of simple polygon-tree algebras}.
For any nonzero path $p$ ending at $v_{m_1}^1$, if $p$ does not pass through the vertex $3$, then it is easy to see that the combination of $p$ and the arrow $v_{m_1}^1\rightarrow v_{m_1-1}^1$ is nonzero; otherwise, since $Q$ is a simple polygon-tree quiver, the arrow $3\rightarrow v_{3}^2$ is not a glued arrow, which implies that the combination of $p$ and the arrow $v_{m_1}^1\rightarrow v_{m_1-1}^1$ is also nonzero. We omit the concrete proof here.
Thus Proposition \ref{proposition mutations for schurian algebras} (a) shows that $B_1$ is derived equivalent to $B_2$.

The negative mutation of $B_2$ at $v_{m_1-1}^1$ is defined, and we take the mutation of QP at $v_{m_1-1}^1$, denoted by $(Q^3,W^3)$ the resulting QP, where $Q^3$ is as Figure 8.4 shows, and $W_3=\sum_{w\in\cs(Q^{3})}w$.
The Jacobian algebra associated to $(Q^3,W^3)$ is denoted by $B_3$. Easily, the positive mutation of $B_3$ at $v_{m_1-1}^1$ is defined, so Proposition \ref{proposition mutations for schurian algebras} (a) shows that $B_2$ is derived equivalent to $B_3$.

For $(Q^3,W^3)$, the negative mutation at $v_{m_1-2}^1$ is defined, we take the mutation of QP at this vertex, and denote the resulting QP by $(Q^4,W^4)$, and its Jacobian algebra by $B_4$. Similar to the above, we can get that $B_3$ is derived equivalent to $B_4$.
For $(Q^4,W^4)$, the negative mutation at $v_{m_1-3}^1$ is defined, we take the mutation and recursively, we get that $B_4$ is derived equivalent to the Jacobian algebra $B_l$ of $(Q^l,W^l)$, where
$Q^l$ is Figure 8.5 shows, and $W^l= \sum_{w\in\cs(Q^{l})}w$.

\setlength{\unitlength}{0.5mm}
\begin{center}
\begin{picture}(200,90)(0,-10)

\put(-20,0){\begin{picture}(80,90)
\put(60,0){\circle*{2}}
\put(60,20){\vector(0,-1){19}}
\put(60,-3){${}_{v_{m_1-1}^1}$}

\qbezier(60,0)(52.5,17.5)(46.5,31.5)

\put(46.5,31.5){\vector(-1,2){1}}

\put(60,20){\circle*{2}}

\put(45,35){\vector(1,-1){14}}
\put(60,17){$_{v_{m_1}^1}$}

\put(60,20){\vector(1,0){19}}
\put(80,20){\circle*{2}}
\put(80,20){$^{c_1}$}
\qbezier(80,20)(62.5,27.5)(48.5,33.5)
\put(48.5,33.5){\vector(-2,1){1}}
\put(45,35){\circle*{2}}
\put(45,35){\vector(-1,0){19}}
\put(46,37){${}_{1}$}
\put(42,31){${}_{i_1}$}

\put(25,35){\circle*{2}}

\put(25,31){$_{i_2}$}

\put(25,35){\vector(-1,-1){14}}
\put(22,37){${}_{2}$}

\put(10,20){\circle*{2}}
\put(10,20){\vector(0,-1){19}}

\put(3,20){${}_{v_3^1}$}

\put(45,55){\circle*{2}}

\put(45,55){\vector(0,-1){19}}

\put(25,35){\vector(0,1){19}}

\put(35,65){\circle*{2}}

\put(35,65){\vector(1,-1){9}}

\put(65,55){\vector(-1,0){19}}
\put(65,55){\circle*{2}}

\put(65,35){\circle*{2}}

\qbezier[15](65,55)(65,45)(65,35)

\qbezier[10](35,65)(30,60)(25,55)
\put(45,35){\vector(1,0){19}}
\put(25,55){\circle*{2}}

\put(10,0){\circle*{2}}
\put(3,0){${}_{v_4^1}$}

\put(45,-15){\circle*{2}}
\put(45,-15){\vector(1,1){14}}
\put(22,-19){$_{v_{m_1-3}^1}$}

\put(5,35){\circle*{2}}
\put(5,35){\vector(1,0){19}}

\put(25,55){\vector(-1,0){19}}

\put(5,55){\circle*{2}}

\qbezier[15](5,55)(5,45)(5,35)

\put(45,-18){$_{v_{m_1-2}^1}$}

\put(25,-15){\circle*{2}}
\put(25,-15){\vector(1,0){19}}
\qbezier[15](10,0)(20,-10)(25,-15)

\put(27,55){$_{3}$}
\put(36,55){$_{m_0}$}

\put(2,55){$^{v_{3}^2}$}

\put(1,31){$_{v_{m_2}^2}$}

\put(62,56){$^{v_{m_n}^n}$}

\put(64,31){$_{v_{3}^n}$}

\put(10,-34){Figure 8.3. $(Q^2,W^2)$}

\end{picture}}

\put(100,0){\begin{picture}(80,90)

\put(45,35){\circle*{2}}
\put(45,35){\vector(-1,0){19}}
\put(46,37){${}_{1}$}
\put(40,31){${}_{i_1}$}

\put(25,35){\circle*{2}}

\put(25,31){$_{i_2}$}

\put(25,35){\vector(-1,-1){14}}
\put(22,37){${}_{2}$}

\put(10,20){\circle*{2}}
\put(10,20){\vector(0,-1){19}}

\put(3,20){${}_{v_3^1}$}

\put(45,55){\circle*{2}}

\put(45,55){\vector(0,-1){19}}

\put(25,35){\vector(0,1){19}}

\put(35,65){\circle*{2}}

\put(35,65){\vector(1,-1){9}}

\put(65,55){\vector(-1,0){19}}
\put(65,55){\circle*{2}}

\put(65,35){\circle*{2}}

\qbezier[15](65,55)(65,45)(65,35)

\qbezier[10](35,65)(30,60)(25,55)
\put(45,35){\vector(1,0){19}}
\put(25,55){\circle*{2}}

\put(10,0){\circle*{2}}
\put(3,0){${}_{v_4^1}$}

\put(45,-15){\circle*{2}}

\put(22,-19){$_{v_{m_1-3}^1}$}

\put(5,35){\circle*{2}}
\put(5,35){\vector(1,0){19}}

\put(25,55){\vector(-1,0){19}}

\put(5,55){\circle*{2}}

\qbezier[15](5,55)(5,45)(5,35)

\put(45,-18){$_{v_{m_1-2}^1}$}

\put(25,-15){\circle*{2}}
\put(25,-15){\vector(1,0){19}}
\qbezier[15](10,0)(20,-10)(25,-15)

\put(27,55){$_{3}$}
\put(36,55){$_{m_0}$}

\put(2,55){$^{v_{3}^2}$}

\put(1,31){$_{v_{m_2}^2}$}

\put(62,56){$^{v_{m_n}^n}$}

\put(64,31){$_{v_{3}^n}$}

\put(45,-15){\vector(0,1){49}}
\put(70,-15){\vector(-1,0){24}}
\put(45,35){\vector(1,-2){24}}
\put(10,-34){Figure 8.4. $(Q^3,W^3)$}
\put(70,-15){\circle*{2}}

\put(70,-18){${}_{v_{m_1-1}^1}$}
\put(90,-18){${}_{v_{m_1}^1}$}
\put(70,-15){\vector(1,0){19}}
\put(90,-15){\circle*{2}}
\put(90,-15){\vector(0,1){26}}
\put(90,12.5){\circle*{2}}
\put(90,12.5){$^{c_1}$}
\put(90,12.5){\vector(-2,1){44}}
\end{picture}}
\end{picture}
\vspace{1cm}
\end{center}

For $(Q^l,W^l)$, the negative mutation at $v_{3}^1$ is defined, we take the mutation of QP at it, and denote the resulting QP by $(Q^{l+1},W^{l+1})$, where $Q^{l+1}$ is as Figure 8.6 shows, and $W^{l+1}=\sum_{w\in\cs(Q^{l+1})}w$. Its Jacobian algebra is denoted by $B_{l+1}$. Similar to the above, we can get that $B_{l}$ is derived equivalent to $B_{l+1}$.

\setlength{\unitlength}{0.5mm}
\begin{center}
\begin{picture}(200,90)(0,-10)

\put(-20,0){\begin{picture}(80,90)

\put(45,35){\circle*{2}}
\put(45,35){\vector(-1,0){19}}
\put(46,37){${}_{1}$}

\put(25,35){\circle*{2}}

\put(26,31){$_{i_2}$}

\put(21.5,37.5){${}_{2}$}

\put(45,55){\circle*{2}}

\put(45,55){\vector(0,-1){19}}

\put(25,35){\vector(0,1){19}}

\put(35,65){\circle*{2}}

\put(35,65){\vector(1,-1){9}}

\put(65,55){\vector(-1,0){19}}
\put(65,55){\circle*{2}}

\put(65,35){\circle*{2}}

\qbezier[15](65,55)(65,45)(65,35)

\qbezier[10](35,65)(30,60)(25,55)
\put(45,35){\vector(1,0){19}}
\put(25,55){\circle*{2}}

\qbezier[15](65,-5)(65,5)(65,15)

\put(65,15){$^{c_1}$}

\put(67,-5){$_{v_5^1}$}

\put(5,35){\circle*{2}}
\put(5,35){\vector(1,0){19}}

\put(25,55){\vector(-1,0){19}}

\put(5,55){\circle*{2}}

\qbezier[15](5,55)(5,45)(5,35)

\put(27,55){$_{3}$}
\put(36,55){$_{m_0}$}

\put(25,15){\circle*{2}}

\put(25,35){\vector(0,-1){19}}
\put(45,15){\circle*{2}}
\put(45,35){\vector(0,-1){19}}
\put(45,15){\vector(-1,0){19}}
\put(25,15){\vector(1,1){19}}

\put(58,26){$_{v_5^1}$}
\put(18,15){$_{v_3^1}$}

\put(47,15){$_{v_4^1}$}

\put(2,55){$^{v_{3}^2}$}

\put(1,31){$_{v_{m_2}^2}$}

\put(62,56){$^{v_{m_n}^n}$}

\put(64,31){$_{v_{3}^n}$}
\put(45,15){\vector(1,-1){19}}
\put(10,-38){Figure 8.5. $(Q^l,W^l)$}
\put(65,-5){\circle*{2}}
\put(65,15){\circle*{2}}
\put(65,15){\vector(-1,1){19}}

\end{picture}}

\put(100,0){\begin{picture}(80,90)
\put(60,0){\circle*{2}}
\put(60,0){\vector(0,1){19}}
\put(60,-3){${}_{v_{m_1}^1}$}

\put(60,20){\circle*{2}}

\put(60,20){\vector(-1,1){14}}
\put(60,23){$_{c_1}$}

\put(45,35){\circle*{2}}
\put(45,35){\vector(-1,0){19}}
\put(46,37){${}_{1}$}
\put(42,31){${}_{i_1}$}

\put(25,35){\circle*{2}}

\put(25,31){$_{v_3^1}$}

\put(25,35){\vector(-1,-1){14}}

\put(1,31){$_{v_{m_2}^2}$}

\put(62,56){$^{v_{m_n}^n}$}

\put(10,20){\circle*{2}}
\put(10,20){\vector(0,-1){19}}

\put(3,20){${}_{v_4^1}$}

\put(45,55){\circle*{2}}

\put(45,55){\vector(0,-1){19}}

\put(25,35){\vector(0,1){19}}

\put(35,65){\circle*{2}}

\put(65,55){\vector(-1,0){19}}
\put(65,55){\circle*{2}}

\put(65,35){\circle*{2}}

\put(25,55){\vector(1,1){9}}

\put(35,65){\vector(-1,1){14}}
\put(20,80){\circle*{2}}
\put(10,70){\circle*{2}}
\put(10,70){\vector(1,-1){14}}
\qbezier[10](10,70)(15,75)(20,80)

\qbezier[15](65,55)(65,45)(65,35)

\qbezier[10](35,65)(40,60)(45,55)
\put(45,35){\vector(1,0){19}}
\put(25,55){\circle*{2}}

\put(10,0){\circle*{2}}
\put(3,0){${}_{v_5^1}$}

\put(45,-15){\circle*{2}}
\put(45,-15){\vector(1,1){14}}
\put(24,-19){$_{v_{m_1-2}^1}$}

\put(45,-18){$_{v_{m_1-1}^1}$}

\put(25,-15){\circle*{2}}
\put(25,-15){\vector(1,0){19}}
\qbezier[15](10,0)(20,-10)(25,-15)

\put(27,54){$_{2}$}
\put(34,68){$_{3}$}

\put(18,54){$_{i_2}$}

\put(36,55){$_{m_0}$}
\put(10,-34){Figure 8.6. $(Q^{l+1},W^{l+1})$}

\put(30,10){$Q_1$}
\put(30,45){$Q_0'$}

\put(17,65){$Q_2$}

\put(50,45){$Q_{n}$}
\end{picture}}
\end{picture}
\vspace{1.5cm}
\end{center}

Note that $B_{l+1}$ is a simple polygon-tree algebra with $N+1$ gluing components $Q_0',Q_1,\dots,Q_N$, and $Q_0'$ has $m_0+1$ vertices. Note that $d_{Q^{l+1}}=d_Q-1$. So Lemma
\ref{the other lemma singularity category of cluster-tilted algebra of D} implies that
$$D^b_{sg}(A)\simeq D^b_{sg}(B_{l+1})\simeq\underline{\mod}\cn_{m_0+1+\sum_{i=1}^{N}m_i-3N+d_{Q^{l+1}}}=\underline{\mod}\cn_{m-3N+d_Q}.$$

For $m_1=3$, we also consider $(Q^1,W^1)$. The negative mutation of $B_1$ at $v_{3}^1$ is defined, and if we take the mutation of QP at $v_{3}^1$, then we already get a QP as Figure 8.6 shows. It is easy to get the desired result in this case.
\end{proof}

To describe the singularity category of a simple polygon-tree algebra, we need to recall a construction of Riedtmann \cite{Rie1}.
Let $\Z\A_n$ be the \emph{translation quiver} of a quiver of type $\A_n$. The translation of $\Z\A_n$ is denoted by $\tau$.
The \emph{mesh category} $K(\Z\A_n)$ is defined as the quotient category of the path category of $\Z\A_n$ by the \emph{mesh ideal}.
For any $r\in \N$, we denote by $K(\Z\A_n)/\tau^r$ the quotient category of $K(\Z\A_n)$ by the cyclic group $\tau^{r\Z}$.

\begin{corollary}\label{corollary singularity category}
Let $A=KQ/I$ be a simple polygon-tree algebra, where the gluing components of $Q$ are $Q_0$, $Q_1$, $\dots$, $Q_N$. Then
$D^b_{sg}(A)$ is triangle equivalent to $K(\Z\A_{m-3N+d_Q-2})/\tau^{m-3N+d_Q}$.
\end{corollary}
\begin{proof}
From \cite{Rie2}, we get that the stable category $\underline{\mod }(\cn_{m-3N+d_Q})$ is equivalent to $$K(\Z\A_{m-3N+d_Q-2})/\tau^{m-3N+d_Q}.$$
So the result follows from
Theorem \ref{proposition to type Q(m,{i_1,...,i_r})} immediately.
\end{proof}

\begin{example}
Let $A=KQ/I$ be the polygon-tree algebra with $Q$ as Figure 9.1 shows. It is not simple, and also a cluster-tilted algebra of type $\E_6$. From \cite{CGL}, we know that
$$D^b_{sg}(A)\simeq \underline{\mod}\cn_{4},$$
which does not satisfy the formula in Theorem \ref{proposition to type Q(m,{i_1,...,i_r})}.

\begin{center}\setlength{\unitlength}{0.7mm}
 \begin{picture}(60,30)
\put(10,10){\circle*{2}}
\put(10,10){\vector(1,1){9}}
\put(20,20){\circle*{2}}
\put(20,20){\vector(1,-1){9}}

\put(30,10){\circle*{2}}
\put(30,10){\vector(-1,0){19}}
\put(30,10){\vector(1,1){9}}
\put(40,20){\circle*{2}}
\put(40,20){\vector(-1,0){19}}
\put(40,20){\vector(1,-1){9}}
\put(50,10){\circle*{2}}
\put(50,10){\vector(-1,0){19}}
\put(50,10){\vector(1,1){9}}
\put(60,20){\circle*{2}}
\put(60,20){\vector(-1,0){19}}

\put(-20,0){Figure 9.1. A non-simple polygon-tree quiver.}
\end{picture}

\end{center}

\end{example}

We have the following direct corollaries.

\begin{corollary}
Let $A=KQ/I$ be a floriated algebra of $(Q_0,\{i_1,\dots,i_n\})$ by $Q_1,\dots,Q_n$. Then
$$D^b_{sg}(A)\simeq \underline{\mod }(\cn_{m-3n+d_Q}).$$
\end{corollary}

\begin{corollary}[\cite{CGL}]
Let $A=KQ/I$ be a floriated algebra of $(Q_0,\{i_1,\dots,i_n\})$ by $Q_1,\dots,Q_n$. If $A$ is a cluster-tilted algebra of type $\D$, then
$$D^b_{sg}(A)\simeq \underline{\mod }(\cn_{m_0+d_Q}).$$
\end{corollary}

Let $A=KQ/I$ be a quotient of the path algebra of the following quiver $Q$ modulo the ideal $I$ generated by the elements described below (in $Q$ the arm with arrows $x_i$ contains $p_i$ arrows):

\begin{center}\setlength{\unitlength}{0.5mm}
 \begin{picture}(110,70)
\put(0,40){\circle*{2}}
\put(2,42){\vector(1,1){16}}
\put(20,60){\circle*{2}}
\put(23,60){\vector(1,0){13}}
\put(40,60){\circle*{2}}
\put(43,60){\vector(1,0){13}}
\put(60,58.5){\large$\cdots$}
\put(70,60){\vector(1,0){13}}
\put(88,60){\circle*{2}}
\put(90,58){\vector(1,-1){16}}
\put(108,40){\circle*{2}}

\put(3,41){\vector(4,1){13}}
\put(20,45){\circle*{2}}
\put(23,45){\vector(1,0){13}}
\put(40,45){\circle*{2}}
\put(43,45){\vector(1,0){13}}
\put(60,43.5){\large$\cdots$}
\put(70,45){\vector(1,0){13}}
\put(88,45){\circle*{2}}
\put(91,44){\vector(4,-1){13}}
\qbezier(4,38)(50,25)(104,38)

\put(4,38){\vector(-3,1){1}}

\put(2,38){\vector(1,-1){16}}
\put(20,20){\circle*{2}}
\put(23,20){\vector(1,0){13}}
\put(40,20){\circle*{2}}
\put(43,20){\vector(1,0){13}}
\put(60,18.5){\large$\cdots$}
\put(70,20){\vector(1,0){13}}
\put(88,20){\circle*{2}}
\put(90,22){\vector(1,1){16}}
\put(50,33){$\eta$}
\put(3,50){$x_1$}
\put(24,62){$x_1$}
\put(44,62){$x_1$}
\put(96,53){$x_1$}

\put(6,39){$x_2$}
\put(24,41){$x_2$}
\put(44,41){$x_2$}
\put(93,39){$x_2$}

\put(3,29){$x_3$}

\put(24,22){$x_3$}
\put(44,22){$x_3$}
\put(96,25){$x_3$}

\put(10,8){$I: x_1^{p_1}+x_2^{p_2}+x_3^{p_3},x_i^{p_i-a}\eta x_i^{a-1}$}
\put(20,-1){$\mbox{ for } i=1,2,3$, $a=1,\dots,p_i$.}

\put(-70,-13){Figure 9.2. Cluster-tilted algebra of a canonical algebra with weights $(p_1,p_2,p_3)$.}
\end{picture}
\vspace{0.8cm}
\end{center}
From \cite{BKL}, we know that $A=KQ/I$ is a cluster-tilted algebra of a canonical algebra with weights $(p_1,p_2,p_3)$.

\begin{corollary}
Let $A=KQ/I$ be the cluster-tilted algebra of a canonical algebra with weights $(p_1,p_2,p_3)$ as above. If $p_1=2$, then $D^b_{sg}(A)\simeq \underline{\mod}\cn_{p_1+p_3-1}$.
\end{corollary}
\begin{proof}

Since $p_1=2$, the quiver $Q$ is as Figure 9.3 shows.
Let $e_{c_1}$ be the idempotent corresponding to the vertex $c_1$. Let $V$ be the quotient algebra of $A$ modulo the ideal generated by $e_{c_1}$, namely $Ae_{c_1}A$. The quiver of $V$ is obtained from $Q$ by removing the vertex $c_1$ and the adjacent arrows $\alpha,\beta$. It is easy to see that $V$ is a floriated algebra.

Let $B=(1-e_{c_1})A(1-e_{c_1})$. Then $k\alpha$ and $k\beta$ are naturally left and right $B$-modules respectively since $\beta\eta=0$ and $\eta\alpha=0$ in $B$.

\begin{center}\setlength{\unitlength}{0.5mm}
 \begin{picture}(110,70)
\put(0,40){\circle*{2}}
\put(2,42){\vector(1,1){16}}
\put(20,60){\circle*{2}}
\put(23,60){\vector(1,0){13}}
\put(40,60){\circle*{2}}
\put(43,60){\vector(1,0){13}}
\put(60,58.5){\large$\cdots$}
\put(70,60){\vector(1,0){13}}
\put(88,60){\circle*{2}}
\put(90,58){\vector(1,-1){16}}
\put(108,40){\circle*{2}}

\put(54,40){\circle*{2}}
\put(3,40){\vector(1,0){48}}
\put(57,40){\vector(1,0){48}}
\qbezier(4,38)(50,25)(104,38)

\put(4,38){\vector(-3,1){1}}

\put(2,38){\vector(1,-1){16}}
\put(20,20){\circle*{2}}
\put(23,20){\vector(1,0){13}}
\put(40,20){\circle*{2}}
\put(43,20){\vector(1,0){13}}
\put(60,18.5){\large$\cdots$}
\put(70,20){\vector(1,0){13}}
\put(88,20){\circle*{2}}
\put(90,22){\vector(1,1){16}}
\put(50,33){$\eta$}
\put(3,50){$x_2$}
\put(24,62){$x_2$}
\put(44,62){$x_2$}
\put(96,53){$x_2$}

\put(24,41){$\beta$}

\put(75,41){$\alpha$}

\put(3,29){$x_3$}

\put(24,22){$x_3$}
\put(44,22){$x_3$}
\put(96,25){$x_3$}
\put(53,42){$c_1$}
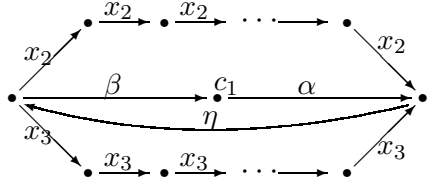
\put(-70,0){Figure 9.3. Cluster-tilted algebra of a canonical algebra with weights $(2,p_2,p_3)$.}
\end{picture}

\end{center}

We identify $A$ with $\left(\begin{array}{cc} B &k\alpha \\ k\beta &k  \end{array} \right)$, where the $k$ in the southeast corner is identified with $e_{c_1}A e_{c_1}$, and $B=(1-e_{c_1})A(1-e_{c_1})$. The morphism $\phi:k\alpha\otimes_k k\beta\rightarrow B$ maps $\alpha\otimes\beta$ to $\alpha\beta$, so $\phi$ is a monomorphism.
Furthermore, it is easy to see that $B/\Im\phi=V$.
Then Proposition \ref{proposition homological epimorphism induces singularity equivalences} yields a triangulated equivalence $D^b_{sg}(A)\simeq D^b_{sg}(V)$.
Using Theorem \ref{proposition to type Q(m,{i_1,...,i_r})}, we get that $D^b_{sg}(A)\simeq \underline{\mod}\cn_{p_1+p_3-1}$ immediately.
\end{proof}

\section{Representation type of polygon-tree algebras}

In this section, we study the representation type of polygon-tree algebras. We set $\D_3=\A_3$ for convenience.

We are more concerned with the number of arrows between the vertices rather than arrows themselves, we assume that the quivers have no loops or oriented $2$-cycles, but allow the arrows to be weighted by positive integers. If the weight of an arrow is $1$, we do not specify it in the picture and call it a single arrow; if an arrow has weight $2$ we call it a double arrow. For convenience, if an arrow $i\rightarrow j$ is weighted by a negative integer $r$, then it means that there is an arrow $j\rightarrow i$ with weight $-r$, and if an arrow $i\rightarrow j$ with weight $0$, then it means that there is no arrow between $i$ and $j$.
By a subquiver of $Q$, we always mean a quiver obtained from $Q$ by taking an induced (full) directed subgraph on a subset of vertices and keeping all its edge weights the same as in $Q$.
\begin{definition}[\cite{FZ1}]
Let $Q$ be a quiver without loops or oriented $2$-cycles. The FZ-mutation of $Q$ at a vertex $k$ (denoted by $\mu^{FZ}_k(Q)$) is a new quiver $Q^*$, described as follows:

1. Add a new vertex $k^*$.

2. If there is an arrow $i\rightarrow k$ with weight $r$, an arrow $k\rightarrow j$ with weight $s$ and an arrow $j\rightarrow i$ with weight $t$ in $Q$, then there is an arrow $j\rightarrow i$ with weight $t-rs$ in $Q^*$.

3. For any vertex $i$ replace an arrow from $i$ to $k$ with an arrow from $k^*$ to $i$, and replace an arrow from $k$ to $i$ with an arrow from $i$ to $k^*$, with the same weights.

4. Remove the vertex $k$.
\end{definition}

The FZ-mutation $\mu_k$ is involutive, so it defines a mutation-equivalence relation on quivers.
A quiver $Q$ is said to be of \emph{finite mutation type} if its mutation-equivalence class is finite. It is well known that, in a finite mutation type quiver with at least three vertices, any edge is a single edge or a double edge. The most basic examples of finite mutation type quivers are Dynkin quivers and extended Dynkin quivers \cite{BR}.

Another important class of finite mutation type quivers has been obtained in \cite{FST} using a construction that associates quivers to certain triangulations of surfaces. In this paper, we will not use this construction, so we do not recall it here. We call these quivers that \emph{come from the triangulation of a surface}.

\begin{theorem}[\cite{FeST}]\label{theorem quiver of finite mutation type}
A connected quiver $Q$ with at least three vertices is of finite mutation type if and only if it comes from the triangulation of a surface or it is mutation-equivalent to one of the exceptional types $\E_6,\E_7,\E_8,\tilde{\E}_6,\tilde{\E}_7,\tilde{\E}_8,\X_6,\X_7$ (see Figure 10.2 and Figure 10.3).
\end{theorem}

\begin{center}\setlength{\unitlength}{0.7mm}
 \begin{picture}(180,50)

\put(110,25){\begin{picture}(80,60)

\put(-15,0){${\E}_7$}
\put(0,0){\circle*{1}}

\put(0,0){\line(1,0){10}}
\put(10,0){\circle*{1}}

\put(10,0){\line(1,0){10}}
\put(20,0){\circle*{1}}

\put(20,0){\line(1,0){10}}
\put(30,0){\circle*{1}}

\put(30,0){\line(1,0){10}}
\put(40,0){\circle*{1}}

\put(40,0){\line(1,0){10}}
\put(50,0){\circle*{1}}

\put(20,10){\line(0,-1){10}}
\put(20,10){\circle*{1}}

\end{picture}}

\put(10,10){\begin{picture}(80,80)

\put(-15,0){${\E}_8$}
\put(0,0){\circle*{1}}

\put(0,0){\line(1,0){10}}
\put(10,0){\circle*{1}}

\put(10,0){\line(1,0){10}}
\put(20,0){\circle*{1}}

\put(20,0){\line(1,0){10}}
\put(30,0){\circle*{1}}

\put(30,0){\line(1,0){10}}
\put(40,0){\circle*{1}}

\put(40,0){\line(1,0){10}}
\put(50,0){\circle*{1}}

\put(50,0){\line(1,0){10}}
\put(60,0){\circle*{1}}

\put(20,10){\line(0,-1){10}}
\put(20,10){\circle*{1}}

\end{picture}}

\put(110,40){\begin{picture}(80,80)

\put(-15,0){${\D}_n$}
\put(0,0){\circle*{1}}

\put(0,0){\line(1,0){10}}
\put(10,0){\circle*{1}}

\put(10,0){\line(1,0){10}}
\put(20,0){\circle*{1}}

\put(30,0){\circle*{1}}

\put(22,-1){$\cdots$}
\put(40,0){\circle*{1}}

\put(30,0){\line(1,0){10}}
\put(40,0){\circle*{1}}

\put(40,0){\line(1,0){10}}
\put(50,0){\circle*{1}}

\put(10,10){\line(0,-1){10}}
\put(10,10){\circle*{1}}

\end{picture}}

\put(10,25){\begin{picture}(80,80)

\put(-15,0){${\E}_6$}
\put(0,0){\circle*{1}}

\put(0,0){\line(1,0){10}}
\put(10,0){\circle*{1}}

\put(10,0){\line(1,0){10}}
\put(20,0){\circle*{1}}

\put(20,0){\line(1,0){10}}
\put(30,0){\circle*{1}}

\put(30,0){\line(1,0){10}}
\put(40,0){\circle*{1}}

\put(20,10){\line(0,-1){10}}
\put(20,10){\circle*{1}}

\end{picture}}

\put(10,40){\begin{picture}(80,80)

\put(-15,0){$\A_n$}
\put(0,0){\circle*{1}}

\put(0,0){\line(1,0){10}}
\put(10,0){\circle*{1}}

\put(10,0){\line(1,0){10}}
\put(20,0){\circle*{1}}

\put(22,-1){$\cdots$}
\put(30,0){\circle*{1}}

\put(30,0){\line(1,0){10}}
\put(40,0){\circle*{1}}

\put(40,0){\line(1,0){10}}
\put(50,0){\circle*{1}}

\end{picture}}
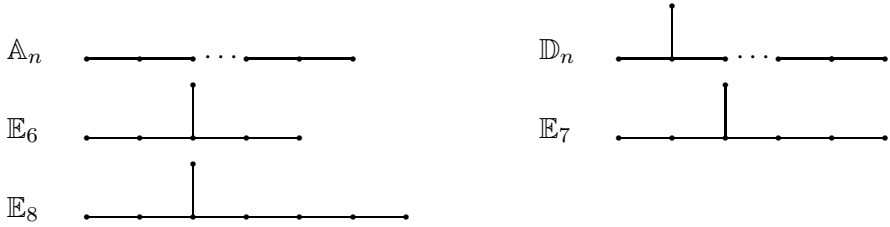
\put(10,0){Figure 10.1. Dynkin quivers: the edges may be oriented arbitrarily.}

\end{picture}
\vspace{0.5cm}
\end{center}

\begin{center}\setlength{\unitlength}{0.7mm}
 \begin{picture}(180,60)

\put(10,10){\begin{picture}(80,80)

\put(-15,0){$\tilde{\E}_7$}
\put(0,0){\circle*{1}}

\put(0,0){\line(1,0){10}}
\put(10,0){\circle*{1}}

\put(10,0){\line(1,0){10}}
\put(20,0){\circle*{1}}

\put(20,0){\line(1,0){10}}
\put(30,0){\circle*{1}}

\put(30,0){\line(1,0){10}}
\put(40,0){\circle*{1}}

\put(40,0){\line(1,0){10}}
\put(50,0){\circle*{1}}

\put(50,0){\line(1,0){10}}
\put(60,0){\circle*{1}}

\put(30,10){\line(0,-1){10}}
\put(30,10){\circle*{1}}

\end{picture}}

\put(110,10){\begin{picture}(80,80)

\put(-15,0){$\tilde{\E}_8$}
\put(0,0){\circle*{1}}

\put(0,0){\line(1,0){10}}
\put(10,0){\circle*{1}}

\put(10,0){\line(1,0){10}}
\put(20,0){\circle*{1}}

\put(20,0){\line(1,0){10}}
\put(30,0){\circle*{1}}

\put(30,0){\line(1,0){10}}
\put(40,0){\circle*{1}}

\put(40,0){\line(1,0){10}}
\put(50,0){\circle*{1}}

\put(50,0){\line(1,0){10}}
\put(60,0){\circle*{1}}
\put(60,0){\line(1,0){10}}
\put(70,0){\circle*{1}}
\put(20,10){\line(0,-1){10}}
\put(20,10){\circle*{1}}

\end{picture}}

\put(10,25){\begin{picture}(80,80)

\put(-15,0){$\tilde{\D}_n$}
\put(0,0){\circle*{1}}

\put(0,0){\line(1,0){10}}
\put(10,0){\circle*{1}}

\put(10,0){\line(1,0){10}}
\put(20,0){\circle*{1}}

\put(30,0){\circle*{1}}

\put(22,-1){$\cdots$}
\put(40,0){\circle*{1}}

\put(30,0){\line(1,0){10}}
\put(40,0){\circle*{1}}

\put(40,0){\line(1,0){10}}
\put(50,0){\circle*{1}}

\put(10,10){\line(0,-1){10}}
\put(10,10){\circle*{1}}

\put(40,10){\line(0,-1){10}}
\put(40,10){\circle*{1}}
\end{picture}}

\put(110,25){\begin{picture}(80,80)

\put(-15,0){$\tilde{\E}_6$}
\put(0,0){\circle*{1}}

\put(0,0){\line(1,0){10}}
\put(10,0){\circle*{1}}

\put(10,0){\line(1,0){10}}
\put(20,0){\circle*{1}}

\put(20,0){\line(1,0){10}}
\put(30,0){\circle*{1}}

\put(30,0){\line(1,0){10}}
\put(40,0){\circle*{1}}

\put(20,10){\line(0,-1){10}}
\put(20,10){\circle*{1}}
\put(20,20){\line(0,-1){10}}
\put(20,20){\circle*{1}}

\end{picture}}

\put(10,40){\begin{picture}(80,80)

\put(-15,0){$\tilde{\A}_n$}
\put(0,0){\circle*{1}}

\put(0,0){\line(1,0){10}}
\put(10,0){\circle*{1}}

\put(10,0){\line(1,0){10}}
\put(20,0){\circle*{1}}

\put(22,-1){$\cdots$}
\put(30,0){\circle*{1}}

\put(30,0){\line(1,0){10}}
\put(40,0){\circle*{1}}

\put(40,0){\line(1,0){10}}
\put(50,0){\circle*{1}}

\put(25,15){\line(-5,-3){25}}
\put(25,15){\circle*{1}}

\put(25,15){\line(5,-3){25}}
\end{picture}}
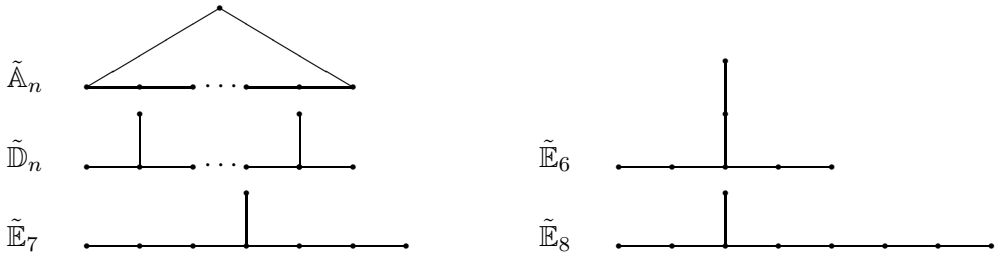
\put(10,0){Figure 10.2. Extended Dynkin quivers: the edges may be}
\put(10,-10){oriented arbitrarily such that the quiver is acyclic.}
\end{picture}
\vspace{0.8cm}
\end{center}

\begin{center}\setlength{\unitlength}{0.7mm}
 \begin{picture}(180,75)

\put(10,10){\begin{picture}(80,80)

\put(-15,0){$\X_6$}
\put(0,0){\circle*{1}}

\put(0,10){\circle*{1}}

\put(0,0){\vector(0,1){10}}

\put(10,0){\circle*{1}}
\put(0,10){\vector(1,-1){10}}

\put(10,-10){\circle*{1}}

\put(10,-10){\line(0,1){10}}

\put(20,0){\circle*{1}}
\put(20,0){\vector(0,1){10}}
\put(20,10){\circle*{1}}
\put(20,10){\vector(-1,-1){10}}

\put(10,0){\vector(1,0){10}}

\put(10,0){\vector(-1,0){10}}

\put(-3,3){$2$}
\put(21,3){$2$}
\end{picture}}

\put(110,10){\begin{picture}(80,80)

\put(-15,0){$\X_7$}
\put(0,0){\circle*{1}}

\put(0,10){\circle*{1}}

\put(0,0){\vector(0,1){10}}

\put(10,0){\circle*{1}}
\put(0,10){\vector(1,-1){10}}

\put(10,-10){\circle*{1}}

\put(10,0){\vector(0,-1){10}}

\put(20,0){\circle*{1}}
\put(20,0){\vector(0,1){10}}
\put(20,10){\circle*{1}}
\put(20,10){\vector(-1,-1){10}}

\put(10,0){\vector(1,0){10}}

\put(10,0){\vector(-1,0){10}}

\put(-3,3){$2$}
\put(21,3){$2$}

\put(20,-10){\circle*{1}}
\put(20,-10){\vector(-1,1){10}}
\put(10,-10){\vector(1,0){10}}

\put(13,-14){$2$}

\end{picture}}

\put(10,25){\begin{picture}(80,80)

\put(-15,10){$\E_8^{(1,1)}$}
\put(0,10){\circle*{1}}

\put(0,10){\line(1,0){10}}
\put(10,10){\circle*{1}}

\put(20,0){\circle*{1}}

\put(20,20){\circle*{1}}

\put(30,10){\circle*{1}}

\put(10,10){\vector(1,-1){10}}

\put(20,20){\vector(-1,-1){10}}
\put(20,20){\vector(1,-1){10}}

\put(30,10){\vector(-1,-1){10}}

\put(20,0){\vector(0,1){20}}

\put(50,10){\circle*{1}}

\put(20,20){\vector(3,-1){30}}
\put(50,10){\vector(-3,-1){30}}

\put(50,10){\line(1,0){10}}

\put(60,10){\circle*{1}}

\put(60,10){\line(1,0){10}}

\put(70,10){\circle*{1}}

\put(70,10){\line(1,0){10}}

\put(80,10){\circle*{1}}

\put(80,10){\line(1,0){10}}

\put(90,10){\circle*{1}}

\put(21,8){$2$}
\end{picture}}

\put(120,50){\begin{picture}(80,80)

\put(-25,10){$\E_7^{(1,1)}$}
\put(-10,10){\circle*{1}}
\put(-10,10){\line(1,0){10}}
\put(0,10){\circle*{1}}

\put(0,10){\line(1,0){10}}
\put(10,10){\circle*{1}}

\put(20,0){\circle*{1}}

\put(20,20){\circle*{1}}

\put(30,10){\circle*{1}}

\put(10,10){\vector(1,-1){10}}

\put(20,20){\vector(-1,-1){10}}
\put(20,20){\vector(1,-1){10}}

\put(30,10){\vector(-1,-1){10}}

\put(20,0){\vector(0,1){20}}

\put(50,10){\circle*{1}}

\put(20,20){\vector(3,-1){30}}
\put(50,10){\vector(-3,-1){30}}

\put(50,10){\line(1,0){10}}

\put(60,10){\circle*{1}}

\put(60,10){\line(1,0){10}}

\put(70,10){\circle*{1}}

\put(21,8){$2$}
\end{picture}}

\put(10,50){\begin{picture}(80,80)

\put(-15,10){$\E_6^{(1,1)}$}

\put(0,10){\circle*{1}}

\put(0,10){\line(1,0){10}}
\put(10,10){\circle*{1}}

\put(20,0){\circle*{1}}

\put(20,20){\circle*{1}}

\put(30,10){\circle*{1}}

\put(10,10){\vector(1,-1){10}}

\put(20,20){\vector(-1,-1){10}}
\put(20,20){\vector(1,-1){10}}

\put(30,10){\vector(-1,-1){10}}

\put(20,0){\vector(0,1){20}}

\put(40,10){\circle*{1}}

\put(30,10){\line(1,0){10}}

\put(50,10){\circle*{1}}

\put(20,20){\vector(3,-1){30}}
\put(50,10){\vector(-3,-1){30}}

\put(50,10){\line(1,0){10}}

\put(60,10){\circle*{1}}

\put(21,8){$2$}

\end{picture}}

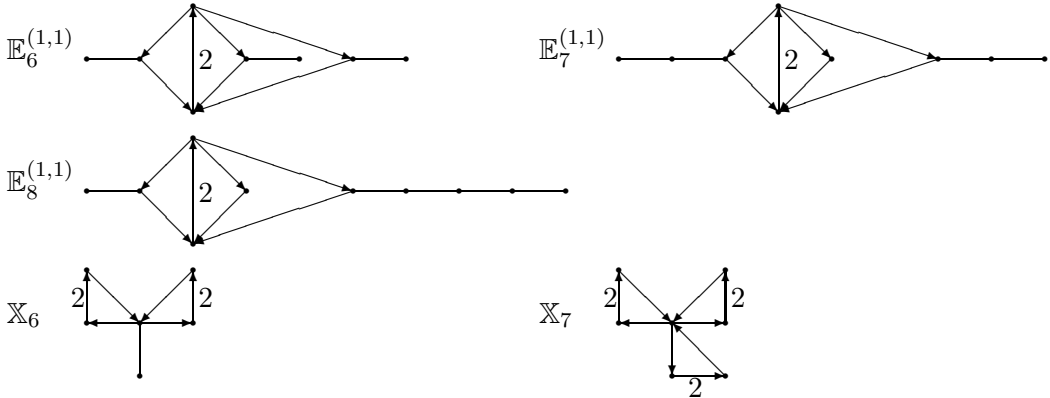
\put(10,-15){Figure 10.3. Exceptional quivers of finite mutation type which are acyclic:}
\put(10,-20){edges with unspecified orientation may be oriented arbitrarily.}

\end{picture}
\vspace{1.8cm}
\end{center}

\begin{center}\setlength{\unitlength}{0.7mm}
 \begin{picture}(180,25)
\put(20,15){\begin{picture}(80,25)

\put(-15,0){$\K_m$}

\put(0,0){\circle*{1.5}}

\put(20,0){\circle*{1.5}}

\put(0,0){\vector(1,0){19}}

\put(7,2){$m$}
\end{picture}}

\put(70,15){\begin{picture}(80,25)

\put(-15,0){$\T_6$}

\put(0,0){\circle*{1.5}}

\put(20,0){\circle*{1.5}}

\put(0,0){\vector(1,1){9}}

\put(10,10){\circle*{1.5}}
\put(20,0){\vector(-1,0){19}}

\put(10,10){\vector(1,-1){9}}

\put(2,5){$2$}
\put(15,5){$2$}
\put(9,-5){$2$}
\end{picture}}

\put(130,15){\begin{picture}(80,25)(0,10)

\put(-15,10){$\T_7$}

\put(0,10){\circle*{1.5}}

\put(20,10){\circle*{1.5}}

\put(10,0){\circle*{1.5}}

\put(10,20){\circle*{1.5}}
\put(0,10){\vector(1,1){9}}
\put(20,10){\vector(-1,1){9}}

\put(0,10){\vector(1,0){19}}

\put(10,20){\vector(0,-1){19}}
\put(20,10){\vector(-1,-1){9}}

\put(0,10){\vector(1,-1){9}}

\put(8.5,13){$2$}

\end{picture}}
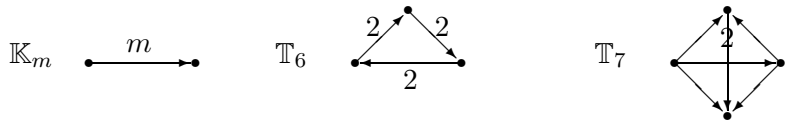
\put(0,-2){Figure 10.4. Quivers of finite mutation type with non-removable double arrows.}
\end{picture}
\vspace{0.2cm}
\end{center}

\begin{theorem}[\cite{S}]\label{theorem mutation finite}
Let $Q$ be a connected quiver of finite mutation type. Suppose also that $Q$ has a subquiver which is mutation-equivalent to $\E_6$ (resp. $\X_6$). Then any quiver which is mutation-equivalent to $Q$ also contains a subquiver which is mutation-equivalent to $\E_6$ (resp. $\X_6$). Furthermore $Q$ is mutation-equivalent to a quiver which is one of the types $\E$ (resp. $\X$) given in Theorem \ref{theorem quiver of finite mutation type}.
\end{theorem}

\begin{lemma}\label{lemma vertices more than 10 is of D}
Let $A=KQ/I$ be a floriated algebra of $(Q_0,\{i_1,\dots,i_n\})$ by $Q_1,\dots,Q_n$. If $Q$ is of finite mutation type, then we have one of the following statements:

(a) $A$ is a cluster-tilted algebra of type $\D$.

(b) $Q$ is mutation-equivalent to a quiver which is one of the types $\E$ given in Theorem \ref{theorem quiver of finite mutation type}.
\end{lemma}
\begin{proof}
Denote by $(Q,W)$ the QP satisfying $A=J(Q,W)$. Denote by $m_i$ the number of vertices in $Q_i$.

From Proposition \ref{proposition floriated algbra to one cycle}, we get that $(Q,W)$ is mutation-equivalent to $(Q^1,W^1)$ as Figure 2.5 shows. Recall that $Q^1$ has only one oriented cycle $c$ which has $m_0+n$ vertices. We consider the following cases.

Case (1) If $m_0+n=3$, then $m_0=3$ and $n=0$, which means that $Q^1$ is of type $\D_3=\A_3$.

Case (2) If $m_0+n=4$, then $m_0=4$ and $n=0$, or $m_0=3$ and $n=1$. It is easy to see that it is of type $\D$.

Case (3) For $m_0+n>4$, if $Q^1$ is not a quiver of type $\D$, then we get that $Q^1$ is not an oriented cycle. So the quiver in Figure 11.1 is a subquiver of $Q^1$, where the oriented cycle has $m_0+n$ vertices. If $m_0+n=5$, then the quiver in Figure 11.1 is in the mutation-equivalence class of $\E_6$, and the result follows from Theorem \ref{theorem mutation finite}.
If $m_0+n>5$, then the quiver in Figure 11.1 also admits a subquiver as Figure 11.2 shows, which is a quiver of type $\E_6$. So $Q^1$ admits a subquiver mutation-equivalent to a quiver of type $\E_6$, and then the result follows from Theorem \ref{theorem mutation finite}.

\setlength{\unitlength}{1mm}
\begin{center}
\begin{picture}(140,15)

    \put(10,0){\begin{picture}(20,20)
                     \put(10,0){\circle*{1}}
                     \put(20,0){\circle*{1}}
                     \put(20,10){\circle*{1}}
                     \put(10,10){\circle*{1}}
                     \put(10,0){\vector(1,0){10}}
                      \put(20,0){\vector(0,1){10}}

                        \put(10,10){\vector(0,-1){10}}
                        \put(15,15){\circle*{1}}
                      \qbezier[8](20,10)(17.5,12.5)(15,15)
                        \put(20,0){\vector(1,0){10}}
                        \put(30,0){\circle*{1}}
                        \put(15,15){\vector(-1,-1){5}}

\put(0,-10){Figure 11.1. Subquiver of $Q^1$.}
\end{picture}}

\put(70,0){\begin{picture}(60,20)

\put(0,0){\circle*{1}}

\put(0,0){\vector(1,0){9}}
\put(10,0){\circle*{1}}

\put(10,0){\vector(1,0){9}}
\put(20,0){\circle*{1}}

\put(20,0){\vector(1,0){9}}
\put(30,0){\circle*{1}}

\put(30,0){\vector(1,0){9}}
\put(40,0){\circle*{1}}

\put(20,0){\vector(0,1){9}}
\put(20,10){\circle*{1}}

\put(0,-10){Figure 11.2. A quiver of type $\E_6$.}
\end{picture}}
\end{picture}
\vspace{4mm}
\end{center}

\end{proof}

\begin{lemma}[\cite{GLS}]\label{lemma representation type of jacobian algebra}
Assume that $Q$ is not mutation equivalent to one of the quivers $\T_1,\T_2,\X_6,\X_7$ or
$\K_m$ with $m\geq3$. Then for any non-degenerate potential $S$ on $Q$ the following hold:

(a) $J(Q,S)$ is representation-finite if and only if $Q$ is of type $\A_n,\D_n(n\geq4),\E_6,\E_7,\E_8$.

(b) $J(Q,S)$ is tame if and only if $Q$ is of finite mutation type.

(c) $J(Q,S)$ is wild if and only if $Q$ is of infinite mutation type.
\end{lemma}

Now we can get our main result of this section.

\begin{theorem}\label{theorem of representation type}
Let $A=KQ/I$ be a polygon-tree algebra, where the gluing components of $Q$ are $Q_0,Q_1,\dots,Q_N$. Then

(a) $A$ is of representation finite type if and only if $A$ is in the mutation-equivalence class of type $\A_3$, $\D_n$($n\geq4$), $\E_6$, $\E_7$, or $\E_8$;

(b) $A$ is of tame representation type which is not representation finite if and only if $A$ is in the mutation-equivalence class of type $\tilde{\E}_6$, $\tilde{\E}_7$, $\tilde{\E}_8$, $\E_6^{(1,1)}$, $\E_7^{(1,1)}$ or $\E_8^{(1,1)}$. In particular, in this case, the number of vertices in $Q$ is no more than $10$.
\end{theorem}
\begin{proof}
Denote by $m_i$ the number of vertices in $Q_i$, and $m=\sum_{i=0}^Nm_i$. Then the number of vertices in $Q$ is $m-2N$.
Since the polygon-tree quiver $Q$ is not in the mutation-equivalence class of type $\A$ when $m-2N>3$, (a) follows from Lemma \ref{lemma representation type of jacobian algebra} directly.

For (b), if $A$ is of tame representation type which is not representation finite, then $Q$ is of finite mutation type, and any floriated subquiver $Q(i)$ ($0\leq i\leq N$) of $Q$ is also of finite mutation type. If $Q(i)$ admits a subquiver in the mutation class of $\E_6$, then Theorem \ref{theorem mutation finite} shows that $Q$ is mutation-equivalent to a quiver which is one of the types $\E$, together with that $A$ is not representation finite, we get that $A$ is in the mutation-equivalence class of $\tilde{\E}_6$, $\tilde{\E}_7$, $\tilde{\E}_8$, $\E_6^{(1,1)}$, $\E_7^{(1,1)}$ or $\E_8^{(1,1)}$.

If each $Q(i)$ does not admit a subquiver in the mutation class of $\E_6$, we get that $Q(i)$ is of type $\D$ for any $0\leq i\leq N$ by Lemma \ref{lemma vertices more than 10 is of D}.
If there exists $Q_{i}$ such that $m_i\geq4$ for some $i$, without loss of generality, we assume that $m_0\geq4$. We assume that the oriented chordless subquivers adjacent to $Q_0$ are
$Q_1,\dots,Q_n$. Note that $n\geq1$ since $Q$ is not of type $\D$. Then $m_j=3$ for $1\leq j\leq n$. Since $Q$ is not of type $\D$, there exists
$Q_k$ with $1\leq k\leq n$ such that it admits an adjacent oriented chordless subquiver different to $Q_0$, and then $Q(k)$ is not of type $\D$, which is a contradiction.
Thus $m_i=3$ for any $0\leq i\leq N$.

If each $Q_i$ admits only one adjacent oriented chordless subquiver, then $Q$ has only $2$ gluing components $Q_0,Q_1$. Since $m_i=3$ for $i=0,1$, it is easy to see that $Q$ is of type $\D$.
So there exists a gluing component admitting at least $2$ adjacent ones.

Without loss of generality, we assume that $Q_1,Q_2$ are adjacent gluing components of $Q_0$. Since $Q$ is not of type $\D$, then one of $Q_1$, $Q_2$ admits an adjacent oriented chordless subquiver different to $Q_0$. We assume that $Q_1$ admits an adjacent oriented chordless subquiver $Q_3$ different to $Q_0$. Then one of the two quivers in Figure 12 is a subquiver of $Q$. From \cite{GP}, we know that both of them are of type $\E_6$. So $Q$ is in the mutation-equivalence class of $\tilde{\E}_6$, $\tilde{\E}_7$, $\tilde{\E}_8$, $\E_6^{(1,1)}$, $\E_7^{(1,1)}$ or $\E_8^{(1,1)}$.
\setlength{\unitlength}{1mm}
\begin{center}
\begin{picture}(200,20)

\put(30,0){\begin{picture}(20,20)
\put(10,0){\circle*{1}}
\put(20,0){\circle*{1}}
\put(20,10){\circle*{1}}
\put(10,10){\circle*{1}}
\put(10,0){\vector(1,0){10}}
\put(20,0){\vector(0,1){10}}
\put(10,0){\vector(0,1){10}}
\put(10,10){\vector(1,0){10}}
\put(20,10){\vector(-1,-1){10}}
\put(20,10){\vector(1,-1){10}}
\put(30,0){\circle*{1}}
\put(30,0){\vector(-1,0){10}}

\put(10,10){\vector(-1,-1){10}}
\put(0,0){\circle*{1}}
\put(0,0){\vector(1,0){10}}

\put(12,7){\tiny$Q_1$}
\put(15,3){\tiny$Q_0$}
\put(22,3){\tiny$Q_2$}
\put(5,3){\tiny$Q_3$}
\end{picture}}

\put(80,0){\begin{picture}(20,20)
\put(10,0){\circle*{1}}
\put(20,0){\circle*{1}}
\put(20,10){\circle*{1}}
\put(10,10){\circle*{1}}
\put(10,0){\vector(1,0){10}}
\put(20,0){\vector(0,1){10}}
\put(10,0){\vector(0,1){10}}
\put(10,10){\vector(1,0){10}}
\put(20,10){\vector(-1,-1){10}}
\put(20,10){\vector(1,-1){10}}
\put(30,0){\circle*{1}}
\put(30,0){\vector(-1,0){10}}
\put(10,20){\circle*{1}}

\put(20,10){\vector(-1,1){10}}

\put(10,20){\vector(0,-1){10}}

\put(12,7){\tiny$Q_1$}
\put(15,3){\tiny$Q_0$}
\put(22,3){\tiny$Q_2$}
\put(12,12){\tiny$Q_3$}
\end{picture}}

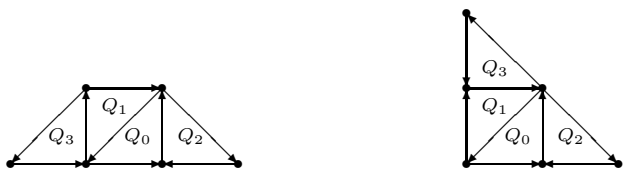
\put(30,-6){Figure 12. Quivers in the mutation-equivalence class of type $\E_6$.}

\end{picture}

\end{center}

\end{proof}

\begin{remark}
In fact, from \cite{GLS}, we know that $J(Q,S)$ is wild if $Q$ is mutation equivalent to one of the quivers $\X_6,\X_7$ and $\K_m$ for $m\geq3$, where $S$ is a non-degenerate potential on $Q$. In the tame case, the quivers $\T_i$ have no polygon-tree quivers in their mutation classes.
\end{remark}

For floriated quivers of type $\E_6^{(1,1)}, \E_7^{(1,1)},\E_8^{(1,1)}$, using Keller's applet for quiver mutations \cite{Ke2}, we can obtain that they are precisely as the following table shows. In particular, the floriated algebras of the following quivers are not cluster-tilted algebras.

\[
\setlength{\tabcolsep}{2pt}
\renewcommand{\arraystretch}{1.2}
\begin{tabular}{|l|l|}
\hline
\multicolumn{1}{|c|}{classes}  &\multicolumn{1}{|c|}{The floriated quivers}\\
\hline
\multicolumn{1}{|c|}{ \raisebox{4ex}[0pt]{$\E^{(1,1)}_7$}}  &\multicolumn{1}{|c|}{\setlength{\unitlength}{0.8mm}
                     \begin{picture}(140,25)
                     \put(0,12){\begin{picture}(20,20)
                     \put(0,5){\circle*{1}}
                     \put(5,0){\circle*{1}}
                     \put(5,10){\circle*{1}}
                     \put(10,5){\circle*{1}}
                      \put(15,0){\circle*{1}}
                          \put(15,10){\circle*{1}}
                           \put(20,5){\circle*{1}}
             \put(5,0){\vector(1,0){10}}
             \put(15,0){\vector(-1,1){5}}
             \put(10,5){\vector(-1,-1){5}}
             \put(5,0){\vector(-1,1){5}}
             \put(0,5){\vector(1,1){5}}
             \put(5,10){\vector(1,-1){5}}
               \put(10,5){\vector(1,1){5}}
             \put(15,10){\vector(1,-1){5}}
             \put(20,5){\vector(-1,-1){5}}

              \put(5,-10){\circle*{1}}
              \put(15,-10){\circle*{1}}
              \put(15,0){\vector(0,-1){10}}
              \put(15,-10){\vector(-1,0){10}}
             \put(5,-10){\vector(0,1){10}}
                        \end{picture}}

                     \end{picture}}\\
\hline

\multicolumn{1}{|c|}{ \raisebox{7ex}[0pt]{$\E^{(1,1)}_8$}}  &\multicolumn{1}{|c|}{\setlength{\unitlength}{0.8mm}
                     \begin{picture}(170,35)

  \put(0,10){\begin{picture}(20,20)
                 \put(10,0){\circle*{1}}
                     \put(20,0){\circle*{1}}
                     \put(20,10){\circle*{1}}
                     \put(10,10){\circle*{1}}
                     \put(10,0){\vector(1,0){10}}
                      \put(20,0){\vector(1,1){5}}
                      \put(25,5){\circle*{1}}
                       \put(25,5){\vector(-1,1){5}}

                      \put(10,10){\vector(-1,-1){5}}
                      \put(20,10){\vector(-1,0){10}}
                    \put(5,5){\circle*{1}}
                    \put(5,5){\vector(1,-1){5}}

                 \put(10,10){\vector(0,1){10}}
                  \put(10,20){\vector(1,0){10}}
                    \put(20,20){\vector(0,-1){10}}
                    \put(10,20){\circle*{1}}
                   \put(20,20){\circle*{1}}

           \put(10,-10){\vector(0,1){10}}
                  \put(20,-10){\vector(-1,0){10}}
                    \put(20,0){\vector(0,-1){10}}
                    \put(10,-10){\circle*{1}}
                   \put(20,-10){\circle*{1}}

                        \end{picture}}
        \put(35,10){\begin{picture}(20,20)
                 \put(10,0){\circle*{1}}
                     \put(20,0){\circle*{1}}
                     \put(20,10){\circle*{1}}
                     \put(10,10){\circle*{1}}
                     \put(10,0){\vector(1,0){10}}
                      \put(20,0){\vector(1,1){5}}
                      \put(25,5){\circle*{1}}
                       \put(25,5){\vector(-1,1){5}}

                      \put(5,5){\vector(1,1){5}}
                      \put(20,10){\vector(-1,0){10}}
                    \put(5,5){\circle*{1}}
                    \put(10,0){\vector(-1,1){5}}

                 \put(10,10){\vector(0,1){10}}
                  \put(10,20){\vector(1,0){10}}
                    \put(20,20){\vector(0,-1){10}}
                    \put(10,20){\circle*{1}}
                   \put(20,20){\circle*{1}}

           \put(10,-10){\vector(0,1){10}}
                  \put(20,-10){\vector(-1,0){10}}
                    \put(20,0){\vector(0,-1){10}}
                    \put(10,-10){\circle*{1}}
                   \put(20,-10){\circle*{1}}

 \put(10,10){\vector(0,-1){10}}

                        \end{picture}}
       \put(70,10){\begin{picture}(20,20)
                 \put(10,0){\circle*{1}}
                     \put(20,0){\circle*{1}}
                     \put(20,10){\circle*{1}}
                     \put(10,10){\circle*{1}}
                     \put(10,0){\vector(1,0){10}}
                      \put(25,5){\vector(-1,-1){5}}
                      \put(25,5){\circle*{1}}
                       \put(20,10){\vector(1,-1){5}}

                      \put(5,5){\vector(1,1){5}}
                      \put(20,10){\vector(-1,0){10}}
                    \put(5,5){\circle*{1}}
                    \put(10,0){\vector(-1,1){5}}

                 \put(10,10){\vector(0,1){10}}
                  \put(10,20){\vector(1,0){10}}
                    \put(20,20){\vector(0,-1){10}}
                    \put(10,20){\circle*{1}}
                   \put(20,20){\circle*{1}}

           \put(10,-10){\vector(0,1){10}}
                  \put(20,-10){\vector(-1,0){10}}
                    \put(20,0){\vector(0,-1){10}}
                    \put(10,-10){\circle*{1}}
                   \put(20,-10){\circle*{1}}

 \put(10,10){\vector(0,-1){10}}
 \put(20,0){\vector(0,1){10}}

                        \end{picture}}
        \put(100,10){\begin{picture}(20,20)
                 \put(10,0){\circle*{1}}
                     \put(20,0){\circle*{1}}
                     \put(20,10){\circle*{1}}
                     \put(10,10){\circle*{1}}
                     \put(10,0){\vector(1,0){10}}
                     \put(20,10){\vector(-1,0){10}}
                   \put(10,10){\vector(0,1){10}}
                  \put(10,20){\vector(1,0){10}}
                    \put(20,20){\vector(0,-1){10}}
                    \put(10,20){\circle*{1}}
                   \put(20,20){\circle*{1}}

           \put(10,-10){\vector(0,1){10}}
                  \put(20,-10){\vector(-1,0){10}}
                    \put(20,0){\vector(0,-1){10}}
                    \put(10,-10){\circle*{1}}
                   \put(20,-10){\circle*{1}}

 \put(10,10){\vector(0,-1){10}}
 \put(20,0){\vector(0,1){10}}

 \put(30,10){\circle*{1}}
 \put(30,0){\circle*{1}}
 \put(20,10){\vector(1,0){10}}
 \put(30,10){\vector(0,-1){10}}
 \put(30,0){\vector(-1,0){10}}

                        \end{picture}}

       \put(140,10){\begin{picture}(20,20)

                     \put(0,5){\circle*{1}}
                     \put(5,0){\circle*{1}}
                     \put(5,10){\circle*{1}}
                     \put(10,5){\circle*{1}}
                      \put(15,0){\circle*{1}}
                          \put(15,10){\circle*{1}}
                           \put(20,5){\circle*{1}}
             \put(5,0){\vector(1,0){10}}
             \put(15,0){\vector(-1,1){5}}
             \put(10,5){\vector(-1,-1){5}}
             \put(5,0){\vector(-1,1){5}}
            \put(0,5){\circle*{1}}
             \put(5,10){\vector(1,-1){5}}
               \put(10,5){\vector(1,1){5}}

             \put(20,5){\vector(-1,-1){5}}

         \put(0,10){\circle*{1}}
             \put(0,10){\vector(1,0){5}}
             \put(0,5){\vector(0,1){5}}

             \put(15,10){\vector(1,-1){5}}

                 \put(5,-10){\circle*{1}}
                 \put(15,-10){\circle*{1}}
                     \put(15,0){\vector(0,-1){10}}
                     \put(15,-10){\vector(-1,0){10}}
                     \put(5,-10){\vector(0,1){10}}
                        \end{picture}}

       \end{picture}}\\
\hline

\end{tabular}
\]

\end{document}